\documentclass[a4paper, reqno]{amsart}

\usepackage{amsmath, amsthm, amssymb}
\usepackage{verbatim}
\usepackage{amsfonts}
\usepackage[all]{xypic}      
\usepackage{yhmath}
\usepackage{bbm}
\usepackage[bookmarks=true, linkbordercolor={0.4 0.35 0.8}]{hyperref}
\usepackage{color}

\newtheorem{theorem}{Theorem}[section]
\newtheorem{corollary}[theorem]{Corollary}
\newtheorem{lemma}[theorem]{Lemma}
\newtheorem{proposition}[theorem]{Proposition}

\theoremstyle{definition}
\newtheorem{definition}[theorem]{Definition}
\newtheorem{example}[theorem]{Example}

\theoremstyle{remark}
\newtheorem{remark}[theorem]{Remark}

\numberwithin{equation}{section}

\newcommand{\lie}{\mathfrak}
\newcommand{\ds}{\displaystyle}

\DeclareMathOperator{\Hom}{Hom}

\DeclareMathOperator{\Id}{Id}

\DeclareMathOperator{\ad}{ad}

\DeclareMathOperator{\supp}{supp}
\DeclareMathOperator{\ph}{ph}

\DeclareMathOperator{\Irr}{Irr}

\newcommand{\cop}{\mathrm{cop}}
\newcommand{\ind}{\mathrm{ind}}
\newcommand{\red}{\mathrm{r}}

\newcommand{\st}{\; | \;}

\newcommand{\defeq}{=}

\newcommand{\CC}{\mathbb{C}}

\newcommand{\NN}{\mathbb{N}}
\newcommand{\RR}{\mathbb{R}}

\newcommand{\ZZ}{\mathbb{Z}}

\newcommand{\KK}{\mathbb{K}}
\newcommand{\LL}{\mathbb{L}}
\newcommand{\JJ}{\mathbb{J}}
\renewcommand{\AA}{\mathbb{A}}

\newcommand{\sH}{\mathcal{H}}
\newcommand{\sE}{\mathcal{E}}
\newcommand{\sX}{\mathcal{X}}

\newcommand{\bD}{\mathbf{D}}
\newcommand{\bP}{\mathbf{P}}
\newcommand{\bQ}{\mathbf{Q}}

\newcommand{\half}{\frac{1}{2}}

\newcommand{\ip}[1]{\langle #1 \rangle}

\newcommand{\ket}[1]{| #1 \rangle}
\newcommand{\bigket}[1]{\left| #1 \right\rangle}
\newcommand{\bra}[1]{\langle #1 |}
\newcommand{\bigbra}[1]{\left\langle #1 \right|}
\newcommand{\braket}[2]{\langle #1 | #2 \rangle}
\newcommand{\bigbraket}[2]{\left.\left\langle #1 \right| #2 \right\rangle}

\newcommand{\into}{\hookrightarrow}
\newcommand{\onto}{\twoheadrightarrow}
\newcommand{\ideal}{\triangleleft}
\newcommand{\hit}{\rightharpoonup}

\newcommand{\slot}{\:\cdot\:}

\newcommand{\dual}{*}

\newcommand{\Univ}{\mathcal{U}}
\newcommand{\Uq}{\breve{\mathcal{U}}_q}
\newcommand{\Poly}{\mathcal{O}}

\newcommand{\SL}{\mathrm{SL}}
\newcommand{\SU}{\mathrm{SU}}
\newcommand{\U}{\mathrm{U}}

\newcommand{\SUq}{\SU_q}

\newcommand{\glx}{\lie{gl}}
\newcommand{\slx}{\lie{sl}}

\newcommand{\bigqqnumber}[2][q]{\left[\!\left[#2\right]\!\right]_{#1}}
\newcommand{\qchoose}[1]{\begin{bmatrix}#1\end{bmatrix}}
\newcommand{\qhypergeometric}[2]{{}_{#1}\overline{\phi}_{#2}}

\newcommand{\dq}{d_q}
\newcommand{\Dq}{D_q}
\newcommand{\legendre}{\mathrm{p}}

\newcommand{\GTs}[7]{
  \left( \begin{array}{cccccccccc}
  \multicolumn{2}{c}{#1} &
  \multicolumn{2}{c}{0} &
  \multicolumn{2}{c}{\cdots} &
  \multicolumn{2}{c}{0} &
  \multicolumn{2}{c}{#2} \\
&\multicolumn{2}{c}{#3}
&\multicolumn{4}{c}{0\cdots0}
&\multicolumn{2}{c}{#4} \\
&&\multicolumn{2}{c}{\ddots}
&&&\multicolumn{2}{c}{\adots}\\
&&&\multicolumn{2}{c}{#5}
&\multicolumn{2}{c}{#6} \\
&&&&\multicolumn{2}{c}{#7}
\end{array} \right)
}

\newcommand{\GTsO}[1]{
  \left( \begin{array}{cccccccccc}
  \multicolumn{2}{c}{#1} &
  \multicolumn{2}{c}{0} &
  \multicolumn{2}{c}{\cdots} &
  \multicolumn{2}{c}{0} &
  \multicolumn{2}{c}{\!\!-#1} \\
&\multicolumn{2}{c}{\;0\;}
&\multicolumn{4}{c}{\;0\,\cdots\,0\;}
&\multicolumn{2}{c}{\;0\;} \\
&&\multicolumn{2}{c}{\ddots}
&&&\multicolumn{2}{c}{\adots}\\
&&&\multicolumn{2}{c}{\;0}
&\multicolumn{2}{c}{0} \\
&&&&\multicolumn{2}{c}{0}
\end{array} \right)
}

\newcommand{\smallGTs}[2]{
  \left( \begin{array}{cccccc}
  \multicolumn{2}{c}{#1} &
  \multicolumn{2}{c}{0} &
  \multicolumn{2}{c}{\!\!-#1} \\
&\multicolumn{2}{c}{\;#2\;}
&\multicolumn{2}{c}{\!\!-#2\;}  \\
&&\multicolumn{2}{c}{0}
\end{array} \right)}

\newcommand{\smallGTsO}[1]{
  \left( \begin{array}{cccccc}
  \multicolumn{2}{c}{#1} &
  \multicolumn{2}{c}{0} &
  \multicolumn{2}{c}{\!\!-#1} \\
&\multicolumn{2}{c}{\;0}
&\multicolumn{2}{c}{0\;}  \\
&&\multicolumn{2}{c}{0}
\end{array} \right)}

\newcommand{\sCI}{\mathcal{CI}}
\newcommand{\sCC}{\mathcal{CC}}

\newcommand{\upp}{\uparrow}
\newcommand{\low}{\downarrow}
\newcommand{\bigupp}{{\displaystyle \uparrow}}
\newcommand{\biglow}{{\displaystyle \downarrow}}
\newcommand{\BGG}{\mathrm{BGG}}

\newcommand{\Roots}{\mathbf{\Delta}}

\newcommand{\counit}{\varepsilon}

\newcommand{\triv}{\mathbbm{1}}		

\newcommand{\rootset}{{I}}

\newcommand{\finitesubset}{\subset\!\subset}

\newcommand{\xket}[1]{\ket{\mathrm{x}_{#1}}}
\newcommand{\yket}[1]{\ket{\mathrm{y}_{#1}}}

\newcommand{\yxbraket}[2]{\braket{\mathrm{y}_{#1}}{\mathrm{x}_{#2}}}

\newcommand{\fixedvec}{\ket{(M_{(m,0,\ldots,0)})^\low}}

\newcommand{\Mult}[1]{\mathsf{M}_l(#1)}
\newcommand{\RMult}[1]{\mathsf{M}_r(#1)}

\newcommand{\coeff}[2]{\langle #1 | \slot | #2 \rangle}		

\newcommand{\Wts}{\mathrm{Wts}}

\newenvironment{bnum}
{\begin{list}{}
    {\setlength{\labelwidth}{15pt}
     \setlength{\leftmargin}{\labelwidth}
    }
}
{\end{list}}

\begin{document}

\title[Equivariant Fredholm modules]{Equivariant Fredholm modules for the full quantum flag manifold of $\SU_q(3)$}

\author{Christian Voigt}
\address{School of Mathematics and Statistics, University of Glasgow, 15 University Gardens, Glasgow G12 8QW, United Kingdom}
\email{christian.voigt@glasgow.ac.uk}

\author{Robert Yuncken}
\address{Clermont Universit\'e, Universit\'e Blaise Pascal, Laboratoire de Math\'ematiques, BP 10448, F-63000 Clermont-Ferrand, France}
\email{yuncken@math.univ-bpclermont.fr}

\thanks{This work was supported by the Engineering and Physical Sciences Research Council Grant EP/L013916/1.}


\maketitle

\begin{abstract}
We introduce $C^*$-algebras associated to the foliation structure of a quantum flag manifold. We use these to construct $\SL_q(3,\CC)$-equivariant 
Fredholm modules for the full quantum flag manifold $\sX_q = \SU_q(3)/T$ of $\SUq(3)$, based on an analytical version of the Bernstein-Gelfand-Gelfand 
complex. As a consequence we deduce that the flag manifold $ \sX_q $ satisfies Poincar\'e duality in equivariant $ KK $-theory. 
Moreover, we show that the Baum-Connes conjecture with trivial coefficients holds for the discrete quantum group dual to $\SU_q(3)$. 
\end{abstract}

\section{Introduction}

In noncommutative differential geometry \cite{Connesbook}, the notion of a smooth manifold is extended beyond its classical scope by adopting a 
spectral point of view. This is centred around the idea of constructing Dirac-type operators associated with possibly noncommutative algebras, capturing 
the underlying Riemannian structure of geometric objects for which ordinary differential geometry breaks down. The key concept in this theory, introduced 
by Connes, is the notion of a spectral triple \cite{Connesgravmat}. 

Quantum groups provide provide a large class of examples of noncommutative spaces, and they have been studied extensively within the framework 
of noncommutative differential geometry. Among the many contributions in this direction let us only mention a few. 
Chakraborty and Pal \cite{CPspectraltriple} defined an equivariant spectral triple on $ \SU_q(2) $, which was studied in detail by 
Connes \cite{Connessuq2}. Later Dabrowski, Landi, Sitarz, van Suijlekom 
and V\'arilly \cite{DLSSVdirac}, \cite{DLSSVindex} defined and studied a deformation of the classical Dirac operator on $ \SU(2) $, thus obtaining 
a different spectral triple on $ \SU_q(2) $. The techniques used in these papers rely on explicit estimates involving Clebsch-Gordan coefficients. 
In a different direction, 
Neshveyev and Tuset exhibited a general mechanism for transporting the Dirac operator on an arbitrary compact simple Lie group to its quantum 
deformation, based on Drinfeld twists and properties of the Drinfeld associator \cite{NT:Deformation}. 
The resulting spectral triples inherit various desirable properties from their classical counterparts, although unfortunately they are difficult to 
study directly since this requires a certain amount of control of the twisting procedure. 

This article is concerned with the quantized full flag manifolds associated to the $q$-deformations of compact semisimple Lie groups, and in particular 
the flag manifold of $\SU_q(3)$, the simplest example in rank greater than one.
In the rank-one case, that is for $ \SU_q(2) $, the flag manifold $ \SU_q(2)/T $ is known as the standard Podle\'s sphere, and Dirac 
operators on it have been defined and studied by several authors \cite{Owczarek, DS:Podles, SW}. A version of the local index formula 
for the Podle\'s sphere is exhibited in \cite{NT:Index, Wagner, RS:Index}, although slight modifications must be made to Connes' original formalism. 

The higher rank situation has proven to be considerably more difficult. 
Kr\"ahmer \cite{Krahmer} gave an algebraic construction of Dirac operators on quantized {\em irreducible} flag manifolds in 
higher rank. These retain a certain rank-one character in their geometry. In particular, the construction in \cite{Krahmer} does not 
cover the case of full flag manifolds. 
On the other hand, the Dirac operator defined by Neshveyev and Tuset can be used to write down spectral triples for arbitrary full quantum flag manifolds. 
However, the most direct way to do so, which was indicated already in \cite{NT:Deformation}, does not suffice to describe the equivariant $ K$-homology group of 
the quantum flag manifold using Poincar\'e duality. More precisely, one only obtains certain multiples of the class of the Dirac operator in this way. 

In this paper, we describe a construction of a Dirac-type class in equivariant $K$-homology for the full flag manifold $\sX_q = \SU_q(3)/T $ 
of $\SU_q(3)$ as a {\em bounded} Fredholm module. This does not give the full ``noncommutative Riemannian'' structure on $ \sX_q $ that a Connes-type spectral 
triple would give. In fact, a key philosophical point behind our construction is that the natural geometric structure on quantized 
flag manifolds in higher rank is not Riemannian but parabolic, in the sense of \cite{CS:Parabolic}.  

Correspondingly, the construction of our Dirac-type class is based not upon the Dirac or Dolbeault operator but upon the Bernstein-Gelfand-Gelfand (BGG) complex, 
see \cite{BGG, BasEas, CSS}. The quantum version of the BGG complex for $\SL_q(n,\CC)$, in its algebraic form, first appeared in \cite{Rosso:BGG}; 
see also \cite{HK2,HK3}. It has not been much studied from an analytical point of view so far.  
In fact, developing a complete unbounded noncommutative version of parabolic geometries seems to be difficult. 
For instance, the BGG complex is neither elliptic nor order $1$, although it does exhibit a kind of subellipticity. In the present work, we convert the BGG complex into a bounded $K$-homology cycle. Such a construction was achieved for a classical flag manifold in \cite{Yuncken:PsiDOs, Yuncken:BGG}. 
A major goal of the present work is to demonstrate that the necessary analysis can also be carried out for a quantized flag manifold. 

In particular, our $ K $-homology class is equivariant not only with respect to $ \SU_q(3) $, but with respect to the complex quantum 
group $ \SL_q(3,\CC) = \bD(\SU_q(3)) $, the Drinfeld double of $ \SU_q(3) $. Drinfeld doubles play an important role in the definition 
of equivariant Poincar\'e duality \cite{NVpoincare} and the proof of the Baum-Connes conjecture for the dual of $ \SU_q(2) $, see \cite{Voigtbcfo}. 
It is worth pointing out that the verification of $ \SL_q(3, \CC) $-equivariance of our cycle is somewhat simpler than in the classical situation. 
We also remark that in the construction of our $ K $-homology class we use some properties of principal series representations of $ \SL_q(3, \CC) $ 
which will be discussed in a separate paper \cite{VYprincipal}. 

Our main result can be formulated as follows. 

\begin{theorem} \label{thm:BGG_element}
The BGG complex for the full flag manifold $\sX_q = \SU_q(3)/T$ of $\SU_q(3)$ can be normalized to give a bounded equivariant $K$-homology cycle 
in the Kasparov group $KK^{\SL_q(3, \CC)}(C(\sX_q), \CC)$.  
The equivariant index of this element with respect to $ \SU_q(3) $ is the class of the trivial representation 
in $ KK^{\SU_q(3)}(\CC, \CC) = R(\SU_q(3))$.
\end{theorem}

We refer to Theorem \ref{thm:K-homology_class} for the precise statement of this result. The main idea behind 
our construction can be sketched as follows. Firstly, corresponding to each of the two simple roots of $\SU_q(3)$ there is a fibration of the quantized 
flag manifold whose fibres are Podle\'s spheres. These fibrations carry families of Dirac-type operators analogous to the operators constructed by Dabrowski-Sitarz. As is common in Kasparov's $KK$-theory, we replace these longitudinal operators by their bounded transforms. We then use 
a variant of the Kasparov product, inspired by the BGG complex, to assemble them into a single $\SL_q(3,\CC)$-equivariant $K$-homology cycle for $\sX_q$.

At present, there is only one ingredient which prevents us from carrying out our construction for the full flag manifold of $ \SU_q(n) $ for any $ n \geq 2 $, 
namely the operator partition of unity in Lemma \ref{lem:operator_po1}. We develop all the harmonic analysis in the generality of $ \SU_q(n) $.

\medskip

Using Theorem \ref{thm:BGG_element} we derive two consequences regarding equivariant $ KK $-theory.  
Firstly, we conclude that the quantum flag manifold $ \sX_q $ satisfies equivariant Poincar\'e duality in $ KK $-theory in the sense of \cite{NVpoincare}. 

\begin{corollary} \label{cor:PD}
The flag manifold $ \sX_q $ is $ \SU_q(3) $-equivariantly Poincar\'e dual to itself. That is, there is a natural isomorphism
$$
KK^{\bD(\SU_q(3))}_*(C(\sX_q) \boxtimes A, B) \cong KK^{\bD(\SU_q(3))}_*(A, C(\sX_q) \boxtimes B)
$$
for all $ \bD(\SU_q(3)) $-$ C^* $-algebras $ A $ and $ B $, where $ \boxtimes $ denotes the braided tensor product with respect to $ \SU_q(3) $. 
\end{corollary}

For the definition and properties of braided tensor products we refer to \cite{NVpoincare}. 
We note that it is crucial here that the class obtained in Theorem \ref{thm:BGG_element} is equivariant with respect 
to $ \bD(\SU_q(3)) = \SL_q(3, \CC) $ and not just $ \SU_q(3) $. 

Secondly, we discuss an analogue of the Baum-Connes conjecture for the discrete quantum group 
dual to $ \SU_q(3) $. In \cite{MNtriangulated}, Meyer and Nest have developed an approach to the Baum-Connes conjecture \cite{BCH} which allows one to construct assembly maps in rather general circumstances, and which applies in particular to duals of $ q $-deformations. As already mentioned above, the
simplest case of $ \SU_q(2) $ was studied in \cite{Voigtbcfo}, and here we show how to go one step further as follows. 

\begin{corollary}
The Baum-Connes conjecture with trivial coefficients $ \mathbb{C} $ holds for the discrete quantum group dual to $ \SU_q(3) $.
\end{corollary}

This result is significantly weaker than the analogous statement for the dual of $ \SU_q(2) $ in \cite{Voigtbcfo}. 
However, let us point out that one cannot hope to carry over the arguments used in \cite{Voigtbcfo} to the higher rank situation. 
Indeed, according to work of Arano \cite{Aranospherical}, the Drinfeld double of $ \SU_q(3) $ has property $ (T) $. 
This forbids the existence of continuous homotopies along the complementary series representations to the trivial representation 
in the unitary dual. Such homotopies are at the heart of the arguments in \cite{Voigtbcfo}. 
In other words, the problem is similar to well-known obstacles to proving the Baum-Connes conjecture with coefficients for the classical 
groups $\SL(n,\CC)$ in higher rank.

\medskip

Let us now explain how the paper is organized. In Section \ref{sec:prelim} we collect some preliminaries on quantum groups and fix our notation. Sections \ref{sec:operator_ideals} and \ref{sec:lattice_of_ideals} contain the definition and basic properties of certain ideals of $ C^* $-algebras associated to the canonical fibrations of a quantum flag manifold. These $C^*$-ideals are defined in terms of the harmonic analysis of the block diagonal quantum subgroups 
of $\SU_q(n)$, and are the basis of all the analysis that follows.  

In Section \ref{sec:statements} we formulate the main technical results about these ideals. These results are parallel to classical facts from the calculus of longitudinally elliptic pseudodifferential operators. The proofs are deferred to subsequent sections, which may be skipped on a first reading. 
Specifically, Section \ref{sec:Gelfand-Tsetlin} collects some facts about 
Gelfand-Tsetlin bases, and in particular the effect of reversing the order of roots used in their definition. 
Section \ref{sec:essorth} introduces the notion of essentially orthotypical quantum subgroups, in analogy with the considerations in \cite{Yuncken:PsiDOs}. 
In Section \ref{sec:PsiDOs2} the analytic properties of longitudinal pseudodifferential-type operators are established. 

Section \ref{sec:SLq} contains some definitions and facts related to complex quantum groups and their representations, and 
it is checked that our constructions are compatible with the natural action of $ \SL_q(n,\CC) $. 
In section \ref{sec:BGG} we describe the analytical quantum BGG complex for the flag manifold of $ \SU_q(3) $, 
and we prove our main theorem. 

The final section \ref{sec:bc} contains the corollaries stated above. That is, we show that $ \sX_q $ is equivariantly Poincar\'e dual 
to itself, and we verify the Baum-Connes conjecture with trivial coefficients for the dual of $ \SU_q(3) $. 

\medskip

Let us conclude with some remarks on notation. The dual of a vector space $V$ is denoted $V^\dual$. We write $ \LL(H, H') $ for the space of bounded
operators between Hilbert spaces $ H $ and $ H' $, and  $ \KK(H,H') $ denotes the space of compact operators.  When $H=H'$ we abbreviate these 
as $ \LL(H) $ and $ \KK(H) $. Depending on the context, the symbol $ \otimes $ denotes either an algebraic tensor product, the tensor product of Hilbert 
spaces or the minimal tensor product of $ C^* $-algebras. All Hilbert spaces in this paper are separable. 

\medskip 

It is a pleasure to thank Uli Kr\"ahmer for inspiring discussions on quantized flag manifolds.


\section{Preliminaries} \label{sec:prelim}

In this section we discuss some preliminaries on quantum groups in general and $ q $-deformations in particular. For more details 
and background we refer to the literature \cite{CPbook}, \cite{KS}, \cite{Majid}.

\subsection{Some notation}

Let $K=\SU(n)$ or $\U(n)$ with $n\geq2$. We write $T$ for the standard maximal torus of $K$, that is, the diagonal subgroup, and $\lie{t}$ for its Lie algebra.  We write $\slx_n$ for $\lie{sl}(n,\CC)$ and $\glx_n$ for $\lie{gl}(n,\CC)$. In either case we denote by $\lie{h} = \lie{t}_\CC$ the Cartan subalgebra. 
We write $\bP$ for the set of weights of $K$, viewed as a lattice in $\lie{h}^\dual$. If $V$ is a $K$-representation and $\mu\in\bP$, the subspace
of vectors of weight $\mu$ in $V$ will be denoted $V_\mu$.  

It will be convenient to identify the weight lattice of $\U(n)$ with $\ZZ^n$, where an element $\mu = (\mu_1,\ldots,\mu_n)\in\ZZ^n$ corresponds to the 
weight $\mu \in \lie{h}^\dual$ given by
$$
\mu(\mathrm{diag}(t_1,\ldots,t_n)) = \mu_1t_1 + \cdots + \mu_nt_n.
$$                                                                         
The corresponding character of $T$ will be denoted by $e^\mu\in C(T)$. We equip $\lie{h}^\dual$ with the  bilinear form which extends the standard pairing on $\bP \cong \ZZ^n$:
$$
( \,(\mu_1,\ldots, \mu_n), \, (\mu_1',\ldots,\mu_n') \, ) = \sum_i \mu_i \mu_i'.
$$

For $\SU(n)$, the weight lattice identifies with the quotient $\ZZ^n / \ZZ(1,\ldots,1)$, and the bilinear form on $\lie{h}^\dual$ is obtained from that above by identifying $\lie{h}^\dual$ with the orthogonal complement of $\CC(1,\ldots,1)$ in $\CC^n$.

We write $\Delta$ for the set of roots of $\SU(n)$ or $\U(n)$; they are the same in both cases.  We fix the set of simple roots $\Sigma = \{\alpha_1,\ldots,\alpha_{n-1}\}$ where $\alpha_i:  \mathrm{diag}(t_1,\ldots,t_n) \mapsto t_i - t_{i+1}$.

\subsection{Quantized universal enveloping algebras}
\label{sec:env_alg}

We shall use the quantized universal enveloping algebras which are denoted $\Uq(\glx_n)$ and $\Uq(\slx_n)$ in \cite{KS} (pages 212 and 164, respectively), 
since these are the versions used in the literature on Gelfand-Tsetlin theory. We briefly recall their definitions.

Fix $q\in(0,1)$. For any $a\in\CC$ we write $[a]_q = \frac{q^a-q^{-a}}{q-q^{-1}}$, and for $a\in\NN$,
\begin{align*}
[a]_q! &=  \prod_{k=1}^a [k]_q, & \qchoose{a \\ m}_q &= \frac{[a]_q!}{[a-m]_q!\,[m]_q!}.
\end{align*}
Often, we shall drop the subscript $q$ in the notation.

The Hopf $*$-algebra $\Uq(\glx_n)$%
is generated by elements $E_i, F_i$ ($i=1,\ldots,n - 1$) and $G_j, G_j^{-1}$ ($j=1,\ldots,n$) with the relations

$$
\begin{array}{ll}
 G_j G_k = G_k G_j, & G_j^{-1}G_j =1 = G_jG_j^{-1}\\[1.5ex]
 G_j E_i G_j^{-1} = \begin{cases}
		      q^{\half} E_i, & j=i,\\
		      q^{-\half} E_i, & j=i+1, \\
		      E_i,& \text{otherwise},
                    \end{cases} 
 \qquad &
 G_j F_i G_j^{-1} = \begin{cases}
		      q^{-\half} F_i, & j=i,\\
		      q^{\half} F_i, & j=i+1, \\
		      F_i,& \text{otherwise},
                    \end{cases}  \\[6ex]
 {[E_i,F_j]} = \ds\delta_{ij} \frac{G_i^2 G_{i+1}^{-2}-G_i^{-2} G_{i+1}^2}{q - q^{-1}}, \\[2.5ex]
\multicolumn{2}{l}{E_i^2 E_{i\pm1} - [2]_q E_i E_{i\pm1} E_i + E_{i\pm1} E_i^2 = 0 = F_i^2 F_{i\pm1} - [2]_q F_i F_{i\pm1} F_i + F_{i\pm1} F_i^2} \\[1.5ex]
\multicolumn{2}{l}{[E_i, E_j] = 0 = [F_i, F_j], |i - j| \geq 2.}
\end{array}
$$
The formulas for the coproduct $\hat{\Delta}: \Uq(\glx_n) \rightarrow \Uq(\glx_n) \otimes \Uq(\glx_n)$ are 
\begin{align*}
  \hat{\Delta}(E_i) &= E_i \otimes G_i G_{i+1}^{-1} + G_i^{-1} G_{i+1} \otimes E_i, \\
  \hat{\Delta}(F_i) &= F_i \otimes G_i G_{i+1}^{-1} + G_i^{-1} G_{i+1} \otimes F_i, \\
  \hat{\Delta}(G_i) &= G_i \otimes G_i, 
\end{align*} 
the counit $ \hat{\epsilon}: \Uq(\glx_n) \rightarrow \CC $ is given by 
\begin{align*}
\hat{\counit}(E_i)& = 0, \qquad \hat{\counit}(F_i) = 0, \qquad \hat{\counit}(G_i)= 1, 
\end{align*}
and the antipode is determined by 
$$
\hat{S}(E_i) = -q E_i, \qquad \hat{S}(F_i) = -q^{-1}F_i, \qquad \hat{S}(G_i) = G_{i}^{-1}. 
$$
Finally, the $*$-structure is given by
$$
E_i^* = F_i, \qquad G_i^* = G_i. 
$$
Throughout, we will use the Sweedler notation $ \hat{\Delta}(X) = X_{(1)} \otimes X_{(2)} $ for the coproduct. 
We note that with this definition of $\Uq(\glx_n)$, weight spaces are defined by saying that $G_i$ acts on vectors of weight $\mu=(\mu_1,\ldots,\mu_n)$ by multiplication by $q^{\half \mu_i}$.

The Hopf $*$-algebra $\Uq(\slx_n)$ is the Hopf $*$-subalgebra of $\Uq(\glx_n)$ generated by the elements $E_i$, $F_i$, $K_i = G_iG_{i+1}^{-1}$ 
and $K_i^{-1}$, for $i=1,\ldots,n-1$. The element $K_i$ acts on vectors of weight $\mu\in\bP$ by multiplication by $q^{\half(\alpha_i,\mu)}$.

\subsection{Quantized algebras of functions}

Fix $K_q = \SU_q(n)$ for $n\geq 2$. The quantized algebra of functions $\Poly(K_q) $ is the space of matrix coefficients of 
finite-dimensional type $1$ representations of $\Uq(\slx_n)$; see \cite{KS} for more details. 
If $\sigma$ is a type $1$ representation of $\Uq(\slx_n)$ and $\xi\in V^\sigma$, $\xi^\dual\in V^{\sigma\dual}$, we denote the associated 
matrix coefficient by the bra-ket notation
$$
\langle\xi^\dual|\slot|\xi\rangle : X \mapsto (\xi^\dual,\sigma(X)\xi), \qquad \text{for $X\in\Uq(\lie{g})$.}
$$

We shall use the $*$-Hopf algebra structure on $\Poly(K_q)$ which makes the evaluation map $\Uq(\slx_n) \times \Poly(\SU_q) \to \CC$ into 
a {\em skew}-pairing of $*$-Hopf algebras, or equivalently, a Hopf pairing of $\Uq(\slx_n)^\cop $ and $\Poly(K_q)$. 
Thus, for all $X,Y\in\Uq(\slx_n)$, $f,g\in\Poly(K_q)$,
\begin{align*}
(XY,f) &= (X,f_{(1)})(Y,f_{(2)}),\\
(X,fg) &= (X_{(1)},g)(X_{(2)},f) \\
(\hat{S}(X),f) &= (X,S^{-1}(f)), 
\end{align*}
where we use Sweedler notation $\Delta(f) = f_{(1)} \otimes f_{(2)}$ for $ f \in \Poly(K_q) $. 
In terms of matrix coefficients the multiplication is given by
\begin{equation}
\label{eq:product_of_coeffs}
\coeff{\xi^\dual}{\xi} \, \coeff{\eta^\dual}{\eta} 
= \coeff{\eta^\dual \otimes \xi^\dual}{\eta \otimes \xi}.
\end{equation}
where $\xi \in V^{\sigma}$, $\xi^\dual \in V^{\sigma\dual}$, $\eta\in V^{\tau}$, $\eta^\dual\in V^{\tau\dual}$ for type $1$ 
representations $\sigma$, $\tau$.  

The comultiplication of $\Poly(K_q) $ defines left and right corepresentations of $\Poly(K_q)$ on itself. They will play very different roles in what 
follows: the left regular corepresentation will be used to define representations of $K_q$, while the right regular representation will be used to 
define $K_q$-invariant differential operators and carry out their harmonic analysis.

The right regular corepresentation of $\Poly(K_q)$ gives rise to a left action of $\Uq(\slx_n)$ according to the formula 
\begin{equation}
\label{eq:hit}
X \hit f  = f_{(1)} (X, f_{(2)}), \qquad \text{for } X\in\Uq(\glx_n), f\in\Poly(\U_q(n)).
\end{equation}
We shall usually write this simply as $X f$. 

The Hilbert space $L^2(K_q)$ is the completion of $\Poly(K_q)$ with respect to the inner product 
$$
\ip{f,g} = \phi(f^*g),
$$
where $\phi$ is the Haar state of $\Poly(K_q)$.

The left and right multiplication action of $f\in\Poly(K_q)$ on $L^2(K_q)$ will be denoted by $ \Mult{f}$ and $ \RMult{f} $, respectively. 
The left multiplication action defines a $ * $-homomorphism $ \Poly(K_q) \rightarrow \LL(L^2(K_q)) $. By definition, 
the $C^*$-completion $C(K_q)$ of $\Poly(K_q)$ is the norm closure of the image of $ \Poly(K_q) $ under this representation. 
In this way one obtains the compact quantum group structure of $K_q$. 

The algebra $\Poly(\U_q(n))$ is defined analogously, as matrix coefficients of type $1$ representations of $\Uq(\glx_n)$. All the above 
constructions carry over to $\U_q(n)$.

\subsection{Representations and duality} 
\label{sec:dual_group}

Let $K_q = \SU_q(n)$. By definition, a unitary representation of $ K_q $ on a Hilbert space $ H $ is a unitary 
element $U \in M(C(K_q) \otimes \KK(H))$ such that $(\Delta \otimes \Id)(U) = U_{13} U_{23} $. 
Here we are using leg numbering notation. We shall often designate unitary $K_q$-representations simply by the Hilbert spaces underlying them. 
If $ H, H' $ are unitary representations of $ K_q $ we write $ \Hom_{K_q}(H,H') $ 
for the space of intertwiners, that is, for the set of all bounded linear maps $ T: H \rightarrow H' $ satisfying $ (\Id \otimes T)U = U'(\Id \otimes T) $. 

A unitary representation $ H $ of $ K_q $ is irreducible if and only if $\Hom_{K_q}(H,H) = \mathbb{C}$. 
All irreducible unitary representations of $ K_q $ are finite dimensional, and we write $\Irr(K_q)$ for the set of their equivalence classes. In the 
context of harmonic analysis, elements of $\Irr(K_q)$ will be referred to as $K_q$-types. 
We shall usually blur the distinction between a specific irreducible representation and its class in $\Irr(K_q)$. 
Unless otherwise stated, the Hilbert space underlying a $K_q$-representation 
$\sigma \in \Irr(K_q) $ will be denoted $V^\sigma$.

We use $\triv_{K_q}$ to denote the trivial representation of $K_q$. For $\sigma \in \Irr(K_q)$, we denote by $ \sigma^c$ the (unitary) conjugate 
representation. If a $K_q$-representation $\pi$ contains $\sigma$ as an irreducible subrepresentation, we write $\sigma \leq \pi$.

We define $C_c(\hat{K}_q)$ as the algebraic direct sum 
$$
C_c(\hat{K}_q) = \bigoplus_{\sigma\in\Irr(K_q)} \!\!\! \LL(V^\sigma). 
$$ 
Its enveloping $C^*$-algebra is denoted $C_0(\hat{K}_q)$, this identifies with the $ C^* $-algebra of functions on the dual discrete quantum 
group $ \hat{K}_q $. 

We will also work with the algebraic direct product 
$$
C(\hat{K}_q) = \prod_{\sigma\in\Irr(K_q)}  \!\!\! \LL(V^\sigma), 
$$ 
which can be identified with the algebraic dual space $\Poly(K_q)^*$ of $\Poly(K_q)$. 
In particular, the quantized universal enveloping algebra $\Uq(\slx_n)$ is naturally a $ * $-subalgebra of $C(\hat{K}_q)$, 
and we will routinely use the same notation for elements of $\Uq(\slx_n)$ and their images in $C(\hat{K}_q)$. 

In our context, the main reason to consider the algebra $C(\hat{K}_q)$ is that it contains some elements outside $\Uq(\slx_n)$ which we shall 
need. In particular, the universal enveloping algebra $\Univ(\lie{h})$ of the Cartan subalgebra $\lie{h}$ of $\slx_n$ 
embeds into $C(K_q)$ if we identify $X\in\lie{h}$ with the operator which acts as $\mu(X)$ on the weight 
space $(V^\sigma)_\mu$ for each $\sigma\in\Irr(K_q)$, $\mu \in \bP$.


\subsection{Quantum subgroups}
\label{sec:lattice_of_subgroups}

Let $K_q = \SU_q(n)$.  
Given a set $\rootset\subseteq\Sigma$ of simple roots, we let $\lie{h}^{\rootset\perp}$ denote the subspace of $\lie{h}$ annihilated by the $\alpha_i \in \rootset$, and let $\lie{h}^{\rootset}$ be its orthocomplement with respect to the invariant bilinear form.  We let $\lie{g}^{\rootset}$ denote the following block-diagonal Lie subalgebra of $\slx_n$:
$$
  \lie{g}^{\rootset} = \lie{h} \oplus \!\! \bigoplus_{\alpha\in\Roots\cap\ZZ\rootset} \!\!\! \lie{g}_\alpha.
$$
This subalgebra admits the decomposition $\lie{g}^\rootset = \lie{s}^\rootset \oplus \lie{h}^{\rootset\perp}$ where $\lie{s}^\rootset= \lie{h}^\rootset \oplus  \bigoplus_{\alpha\in\Roots\cap\ZZ\rootset} \lie{g}_\alpha$ is semisimple and $\lie{h}^{\rootset\perp}$ is central.  The subalgebra $\lie{k}^{\rootset} = \lie{g}^{\rootset} \cap \lie{su}_n$ is the Lie algebra of a block-diagonal subgroup $K^{\rootset} \subseteq \SU(n)$.  

The analogous families of closed quantum subgroups of $\SU_q(n)$ are defined as follows. Here, we use the notation ${\langle x_j \rangle}$ to denote the $\sigma(C(\hat{K}_q), \Poly(K_q))$-closed subalgebra of $C(\hat{K}_q)$ generated by a collection of elements $x_j \in C(\hat{K}_q)$.  For each $\rootset\subseteq\Sigma$, we define
$$
\begingroup
\renewcommand{\arraystretch}{1.5}
\begin{array}{cc}
 C(\hat{K}^\rootset_q) = {\langle  X\in\Univ(\lie{h}),~ E_i, F_i ~(i\in\rootset) \rangle} ,\quad&\quad
 C(\hat{S}^\rootset_q) = {\langle  X\in\Univ(\lie{h}^\rootset),~ E_i, F_i ~(i\in\rootset)\rangle} ,\\ 
 C(\hat{T}^\rootset) = {\langle X\in \Univ(\lie{h}^\rootset) \rangle}, &
 C(\hat{T}^{\rootset\perp}) = {\langle X\in \Univ(\lie{h}^{\rootset\perp}) \rangle} .
\end{array}
\endgroup
$$
We then define  $ \Poly(K^\rootset_q) $, $ \Poly(S^\rootset_q) $, $\Poly(T^\rootset)$ and $\Poly(T^{\rootset\perp})$ to be the images of $ \Poly(K_q) $  under the induced surjection of $C(\hat{K}_q)^\dual$ onto $C(\hat{K}^\rootset_q)^\dual$, $C(\hat{S}^\rootset_q)^\dual$, $C(\hat{T}^\rootset)^\dual$ and $C(\hat{T}^{\rootset\perp})^\dual$, respectively.   They are Hopf *-algebras under the induced operations.  

In particular, $\Poly(K^\emptyset_q)$ is isomorphic to $\Poly(T)$. We write $\pi_T$ for the projection homomorphism $\Poly(K_q) \onto \Poly(T)$, and 
for its extension to the $C^*$-algebras. At the other extreme, we have $\Poly(K^\Sigma_q) = \Poly(K_q)$.

The quantum subgroups corresponding to the singleton subsets $I = \{\alpha_i\}$ with $i=1,\ldots,n-1$ will play a particularly important role. In this case, we will write $K^i_q, S^i_q, T^i, T^{i\perp}$ for the above quantum groups. Note that $S^i_q \cong \SU_q(2)$. We will also write $\Uq(\lie{s}^i_q)$ for the Hopf 
subalgebra of $\Uq(\slx_n)$ generated by $E_i$, $F_i$, $K_i$ and $K_i^{-1}$.


\subsection{The quantized flag manifold}
\label{sec:flag_manifolds}

Here, we summarize the basic definitions and properties of quantum flag manifolds. For more details see \cite{CPbook}, \cite{DStok}, \cite{HK1}, \cite{St}.

The full flag manifold of $K_q = \SU_q(n)$ is the quantum space $\sX_q = K_q/T$, defined via its algebra of functions as follows. 
The algebra $\Poly(K_q)$ is a right $\Poly(T)$-comodule algebra by restriction of the canonical right coaction of $\Poly(K_q)$ along the projection 
homomorphism $\pi_T: \Poly(K_q) \rightarrow \Poly(T)$. By definition, the algebra $\Poly(\sX_q)$ is the $*$-subalgebra of $\Poly(T)$-coinvariant elements, 
that is, 
\begin{align*}
  \Poly(\sX_q) &= \{ f\in \Poly(K_q) \st (\Id \otimes \pi_T) \Delta(f) = f \otimes 1 \} \\
    & = \{ f\in \Poly(K_q) \st K_i  f = f \text{ for all }  i=1,\ldots,n-1 \}.
\end{align*}

More generally, for any $\mu=(m_1,\ldots,m_n)\in\bP$ we define the section space of the induced line bundle $\sE_\mu$ over $\sX_q$ by
\begin{align*}
\Poly(\sE_\mu) &= \{ f\in\Poly(K_q) \st (\Id \otimes \pi_T) \Delta(f) = f \otimes e^\mu \} \\
&=\{ f\in \Poly(K_q) \st K_i  f = q^{\half (m_i - m_{i+1})} f \text{ for all } i=1,\ldots,n-1 \}.
\end{align*}
In other words, $\Poly(\sE_\mu)$ is the $\mu$-weight space of the right regular action of $ T $. 
Similarly, $L^2(\sE_\mu)$ and $C(\sE_\mu)$ are the right $\mu$-weight spaces of  $L^2(\U_q(n))$ and $C(\U_q(n))$, respectively. They are the closures of $ \Poly(\sE_\mu) $ in $ L^2(\U_q(n))$ and $ C(\U_q(n)/T) $, respectively. We will abbreviate direct sums of the form 
$ \Poly(\sE_\mu)\oplus \Poly(\sE_\nu)$ as $ \Poly(\sE_\mu\oplus\sE_\nu)$, and use analogous notation for their completions.

Multiplication in $\Poly(K_q)$ restricts to a map $\Poly(\sE_\mu) \otimes \Poly(\sE_\nu) \to \Poly(\sE_{\mu+\nu})$ for any $\mu,\nu\in\bP$. In particular, each $\Poly(\sE_\mu)$ is a bimodule over $\Poly(\sE_0) = \Poly(\sX_q)$. 
These modules are projective as either left or right $\Poly(\sX_q)$-modules since $ \Poly(K_q/T) \subset \Poly(K_q) $ is a faithfully flat 
Hopf-Galois extension \cite{MS}. 

Later on we will need an analogue of trivializing partitions of unity for the line bundles $\sE_\mu$.
These are described in the following lemma, which is an immediate consequence of Hopf-Galois theory, see \cite{SHopfgalois}. 

\begin{lemma}
\label{lem:bundle_po1}
For any $\mu\in\bP$, there exists a finite collection of sections $f_1,\ldots,f_k \in \Poly(\sE_\mu)$ and $g_1, \dots, g_k \in \Poly(\sE_{-\mu}) $ 
such that $\sum_{j=1}^k f_j g_j = 1  \in \Poly(\sX_q)$.
\end{lemma}

We will be interested in operators arising from the action of $ \Uq(\lie{sl}_n) $ on the above line bundles. 
Let $X\in\Uq(\lie{g})$ be of weight $\nu = (k_1,\ldots, k_n)$ for the left adjoint action, {\em i.e.}, $K_i X K_i^{-1} = q^{\half (k_i-k_{i+1})}X$ for all $i$. Then the right regular action $X:\Poly(K_q) \to \Poly(K_q)$ given by $ X f = X \hit f $ restricts to a map $X:\Poly(\sE_\mu) \to \Poly(\sE_{\mu+\nu})$ 
for every $\mu\in\bP$. 
In this way, $X$ defines an unbounded operator from $L^2(\sE_\mu)$ to $L^2(\sE_{\mu+\nu})$ with dense domain $\Poly(\sE_\mu)$, such that $X^*X$ is essentially self-adjoint. It should be thought of as a $K_q$-invariant differential operator.


\section{Isotypical decompositions and associated $C^*$-categories}
\label{sec:operator_ideals}

In this section and the next, we introduce the fundamental analytical structures which will be used throughout the remainder of the paper.  The idea is to describe the behaviour of certain linear operators with respect to a decomposition into isotypical subspaces.

\subsection{Isotypical decompositions}

Let $K_q$ be a compact quantum group. Any unitary representation $\pi$ of $ K_q $ on a Hilbert space $H$ can be decomposed into 
a direct sum of its isotypical components, 
$$
H = \bigoplus_{\sigma\in \Irr(K_q)} H_\sigma,
$$
where $H_\sigma \cong \Hom_{K_q}(V^\sigma,H) \otimes V^\sigma$.
We denote by $p_\sigma$ the orthogonal projection onto $H_\sigma$. More generally, for any set $S \subseteq \Irr(K_q)$, we write  
$ p_S = \sum_{\sigma \in S} p_\sigma$, so that $ p_S $ is the orthogonal projection onto $H_S \defeq \bigoplus_{\sigma \in S} H_\sigma$.

An important observation for what follows is that certain sufficiently nice subspaces of $H$, such as weight spaces, still admit a $K_q$-isotypical decomposition even though they may not be $K_q$-subrepresentations.  This is the point of the following definition. 

\begin{definition}
\label{def:G-harmonic}
A {\em $K_q$-harmonic space} is a Hilbert space of the form $\sH = PH$, where $H$ is a unitary $K_q$-representation space and $P:H\to H$ is an orthogonal projection which commutes with every isotypical projection $p_\sigma$ for $\sigma\in\Irr(K_q)$. 
In this case each $p_\sigma$ restricts to a projection on $\sH$, and we call $\sH_\sigma = p_\sigma PH$ the {\em $\sigma$-isotypical subspace} of $\sH$.
\end{definition}

For us, the key example of a $K_q$-harmonic space will be the $L^2$-section space of a homogeneous line bundle over the quantized flag manifold. 
The corresponding Hilbert space is not a subrepresentation of the right regular representation of $\SU_q(n)$, but it is a $\SU_q(n)$-harmonic space 
with respect to the right regular representation, see Example \ref{ex:fully_harmonic} below.


\subsection{Harmonically finite and harmonically proper operators}
\label{sec:harmonically_finite_and_proper}

Let $\sH$, $\sH'$ be $K_q$-harmonic spaces and let $T\in\LL(\sH,\sH')$ be a bounded linear operator between them. We denote by 
$T_{\sigma\tau} \defeq p_\sigma T p_\tau $ for $\sigma,\tau\in\Irr(K_q)$ the matrix components of $T$ with respect to the $K_q$-isotypical 
decompositions.

\begin{definition}
\label{def:harmonically_proper}
Let $\sH$, $\sH'$ be $K_q$-harmonic spaces. With the notation above, we say that 
\begin{bnum}
\item[a)] $T$ is {\em $K_q$-harmonically finite} if $T_{\sigma\tau} = 0$ for all but finitely many pairs $\sigma,\tau\in\Irr(K_q)$;
\item[b)] $T$ is {\em $K_q$-harmonically proper} if the matrix of $T$ is row- and column-finite, that is, if for each fixed $\sigma$ we have $T_{\sigma\tau} = 0$ for all but finitely many $\tau\in\Irr(K_q)$ and $T_{\tau\sigma} = 0$ for all but finitely many $\tau\in\Irr(K_q)$.  
\end{bnum}
\end{definition}

\begin{definition}
\label{def:K}
Let $\sH$, $\sH'$ be $K_q$-harmonic spaces.  
\begin{bnum}
\item[a)] We define $\KK_{K_q}(\sH,\sH')$ to be the norm-closure of the set of $K_q$-harmonically finite operators in $\LL(\sH,\sH')$ 
\item[b)] We define $\AA_{K_q}(\sH,\sH')$ to be the norm-closure of the set of $K_q$-harmonically proper operators in $\LL(\sH,\sH')$.
\end{bnum}
If $\sH = \sH'$ we will simply write $\KK_{K_q}(\sH)$ and $\AA_{K_q}(\sH)$, respectively.  
\end{definition}

These definitions can be thought of as defining the $\Hom$-sets of $C^*$-categories $\KK_{K_q}$ and $\AA_{K_q}$ whose objects are ${K_q}$-harmonic spaces. 
This observation will serve us as a notational convenience, since it allows us to write statements such as $T\in \KK_{K_q}$ if the domain and 
target spaces of the operator $T$ are understood.

\begin{remark}
\label{rmk:Roe_algebras}
The above definitions can be reinterpreted in the language of coarse geometry.  A $K_q$-harmonic space $\sH$ is a geometric $|\Irr(K_q)|$-Hilbert space, which is merely to say that it admits a representation of $C_0(\Irr(K_q)) = Z(C_0(\hat{K}_q))$.  The algebra $\AA_{K_q}(\sH)$ is basically the Roe algebra with respect to the indiscrete coarse structure on the discrete space $\Irr(K_q)$, see \cite{Roe:coarse_book} for more information. The fact that Roe algebras over dual spaces enter into our $K$-homology construction is no surprise: see the discussions in \cite{Roe:dual_control}, \cite{Luu}, \cite{Yuncken:coarse}.
\end{remark}


\subsection{Alternative characterizations}

We write $S\finitesubset\Irr({K_q})$ if $S$ is a finite set of ${K_q}$-types.  Recall that we write $p_S \defeq \sum_{\sigma\in S} p_\sigma$.  It is convenient to regard $(p_S)_{S\finitesubset \Irr({K_q})}$ as a net of projections, where the indexing set is ordered by inclusion of subsets. 

The following two lemmas are exact analogues of Lemmas 3.4 and 3.5 of \cite{Yuncken:PsiDOs}, with essentially the same proofs.  
 
\begin{lemma}
\label{lem:K_equivalent_defns}
Let $T \in \LL(\sH,\sH')$, where $\sH$, $\sH'$ are ${K_q}$-harmonic spaces. The following conditions are equivalent: 
\begin{bnum}
\item[a)] $T\in\KK_{K_q}(\sH,\sH')$,
\item[b)] $\lim_{S\finitesubset\Irr({K_q})} (1-p_S) T = 0 = \lim_{S\finitesubset\Irr({K_q})} T (1-p_S)$ in the norm topology.
\item[c)] $\lim_{S\finitesubset\Irr({K_q})} p_S T p_S = T$ in the norm topology.
\end{bnum}
\end{lemma}

\begin{lemma}
\label{lem:A_equivalent_defns}
Let  $A \in \LL(\sH,\sH')$, where $\sH$, $\sH'$ are ${K_q}$-harmonic spaces. The following conditions are equivalent: 
\begin{bnum}
\item[a)] $A\in \AA_{K_q}(\sH,\sH')$,
\item[b)] For any finite set $S'\finitesubset \Irr({K_q})$, 
$$
\lim_{S\finitesubset\Irr({K_q})} (1-p_S) A p_{S'} = 0=\lim_{S\finitesubset\Irr({K_q})} p_{S'} A (1-p_S)
$$ 
in the norm topology,
\item[c)] For any finite set $S\finitesubset \Irr({K_q})$, $A p_S \in \KK_{K_q}(\sH,\sH')$ and $p_S A \in \KK_{K_q}(\sH,\sH')$.
\item[d)] $A$ is a two-sided multiplier of $\KK_{K_q}$, that is, for any ${K_q}$-harmonic space $\sH''$, $TA\in\KK_{K_q}(\sH,\sH'')$ for all $T\in\KK_{K_q}(\sH',\sH'')$ and $AT\in\KK_{K_q}(\sH'',\sH')$ for all $T\in\KK_{K_q}(\sH'',\sH)$.
\end{bnum}
\end{lemma}


\subsection{Basic properties}
\label{sec:subgroups}


If the ${K_q}$-isotypical components of a $K_q$-harmonic space $\sH$ are all finite dimensional, we shall say that $\sH$ {\em has finite ${K_q}$-multiplicities}.  In this case, the family $(p_\sigma)_{\sigma\in\Irr({K_q})}$ is a complete system of mutually orthogonal finite-rank projections on $\sH$, so the following result follows from Lemma \ref{lem:K_equivalent_defns}.

\begin{lemma}
\label{lem:compact_operators}
If either $\sH$ or $\sH'$ has finite ${K_q}$-multiplicities then $\KK_{K_q}(\sH,\sH') = \KK(\sH,\sH')$, the set of compact operators from $\sH$ to $\sH'$.
\end{lemma}

If $H$ is a $ {K_q} $-representation and $ {K'_q} $ is a closed quantum subgroup of ${K_q}$, then $H$ is a $ {K'_q}$-representation by restriction. 
We thus have projections $p_{S'}$ on $H$ for every $S'\subset\Irr({K'_q})$. The following result is a straightforward consequence of considering successive isotypical decompositions. 

\begin{lemma}
\label{lem:commuting_projections}
Let ${K'_q} \subseteq {K_q}$ be a closed quantum subgroup. For any $S \subseteq \Irr({K_q})$, $S' \subseteq \Irr({K'_q})$, 
the projections $p_{S}$ and $p_{S'}$ commute.
In particular, if $H$ is a unitary ${K_q}$-representation space and $\tau\in\Irr({K'_q})$ then $p_\tau H$ is a ${K_q}$-harmonic space. 
\end{lemma}

\begin{lemma}
\label{lem:subgroups}
Let ${K'_q} \subseteq K_q $ be a closed quantum subgroup. Suppose that $\sH_1$  and $\sH_2$ are simultaneously ${K_q}$-harmonic and ${K'_q}$-harmonic 
spaces, in the sense that $\sH_i = P_iH_i$ for $i = 1,2$ where $H_i$ is a unitary $ {K_q} $-representation, and $P_i:H_i\to H_i$ is an orthogonal projection 
which commutes with both the ${K_q}$- and the ${K'_q}$-isotypical projections. 
Then $\KK_{K_q}(\sH_1,\sH_2) \subseteq \KK_{{K_q}'}(\sH_1,\sH_2)$.  
\end{lemma}

\begin{proof}
Let $T\in\LL(\sH_1,\sH_2)$ be ${K_q}$-harmonically finite, so $p_S T p_S = T $ for some finite set $S\finitesubset\Irr({K_q})$. 
Only finitely many $ {K'_q} $-types occur in each $ \sigma \in S $, so $ T $ is also ${K'_q}$-harmonically finite. 
The claim follows. 
\end{proof}

\subsection{Commuting generating quantum subgroups}
\label{sec:commuting_generating}

\begin{definition}
\label{def:commuting_generating}
Let $K_{1,q}$, $K_{2,q}$ be closed quantum subgroups of a compact quantum group ${K_q}$, defined by projections $\pi_i:\Poly({K_q})\onto \Poly(K_{i,q})$ for $i=1,2$.  We shall say that $K_{1,q}$ and $K_{2,q}$ are {\em commuting and generating} if $(\pi_1\otimes \pi_2)\Delta = (\pi_1\otimes\pi_2)\Delta^\cop$ holds, 
and this map is an injection of $\Poly({K})$ into $\Poly({K}_1)\otimes\Poly({K}_2)$.
\end{definition}

The subgroups ${K}_{i,q}$ give rise to injections $\pi_i^*:C(\hat{K}_{1,q}) \to C(\hat{K_q})$. 
One can check that $K_{1,q}$ and $K_{2,q}$ are commuting and generating if and only if $\pi_1^*(C(\hat{K}_{1,q}))$ and $\pi_2^*(C(\hat{K}_{2,q}))$ commute and generate a subalgebra of $C(\hat{K}_q)$ which is separating for $\Poly(K_q)$.  The latter condition is often easier to check.

Consider the direct product $ K_{1,q} \times K_{2,q} $ defined by the tensor product $ \Poly(K_{1,q} \times K_{2,q}) = \Poly(K_{1,q}) \otimes \Poly(K_{2,q}) $. 
Note that $\Irr(K_{1,q} \times K_{2,q}) = \Irr(K_{1,q}) \times \Irr(K_{2,q})$, where a pair $(\sigma_1, \sigma_2) \in \Irr(K_{1,q}) \times \Irr(K_{2,q})$ is identified with the obvious corepresentation $\sigma_1\times\sigma_2 $ of $\Poly(K_{1,q}) \otimes \Poly(K_{2,q})$ on $ \KK(V^{\sigma_1} \otimes V^{\sigma_2})$.  Thanks to the embedding $(\pi_1\otimes \pi_2)\Delta:\Poly(K_q) \to \Poly(K_{1,q}) \otimes \Poly(K_{2,q})$, any corepresentation $\sigma$ of $K_q$ defines a corepresentation $\tilde\sigma$ of $K_{1,q} \times K_{2,q}$. If $\sigma$ is irreducible, an application of Schur's Lemma shows 
that $\tilde\sigma=\sigma_1 \times \sigma_2$ for some $\sigma_i\in\Irr(K_{i,q})$, and moreover $\sigma$ is uniquely determined by  $(\sigma_1,\sigma_2)$.  
We therefore have an injection $\Irr(K_q) \into \Irr(K_{1,q}) \times \Irr(K_{2,q})$.

\begin{lemma}
\label{lem:product_groups}
Let $K_{1,q}, K_{2,q}$ be commuting and generating closed quantum subgroups of a compact quantum group $K_q$.  Then for any $K_q$-representations $H$, $H'$ we have
\begin{equation}
\label{eq:K_for_products}
\KK_{K_q}(\sH,\sH') = \KK_{K_{1,q}}(\sH,\sH')\cap\KK_{K_{2,q}}(\sH,\sH'),
\end{equation}
and
\begin{equation}
\label{eq:A_for_products}
\AA_{K_q}(\sH,\sH') \supseteq \AA_{K_{1,q}}(\sH,\sH')\cap\AA_{K_{2,q}}(\sH,\sH').
\end{equation}
\end{lemma}

\begin{proof}
For $\sigma \in \Irr(K_q)$, let $\tilde\sigma = \sigma_1 \times \sigma_2$ be the associated representation of $K_{1,q}\times K_{2,q}$. The isotypical projection for $\sigma$ is given by $p_\sigma = p_{\sigma_1} p_{\sigma_2}$.  It follows that an operator $T: H \to H'$ is $K_q$-harmonically finite if and only if it is both $K_{1,q}$- and $K_{2,q}$-harmonically finite.  This proves Equation \eqref{eq:K_for_products}. Equation \eqref{eq:A_for_products} follows from the characterization of $\AA_{K_q}$ as multipliers of $\KK_{K_q}$, as in Lemma \ref{lem:A_equivalent_defns}.
\end{proof}

\subsection{Harmonic properties of tensor products}
\label{sec:tensor_products}

If $\sH_1 = P_1H_1$ and $\sH_2 = P_2H_2$ are ${K_q}$-harmonic spaces, following the notation of Definition \ref{def:G-harmonic}, then the tensor product $\sH_1\otimes\sH_2 = (P_1\otimes P_2)H_1\otimes H_2$ is naturally a ${K_q}$-harmonic space with respect to the tensor product representation of $K_q$ on $H_1\otimes H_2$.

\begin{lemma}
\label{lem:K_tensor_K}
Let ${K_q}$ be a compact quantum group. Then $\KK_{K_q} \otimes \KK_{K_q} \subseteq \KK_{K_q}$, in the sense that for any ${K_q}$-harmonic spaces  $\sH_1,\sH_2,\sH'_1,\sH'_2$ and any $T_1\in\KK_{K_q}(\sH_1,\sH'_1)$ and $T_2\in\KK_{K_q}(\sH_2,\sH'_2)$ we have $T_1 \otimes T_2 \in \KK_{K_q}(\sH_1\otimes\sH_2,\sH'_1\otimes\sH'_2)$.

Similarly, $\AA_{K_q}\otimes\KK_{K_q}\subseteq\AA_{K_q}$ and $\KK_{K_q}\otimes\AA_{K_q} \subseteq\AA_{K_q}$. 
\end{lemma}

\begin{proof}
Suppose $T_1$ and $T_2$ are ${K_q}$-harmonically finite, so that for $i=1,2$ there are finite sets $S_1$, $S_2 \subset \Irr({K_q})$ such that $p_{S_i} T_i p_{S_i} = T_i$. If $S$ denotes the set of all irreducible ${K_q}$-types which occur in some $\sigma_1\otimes\sigma_2$ with $\sigma_i\in S_i$, then $T_1 \otimes T_2 = p_S (T_1 \otimes T_2) p_S$.  From this we deduce $\KK_{K_q} \otimes \KK_{K_q} \subseteq \KK_{K_q}$.

Now suppose $A$ is ${K_q}$-harmonically proper and $T$ is ${K_q}$-harmonically finite. 
Fix $S\subset\Irr({K_q})$ a finite set of ${K_q}$-types such that $p_STp_S = T$. Take $\sigma\in\Irr({K_q})$ arbitrary. Then
$$ 
(A \otimes T)p_\sigma = (A\otimes T) (1 \otimes p_S) p_\sigma.
$$  
Let $\tau\in S$.  For any $\tau'\in\Irr(K_q)$, we have $\sigma \leq \tau'\otimes\tau$ if and only if $ \tau'\leq \sigma \otimes \tau^c $.   This implies that there are only 
finitely many $\tau' \in \Irr({K_q})$ for which $(p_{\tau'} \otimes p_S)p_\sigma \neq 0$.  Letting $S'\finitesubset\Irr({K_q})$ denote the set of such $\tau'$, we have
$$
(A \otimes T)p_\sigma = (A \otimes T)(p_{S'} \otimes p_S) p_\sigma ,
$$
From the ${K_q}$-harmonic properness of $A$ and $T$ we can deduce that $(A \otimes T) p_\sigma \in \KK_{K_q}$.  A similar argument 
shows $p_\sigma (A \otimes T) \in \KK_{K_q}$ for all $\sigma\in\Irr({K_q})$, whence Lemma \ref{lem:A_equivalent_defns} shows that $A \otimes T \in \AA_{K_q}$. 
Clearly, a similar argument works for $ T \otimes A $. 
\end{proof}

Recall that $\triv_{K_q}$ denotes the trivial representation of ${K_q}$. Later we shall make much use of the following trick, which allows us to replace 
an arbitrary isotypical projection by the trivial one.

\begin{lemma}
\label{lem:reduction_to_triv}
Let $\sigma\in\Irr({K_q})$ and let $V$ be any finite dimensional representation of ${K_q}$ which contains $\sigma$ as a subrepresentation. There exist 
intertwiners $\iota: \CC \to V^c \otimes V $ and $ \bar{\iota}: V^c \otimes V \to \CC$ such that on any unitary ${K_q}$-representation $H$, the isotypical projection $p_\sigma$ factorizes as
$$
 \xymatrix{
  p_\sigma:  H \ar[r]^-{\Id_H \otimes \iota} 
   & H \otimes V^c \otimes V \ar[r]^-{p_{\triv_{K_q}}\otimes\Id_{V}} 
   & H \otimes V^c \otimes V \ar[r]^-{\Id_H\otimes \bar{\iota}} 
   & H.
 }
$$ 
\end{lemma}

\begin{proof}
If $V = V^\sigma$, this follows from standard facts about the contragredient representation. 
If $ V^\sigma$ is merely a subrepresentation of $V$, then we can use the inclusion map $ V^\sigma \rightarrow V $ 
and the projection $ V \rightarrow V^\sigma $, as well as the corresponding maps for the contragredient representation, to reduce 
to the previous situation. 

\end{proof}


\section{The lattice of $C^*$-ideals}
\label{sec:lattice_of_ideals}

We now specialize to the quantum group ${K_q}=\SU_q(n)$, although we note that the constructions and results of this section translate naturally to more 
general $q$-deformed compact semisimple Lie groups. We recall the family of quantum subgroups $K^\rootset_q \subseteq K_q$ defined in Section \ref{sec:lattice_of_subgroups}.

\begin{definition}
\label{def:fully-harmonic}
A {\em fully $K_q$-harmonic space} is 
a Hilbert space of the form $\sH = PH$, where $H$ is a unitary representation of $K_q$ and $P$ is an orthogonal projection which commutes with 
all isotypical projections of each $K^\rootset_q$, $\rootset\subseteq\Sigma$.
\end{definition}

Thus, a fully $K_q$-harmonic space is simultaneously a $K_q^\rootset$-harmonic space for every $\rootset\subseteq\Sigma$.  Between fully $K_q$-harmonic spaces $\sH$ and $\sH'$, we have the spaces $\KK_{K^\rootset_q}(\sH,\sH')$ and $\AA_{K^\rootset_q}(\sH,\sH')$ for every $\rootset\subseteq\Sigma$.
To avoid unwieldy subscripts, we shall write $\KK_{\rootset}$ and $\AA_{\rootset}$ for $\KK_{K^\rootset_q}$ and $\AA_{K^\rootset_q}$ 
in the sequel.  When $\rootset=\{i\}$ is a singleton, we shall write $\KK_i$ and $\AA_i$.  

The only examples of fully $K_q$-harmonic spaces we shall actually need are the following.

\begin{example}
\label{ex:fully_harmonic}
\begin{bnum}
\item[a)] Any unitary representation of $K_q$ is a fully $K_q$-harmonic space.
\item[b)] By Lemma \ref{lem:commuting_projections}, any weight space of a $K_q$-representation is a fully $K_q$-harmonic space.
\item[c)] In particular, the $L^2$-section space $L^2(\sE_\mu)$ of a homogeneous line bundle over the quantized flag manifold of $K_q$ is a fully $K_q$-harmonic space.
Note that the harmonic structure here comes from the {\em right} regular corepresentation.  
\end{bnum}
\end{example}

Lemma \ref{lem:subgroups} shows that we have a whole lattice of $C^*$-categories $(\KK_\rootset)_{\rootset\subseteq\Sigma}$ for the fully $K_q$-harmonic spaces. Note that the ordering is reversed: $\KK_{\rootset_1} \subseteq \KK_{\rootset_2}$ if $\rootset_1 \supseteq \rootset_2$.

We point out, however, that this is typically not a lattice of $C^*$-ideals, that is, given sets $\rootset_1 \supset \rootset_2$ of simple roots and a fully $K_q$-harmonic space $\sH$ we typically do not have $\KK_{\rootset_1}(\sH) \triangleleft \KK_{\rootset_2}(\sH)$.
To obtain a lattice of ideals, we must reduce $\KK_\rootset$ slightly, by restricting the class of operators we are working with.

\begin{definition}
\label{def:J}
For fully $K_q$-harmonic spaces $\sH$, $\sH'$, we define 
$$
\AA(\sH,\sH') = \bigcap_{\rootset\subseteq\Sigma} \AA_\rootset(\sH,\sH').
$$
We also define
$$
\JJ_\rootset(\sH,\sH') \defeq \KK_\rootset(\sH,\sH') \cap \AA(\sH,\sH')
$$ 
for each $\rootset\subseteq\Sigma$.
\end{definition}
In other words, $\AA$ is the simultaneous multiplier category of all the $C^*$-categories $\KK_{\rootset}$.  
Again, we view the spaces defined in definition \ref{def:J} as the morphism sets of $C^*$-categories $\AA$ and $\JJ_\rootset$ whose objects are fully $K_q$-harmonic spaces. 
It is immediate from Lemma \ref{lem:subgroups} that the $\JJ_\rootset$ form a lattice of ideals, as we record in the following lemma. 

\begin{lemma}
\label{lem:lattice_of_ideals}
If $ \rootset_1 \supseteq \rootset_2$ then $\JJ_{\rootset_1} \ideal \JJ_{\rootset_2} $. 
\end{lemma}

In particular, we have $\JJ_{\rootset_1\cup\rootset_2} \subseteq \JJ_{\rootset_1} \cap \JJ_{\rootset_2}$ for any $\rootset_1,\rootset_2 \subseteq \Sigma$. In fact, it will be shown later that $\JJ_{\rootset_1\cup\rootset_2} = \JJ_{\rootset_1} \cap \JJ_{\rootset_2}$, see Theorem \ref{thm:lattice_of_ideals} and its 
proof in Section \ref{sec:lattice_of_ideals_II}.

\begin{lemma}
\label{lem:J_compact_operators}
Let $\sH$, $\sH'$ be fully $K_q$-harmonic spaces. Then $ \JJ_\Sigma(\sH,\sH')= \KK_\Sigma(\sH,\sH')$. In particular, if either $\sH$ or $\sH'$ has 
finite $K_q$-multiplicities then $ \JJ_\Sigma(\sH,\sH')=\KK(\sH,\sH')$, the set of compact operators from $\sH$ to $\sH'$.
\end{lemma}

\begin{proof}
By Lemma \ref{lem:subgroups}, we have 
$\KK_\Sigma \subseteq \bigcap_{\rootset\subseteq\Sigma} \KK_\rootset \subseteq \bigcap_{\rootset\subseteq\Sigma} \AA_\rootset = \AA$. 
This proves the first statement.  The second follows using Lemma \ref{lem:compact_operators}.
\end{proof}

The spaces of interest to us will have finite $K_q$-multiplicities, however they will usually not have finite $K^\rootset_q$-multiplicities 
for $\rootset\neq\Sigma$.


\section{Longitudinal pseudodifferential operators: statement of results}
\label{sec:statements}

In this section, we give statements of the necessary results concerning the lattice of $C^*$-categories $(\JJ_\rootset)_{\rootset\subseteq\Sigma}$ and the ``pseudodifferential'' operators $\ph(E_i)$ and $\ph(F_i)$. All of these results will be discussed at the generality of $\SU_q(n)$. 
The proofs of the theorems below will be deferred until Sections \ref{sec:lattice_of_ideals_II} and \ref{sec:PsiDOs2}. The reader willing to accept 
their veracity may safely skip forward to Section \ref{sec:SLq} after this section.

\medskip

We begin with general results on the lattice of ideals $(\JJ_\rootset)_{\rootset\subseteq\Sigma}$.

\begin{theorem}
\label{thm:lattice_of_ideals}
Let $K_q = \SU_q(n)$ for $n\geq2$ and let $\sH$, $\sH'$ be fully $K_q$-harmonic spaces.
\begin{bnum}
\item[a)] $\AA(\sH,\sH') = \bigcap_{i\in\Sigma} \AA_i(\sH,\sH')$.
\item[b)] For any $\rootset\subseteq\Sigma$ and any $\sigma\in\Irr(K^\rootset_q)$, $p_\sigma \in \JJ_\rootset(\sH)$.
\item[c)] For any $\rootset, \rootset' \subseteq \Sigma$, $\JJ_{\rootset}(\sH,\sH') \cap \JJ_{\rootset'}(\sH,\sH') = \JJ_{\rootset \cup \rootset'}(\sH,\sH')$.
\item[d)] If either $\sH$ or $\sH'$ has finite $K_q$-multiplicities then $\JJ_\Sigma(\sH,\sH') = \KK(\sH,\sH')$, the compact operators from $\sH$ to $\sH'$, and hence $\bigcap_{i\in\Sigma} \JJ_i(\sH,\sH') = \KK(\sH,\sH')$.
\end{bnum}
\end{theorem}

\medskip

Next we consider longitudinal pseudodifferential operators along various fibrations of the quantum flag manifold. 
The guiding philosophy is that for $\mu, \nu \in \bP$, $\AA_\rootset(L^2(\sE_\mu), L^2(\sE_{\nu}))$ should be thought of as containing 
the order zero longitudinal pseudodifferential operators along the leaves of the fibration $K_q/T \onto K_q/K_q^\rootset$, 
while $\KK_\rootset(L^2(\sE_\mu), L^2(\sE_{\nu}))$ should be thought of as the ideal of negative order longitudinal pseudodifferential operators. 
For instance, in the case $q = 1$ the space $\AA_\rootset(L^2(\sE_\mu), L^2(\sE_{\nu}))$ contains all order zero longitudinal pseudodifferential operators along the fibration, although it also contains many other operators, such as translations by the group action. Nevertheless, the reader should keep the analogy in mind when interpreting the next theorem.

Let $\mu\in\bP$ and $i\in\Sigma$. The unbounded operator 
$$
D_i = \begin{pmatrix} 0 & F_i \\ E_i & 0 \end{pmatrix} \qquad \text{on }L^2(\sE_\mu \oplus \sE_{\mu+\alpha_i})
$$
is an essentially self-adjoint operator with domain $\Poly(\sE_\mu \oplus \sE_{\mu+\alpha_i})$.
It is to be thought of as a longitudinal differential operator along the leaves of the fibration $\sX_q \twoheadrightarrow K_q / K_{i,q}$. 
Notice that these operators are essentially families of Dirac operators of the type considered by Dabrowski-Sitarz in \cite{DS:Podles}, over the base 
space $ K_q/K_{i,q} $. 
We denote the operator phase of $ D_i $ by
$$
\ph(D_i) = \begin{pmatrix} 0 & \ph(F_i) \\ \ph(E_i) & 0 \end{pmatrix}, 
$$ 
which should be thought of as a longitudinal pseudodifferential operator of order $0$.  

Recall that if $f\in C(\sE_\nu)$ for some $\nu\in\bP$, then the left and right multiplication actions $\Mult{f}$, $\RMult{f}$ define operators 
in $ \LL(L^2(\sE_\mu), L^2(\sE_{\mu+\nu}))$.

\begin{theorem}
\label{thm:PsiDOs}
Let $K_q = \SU_q(n)$ for $n\geq2$.  Let $\mu\in\bP$, $i\in\Sigma$ and $f\in \Poly(\sE_\nu)$ for some $\nu\in\bP$. Then the following hold. 
\begin{bnum}
\item[a)] $\Mult{f}$ and $\RMult{f}$ are in $\AA(L^2(\sE_\mu), L^2(\sE_{\mu+\nu}))$.
\item[b)] $\ph(D_i) \in \AA( L^2(\sE_\mu \oplus \sE_{\mu+\alpha_i}))$  
\item[c)] For any $\psi\in C_0(\RR)$, we have $\psi(D_i) \in \JJ_i(L^2(\sE_\mu \oplus \sE_{\mu+\alpha_i}))$, or equivalently,  $D_i$ has resolvent in $\JJ_i(L^2(\sE_\mu \oplus \sE_{\mu+\alpha_i}))$.
\item[d)] The following diagram 
$$
\xymatrix{
L^2(\sE_\mu\oplus\sE_{\mu+\alpha_i}) \ar[r]^{\ph(D_i)} \ar[d]_{\Mult{f}} 
      & L^2(\sE_\mu\oplus\sE_{\mu+\alpha_i}) \ar[d]^{\Mult{f}} \\
  L^2(\sE_{\mu+\nu} \oplus \sE_{\mu+\nu+\alpha_i}) \ar[r]_{\ph(D_i)} 
      & L^2(\sE_{\mu+\nu} \oplus \sE_{\mu+\nu+\alpha_i}) 
 }
$$
commutes up to an element of $ \JJ_i(L^2(\sE_\mu \oplus \sE_{\mu+\alpha_i}), L^2(\sE_{\mu+\nu} \oplus \sE_{\mu+\nu+\alpha_i}))$. 
\end{bnum}
\end{theorem}

\begin{remark}
\label{rmk:phEi_versions}
By slight abuse of notation, we will usually abbreviate part $d)$ of Theorem \ref{thm:PsiDOs} by writing
$ [\ph(D_i), \Mult{f}] \in \JJ_i(L^2(\sE_\mu \oplus \sE_{\mu+\alpha_i}), L^2(\sE_{\mu+\nu} \oplus \sE_{\mu+\nu+\alpha_i})) $ in the sequel. 
Notice that above statements about $\ph(D_i)$ can be restated as results about $\ph(E_i)$ and $\ph(F_i)$. 
In particular, part d) is equivalent to the commutativity of the diagrams
$$
 \xymatrix{
  L^2(\sE_\mu) \ar[r]^{\ph(E_i)} \ar[d]_{\Mult{f}} 
      & L^2(\sE_{\mu+\alpha_i}) \ar[d]^{\Mult{f}} 
  & L^2(\sE_{\mu+\alpha_i}) \ar[r]^{\ph(F_i)} \ar[d]_{\Mult{f}} 
      & L^2(\sE_\mu) \ar[d]^{\Mult{f}} \\ 
  L^2(\sE_{\mu+\nu}) \ar[r]_{\ph(E_i)} 
      & L^2( \sE_{\mu+\nu+\alpha_i}) 
  & L^2(\sE_{\mu+\nu+\alpha_i}) \ar[r]_{\ph(F_i)} 
      & L^2(\sE_{\mu+\nu}) 
 }
$$
modulo $ \JJ_i(L^2(\sE_\mu), L^2(\sE_{\mu+\nu + \alpha_i}))$ and $ \JJ_i(L^2(\sE_{\mu+\alpha_i}), L^2(\sE_{\mu+\nu})) $, respectively. 
\end{remark}


\section{Comparisons of Gelfand-Tsetlin bases}
\label{sec:Gelfand-Tsetlin}

This section and the next provide the technical results from harmonic analysis which will be used to prove Theorems \ref{thm:lattice_of_ideals} and \ref{thm:PsiDOs}. It will be convenient to work with the quantum group $\U_q(n)$ rather than $\SU_q(n)$.


\subsection{Quantum subgroups of $\U_q(n)$}
\label{sec:subgroups_of_U_q}

We first introduce notation for the block diagonal quantum subgroups of $\U_q(n)$. These definitions follow the notation and conventions of Sections \ref{sec:env_alg} and \ref{sec:lattice_of_subgroups}. We define $C(\widehat{\U_q(n)}) = \Poly(\U_q(n))^*$. For $I\subseteq\Sigma$, we define the $\sigma(C(\widehat{\U_q(n)}), \Poly(\U_q(n)) )$-closed subalgebra
$$
C(\widehat{\U^\rootset_q(n)}) = {\langle   E_i, F_i ~(i\in\rootset) , ~ G_j ~(j=1,\ldots,n) \rangle} .
$$
and denote the associated closed quantum subgroup of $\U_q(n)$ by $\U_q^\rootset(n)$.

In the particular cases $I = \{1,\ldots,k-1\}$, we will decompose $\U_q^\rootset(n)$ as follows. 
Let $ C(\widehat{\U^\upp_q(k)}) = {\langle E_i, F_i ~(i=1,\ldots,k-1), G_j ~(j=1,\ldots,k) \rangle} $ and 
$ C(\widehat{Z^\upp_k}) = {\langle G_j (j=k+1,\ldots,n) \rangle} $, and let $\U^\upp_q(k)$ and 
$Z^\upp_k$ be the dual closed quantum subgroups of $\U_q(n)$. Then $\U_q^{\{1,\ldots,k-1\}}(n) = \U^\upp_q(k) \times Z^\upp_k$. 
The superscript $\upp$ refers to the fact that $\U^\upp_q(k)\cong\U_q(k)$ is embedded in the ``upper-left corner'' of $\U_q(n)$. 
We likewise decompose $\U_q^{\{n-k+1,\ldots,n-1\}}(n) = \U_q^\low(k) \times Z^\low_k$ where the two components are dual 
to $ C(\widehat{\U^\low_q(k)}) = {\langle E_i, F_i ~(i=n-k+1,\ldots,n-1), G_j~ (j=n-k+1,\ldots,n) \rangle} $ and 
$ C(\widehat{Z^\low_k}) = {\langle G_j~ (j=1,\ldots,n-k) \rangle} $, respectively.

\subsection{Upper and lower Gelfand-Tsetlin bases}
\label{sec:alternative_GTs_bases}

Consider the nested family of quantum groups
$$
  T = \U_q^\emptyset(n) \subset \U_q^{\{1\}}(n) \subset \U_q^{\{1,2\}}(n) \subset \cdots \subset 
    \U_q^{\Sigma}(n) = \U_q(n). 
$$
The isotypical projections of these quantum subgroups are mutually commuting. Gelfand-Tsetlin theory is based upon the observation that the simultaneous isotypical decomposition for all of these subgroups yields components of dimension one, and thus provides a basis which is well-adapted for all of them.  We shall refer to this as the {\em upper} Gelfand-Tsetlin basis.  We recall the main facts about the Gelfand-Tsetlin basis here, and refer to \cite[\S7.3]{KS} for the details.

The highest weights of type $1$ representations of $\Uq(\glx_n)$ are given by those $\mu = (\mu_1,\ldots,\mu_n) \in \ZZ^n$ with $\mu_1\geq\mu_2\geq\cdots\geq \mu_n$.  We denote the irreducible representation with highest weight $\mu$ by $\sigma^\mu$.  The Gelfand-Tsetlin basis for $V^{\sigma^\mu}$ is indexed by tableaux of integers of the form
$$
(M) = \left( \begin{array}{cccccccccc}
  \multicolumn{2}{c}{m_{n,1}} &
  \multicolumn{2}{c}{m_{n,2}} &
  \multicolumn{2}{c}{\cdots} &
  \multicolumn{2}{c}{m_{n,n\!-\!1}} &
  \multicolumn{2}{c}{m_{nn}} \\
&\multicolumn{2}{c}{m_{n\!-\!1,1}}
&\multicolumn{4}{c}{\cdots}
&\multicolumn{2}{c}{m_{n\!-\!1,n\!-\!1}} \\
&&\multicolumn{2}{c}{\ddots}
&&&\multicolumn{2}{c}{\adots}\\
&&&\multicolumn{2}{c}{m_{21}}
&\multicolumn{2}{c}{m_{22}} \\
&&&&\multicolumn{2}{c}{m_{11}}
\end{array} \right),
$$
where the top row is equal to $\mu$ and the lower rows satisfy the interlacing conditions $m_{i\!+\!1,j} \leq m_{ij} \leq m_{i\!+\!1,j\!+\!1}$ for all $i,j$.  The corresponding basis element, which will be denoted $\ket{(M)^\upp}$, is determined up to phase by the fact that for each $k=1,\ldots,n$, the vector 
$\ket{(M)^\upp}$ belongs to a $\U_q^\upp(k)$-subrepresentation with highest weight $(m_{k1},\ldots,m_{kk})$.  Moreover, $\ket{(M)^\upp}$ is a weight vector with weight 
\begin{equation}
 \label{eq:GTs-weight}
 (s_1-s_0,s_2-s_1,\ldots,s_n-s_{n-1}), 
\end{equation}
where $s_i\defeq \sum_{j=1}^i m_{ij}$ is the sum of the $i$th row and $s_0=0$ by convention.

There is an alternative basis of $V^{\sigma^\mu}$ adapted to the lower-right inclusions
$$
  T = \U_q^\emptyset(n) \subset \U_q^{\{n-1\}}(n) \subset \U_q^{\{n-2,n-1\}}(n) \subset \cdots \subset
    \U_q^{\Sigma}(n) = \U_q(n).
$$
This basis is most easily introduced by invoking the Hopf $*$-automorphism $\Psi$ of $\Uq(\lie{gl}_n)$ defined by:
\begin{equation}
 \label{eq:automorphism}
  \Psi(G_j)= G_{n+1-j}^{-1}, \qquad \Psi(E_i)= E_{n-i}, \qquad \Psi(F_i) = F_{n-i}.
\end{equation}
Note that a highest weight vector for $\sigma^\mu$ is also a highest weight vector for the irreducible representation $\sigma^\mu\circ\Psi$, but with 
weight $\mu' = (-\mu_n,\ldots,-\mu_1)$. By Schur's Lemma, there is a unitary $\psi_\mu:V^{\sigma^{\mu'}} \to V^{\sigma^{\mu}}$ (unique up to scalar multiple) which intertwines $\sigma^{\mu'}\circ\Psi$ and $\sigma^\mu$.
We define the {\em lower} Gelfand-Tsetlin basis vectors by $\ket{(M)^\low} = \psi_\mu\ket{(M)^\upp}$, where $(M)$ is a Gelfand-Tsetlin tableau for the representation $\sigma^{\mu'}$.

\medskip

\subsection{Class $1$ representations}
\label{sec:class_1}

Often, we will only be interested in the irreducible $\U_q(n)$-representations which contain a trivial $\U_q^{\{1,\ldots,n-2\}}(n)$-subrepresentation. These are a special case of the {\em class $1$ representations} (see \cite[\S7.3.4]{KS}).  A Gelfand-Tsetlin vector is
contained in a trivial subrepresentation of $ \U_q^\upp(n-1) $
if and only if it is of the form
$$
  \xi_m \defeq \bigket{\GTs{m}{\!\!-m'}{\;0\;\;}{\;\;\;0\;}{\;0}{\;0}{0}^{\!\!\!\!\bigupp\;\;}}
$$
for some $m,m'\in\NN$.  To be contained in a trivial $\U_q^{\{1,\ldots,n-2\}}(n)$-representation, it must additionally be of weight $0$, which is to say $m=m'$.  Thus, the representations of interest are precisely those with highest weight of the form $\mu = (m,0,\ldots,0,-m)$. 
Note that in this case, $\sigma^\mu \cong \sigma^{\mu}\circ\Psi$, so that the upper and lower Gelfand-Tsetlin bases are indexed by the same set of tableaux.

We state the Gelfand-Tsetlin formulae for such representations; compare \cite[\S7.3.4]{KS}. The generic basis vector is
$$
\ket{(M)^\upp} = \bigket{\GTs{m_n}{-m'_n}{m_{n\!-\!1}}{m'_{n\!-\!1}}{m_2}{m'_2}{m_1}^{\!\!\!\!\bigupp\;\;}},
$$
where we are putting $m_n = m'_n = m$ for ease of notation.
We write $(M\pm\delta_{ij})$ to denote the Gelfand-Tsetlin tableau obtained from $(M)$ by adding $\pm1$ to the $(i,j)$-entry.
The action of the generators of $\Uq(\glx_n)$ is given by
\begin{align}
\lefteqn{ E_{k-1} \ket{(M)^\upp} }\quad  \nonumber \\
&= \textstyle
  \left( 
  \frac{[m_k-m_{k-1}][m_{k-1}-m_k'+k-1][m_{k-1}-m_{k-2}+1][m_{k-1}-m_{k-2}'+k-2]}
    {[m_{k-1}-m_{k-1}'+k-1][m_{k-1}-m_{k-1}'+k-2]}
  \right)^{\half} \ket{(M+\delta_{k-1,1})^\upp} \nonumber \\
&\quad + \textstyle \left( 
  \frac{[m_k-m_{k-1}'+k-2][m_{k-1}'-m_k'+1][m_{k-2}-m_{k-1}'+k-3][m_{k-2}'-m_{k-1}']}
    {[m_{k-1}-m_{k-1}'+k-2][m_{k-1}-m_{k-1}'+k-3]}
  \right)^{\half} \ket{(M+\delta_{k-1,k-1})^\upp},
\label{eq:GTs-E} \\
\lefteqn{ F_{k-1} \ket{(M)^\upp} } \quad \nonumber \\
&= \textstyle
  \left( 
  \frac{[m_k-m_{k-1}+1][m_{k-1}-m_k'+k-2][m_{k-1}-m_{k-2}][m_{k-1}-m_{k-2}'+k-3]}
    {[m_{k-1}-m_{k-1}'+k-2][m_{k-1}-m_{k-1}'+k-3]}
  \right)^{\half} \ket{(M-\delta_{k-1,1})^\upp} \nonumber \\
& \quad  + \textstyle \left( 
  \frac{[m_k-m_{k-1}'+k-1][m_{k-1}'-m_k'][m_{k-2}-m_{k-1}'+k-2][m_{k-2}'-m_{k-1}'+1]}
    {[m_{k-1}-m_{k-1}'+k-1][m_{k-1}-m_{k-1}'+k-2]}
  \right)^{\half} \ket{(M-\delta_{k-1,k-1})^\upp},
\label{eq:GTs-F} \\
\lefteqn{ G_i \ket{(M)^\upp} = q^{\half(s_i-s_{i-1})}\ket{(M)^\upp} ,} \quad
\label{eq:GTs-K}
\end{align}
where, as before, $s_i$ is the sum of the $i$th row of $(M)$ and $s_0=0$.


\subsection{Change of basis formula}
\label{sec:change_of_basis}

We now describe certain cases of the change of basis transformation between the upper and lower Gelfand-Tsetlin bases introduced in Section \ref{sec:alternative_GTs_bases}.  We shall concentrate entirely on the family of representations of highest weight $\mu \defeq (m,0,0,\ldots,0,-m)$ for $m\in\NN$.
In either choice of Gelfand-Tsetlin basis, the zero-weight subspace of $\sigma^\mu$ is spanned by the vectors $\ket{(M_\mathbf{m})^\upp}$ or $\ket{(M_\mathbf{m})^\low}$ with tableaux
$$
  M_\mathbf{m} = \GTs{m}{-m}{\!m_{n-1}}{\!\!\!-m_{n-1}}{\!m_2}{\!\!-m_2}{0}.
$$
Here we use $\mathbf{m}$ to denote the increasing $n$-tuple $\mathbf{m} \defeq (m_1=0, m_2, \ldots, m_n=m)$. 

Our first goal is to compute the coefficients of the $\U_q^\low(n-1)$-invariant vector $ \fixedvec $ 
with respect to the upper Gelfand-Tsetlin basis. We write 
\begin{equation}
\label{eq:fixed_vec_coefficients}
\fixedvec = \sum_\mathbf{m} a_\mathbf{m} \ket{(M_\mathbf{m})^\upp}.
\end{equation}
Let us apply $E_k$ to this. The Gelfand-Tsetlin formula \eqref{eq:GTs-E} shows that the coefficient 
of $\ket{(M_\mathbf{m} + \delta_{k,1})^\upp}$ in $E_{k} \fixedvec$ is
\begin{multline}
   \textstyle
  \left( 
  \frac{[m_{k+1}-m_{k}][m_{k}+m_{k+1}+k][m_{k}-m_{k-1}+1][m_{k}+m_{k-1}+k-1]}
    {[2m_{k}+k][2m_{k}+k-1]}
  \right)^{\half} a_\mathbf{m} \\
   + \textstyle \left( 
  \frac{[m_{k+1}+m_{k}+k][-m_{k}+m_{k+1}][m_{k-1}+m_{k}+k-1][-m_{k-1}+m_{k}+1]}
    {[2m_{k}+k+1][2m_{k}+k]}
  \right)^{\half} a_{\mathbf{m}+\delta_{k}}, 
 \label{eq:coefficient}
\end{multline}
where $\mathbf{m}+\delta_{k}$ denotes the $n$-tuple obtained by adding $1$ to the $k$th entry of $\mathbf{m}$.   Since $E_k \fixedvec =0$ for all  $2 \leq k \leq n - 1 $, we obtain the recurrence relation
\begin{equation}
\label{eq:recurrence_relation1}
a_{\mathbf{m}+\delta_{k}} = - \frac{[2m_{k}+k+1]^\half}{[2m_{k}+k-1]^\half} a_\mathbf{m}
\end{equation}
for all $\mathbf{m}$ and all $2 \leq k \leq n - 1 $.
This multi-parameter recurrence relation has the solution
\begin{equation}
\label{eq:recurrence_solution1}
a_\mathbf{m} = (-1)^{|\mathbf{m}|} A \prod_{k=2}^{n-1} [2m_k+k-1]^\half,
\end{equation}
where $|\mathbf{m}| = m_1 + \cdots + m_{n}$ and $A\in\CC$ is some overall constant.  
This constant is determined up to a phase by the fact that $\fixedvec$ has norm one. \linebreak We  will 
assume a choice of phases for the Gelfand-Tsetlin bases such that \linebreak $\braket{(M_{(m,0,\ldots,0)})^\low}{(M_{(m,0,\ldots,0)})^\upp}$ is positive. 
Then $A$ is positive. From \eqref{eq:recurrence_solution1}, we calculate 
$$
1 
= \sum_\mathbf{m} |a_\mathbf{m}|^2 = A^2 \sum_\mathbf{m}  \prod_{k=2}^{n-1} [2m_k+k-1],
$$
where the sum  is over all $n$-tuples $\mathbf{m}$ with $0=m_1 \leq m_2 \leq \cdots \leq m_n = m$. 
An inductive argument shows that
$$
 \sum_\mathbf{m} \prod_{k=2}^{n-1} [2m_k + k -1] = [n-2]! \qchoose{m_n+n-2 \\ n-2}^2,
$$
and one obtains that $A = [n-2]!^{-\half} \qchoose{m+n-2 \\ n-2}^{-1}$. In summary, we have proved the following formula. 

\begin{proposition}
\label{prop:GTs-comparison1}
In the irreducible representation of $\U_q(n)$ with highest weight $(m,0,\ldots,0,-m)$,
$$
\fixedvec = \sum_\mathbf{m} \frac{(-1)^{|\mathbf{m}|} \prod_{k=2}^{n-1} [2m_k+k-1]^\half}{[n-2]!^\half \qchoose{m+n-2 \\ n-2}} \ket{(M_\mathbf{m})^\upp},
$$ 
where the sum  is over all $n$-tuples $\mathbf{m}$ with $0=m_1 \leq m_2 \leq \cdots \leq m_n = m$.
\qed
\end{proposition}


\subsection{Change of basis formula for $\U_q(3)$}

In the case of $\U_q(3)$, the above calculation gives the following change-of-basis coefficients for the trivial $\U_q^{\{2\}}(3)$-type:
\begin{equation}
\label{eq:GTs_comparison1_for_SU3}
\bigbraket{\smallGTsO{m}^\low}{\smallGTs{m}{j}^\upp} = (-1)^{j+m} \frac{[2j+1]^\half}{[m+1]}.
\end{equation}
The complete change-of-basis coefficients between the two Gelfand-Tsetlin bases of any $\U_q(3)$-representation were computed in \cite{MSK}. They are given by $q$-Racah coefficients.  We will only need the following special cases.

\begin{proposition}
 \label{prop:GTs_comparison2_for_SU3}
In the representation of $\U_q(3)$ with highest weight $(m,0,-m)$, consider the vectors
$$
\xket{j} = [2j+1]^{-\half} \bigket{\smallGTs{m}{j}^\upp}, \quad \yket{k} = [2k+1]^{-\half} \bigket{\smallGTs{m}{k}^\low}.
$$ 
Then
\begin{equation}
 \label{eq:GTs_comparison2_for_SU3}
 \yxbraket{k}{j} =  \frac{(-1)^{j+k+m}}{[m+1]}
  \qhypergeometric{4}{3} \left(\left. 
    \begin{array}{c}
      q^{-2k}, q^{2(k+1)}, q^{-2j}, q^{2(j+1)} \\
      q^{-2m}, q^{2(m+2)}, q^2
    \end{array}
  \right| q^2; q^2 \right),
\end{equation}
where $\qhypergeometric{4}{3}$ denotes the $q$-hypergeometric function.  
\end{proposition}

In the $q$-Racah notation of \cite[\S14.2]{KLS}, this translates as
$$
\yxbraket{k}{j} =  (-1)^{j+k+m} [m+1]^{-1} R_k(\mu(j) ; q^{2(m+1)}, q^{-2(m+1)}, 1, 1 | q^2 ),
$$
though we shall not actually use this.

It is rather cumbersome to reconcile the notation and terminology of \cite{MSK} with ours. For this reason, we outline a short proof of Proposition \ref{prop:GTs_comparison2_for_SU3} in Appendix \ref{sec:change_of_basis_proof}.

\subsection{Action of the phase of $ E_1 $ on the lower Gelfand-Tsetlin basis}
\label{sec:asymptotics}

The final task of this section is to compute the action of $\ph(E_1)$ with respect to the lower Gelfand-Tsetlin basis.
Obtaining an explicit formula is difficult. Instead, we compute the action asymptotically as the highest weight $\mu$ goes to infinity, which is 
all that will be necessary for our purposes.

We shall make use of the following basic estimate for products of values near $1$, whose proof we leave to the reader.

\begin{lemma}
\label{lem:product_estimate}
Fix $q\in(0,1)$ and $N\in\NN$.  There is a constant $c_N$ with the following property: For any real numbers $0 \leq d_1,\ldots,d_N \leq q$ and  $-1\leq r_1,\ldots,r_N \leq1$, 
$$
  \left| 1 - \prod_{i=1}^N (1-d_i)^{r_i} \right| \leq c_N \sum_{i=1}^N |r_i| d_i.
$$ 
\end{lemma}

Next we prove an estimate on the change of basis coefficients from Proposition \ref{prop:GTs_comparison2_for_SU3} which will allow us to reduce the $q$-hypergeometric function from $\qhypergeometric{4}{3}$ to $\qhypergeometric{2}{1}$.

\begin{lemma}
\label{lem:legendre_approximation2}
Fix $k\in\NN$.  There is a constant $C(k)$ such that for all $j,m\in\NN$
\begin{equation}
\label{eq:legendre_approximation2}
  \left| (-1)^{j+k+m} \textstyle \frac{[2j+1]}{[j]^\half [j+1]^\half} \yxbraket{k}{j}
   \;-\; q^{m-j+\half} (q-q^{-1}) \legendre_k(q^{2(m-j)}|q^2) \right|
    \;\leq\; C(k) q^{m+j},
\end{equation}
where $\xket{j}$ and $\yket{k}$ are as in Proposition \ref{prop:GTs_comparison2_for_SU3}, and $\legendre_k(\,\cdot\,|q^2)$ is the little $q^2$-Legendre polynomial:
$$
\legendre_k(x|q^2) \defeq \qhypergeometric{2}{1} \left(\left. 
    \begin{array}{c}
      q^{-2k}, q^{2(k+1)} \\
      q^2
    \end{array}
  \right| q^2; q^2 x \right).
$$
\end{lemma}

\begin{proof}
Proposition \ref{prop:GTs_comparison2_for_SU3} says
\begin{multline}
\label{eq:Racah_Pochammer_sum}
 (-1)^{j+k+m}  \frac{[2j+1]}{[j]^\half [j+1]^\half} \yxbraket{k}{j} \\
  = \sum_{l=0}^k \frac{1}{[m+1]} \frac{[2j+1]}{[j]^\half [j+1]^\half}
   \frac{(q^{-2k},q^{2(k+1)},q^{-2j},q^{2(j+1)}; q^2)_l}{(q^{-2m},q^{2(m+2)},q^2,q^2; q^2)_l} q^{2l}.
\end{multline}
On the other hand
\begin{multline}
\label{eq:Legendre_Pochammer_sum}
  q^{m-j+\half} (q-q^{-1}) \legendre_k(q^{2(m-j)}|q^2) \\
  = \sum_{l=0}^k  (q-q^{-1})
   \frac{(q^{-2k},q^{2(k+1)}; q^2)_l}{(q^2,q^2; q^2)_l} q^{(2l+1)(m-j)+(2l+\half)}.
\end{multline}

Denote by $A_l$ and $B_l$ the $l$th summand of \eqref{eq:Racah_Pochammer_sum} and \eqref{eq:Legendre_Pochammer_sum}, respectively.  Note that
$
  |B_l| \leq C_1(k) q^{m-j}
$
for some constant $C_1(k)$.  

If $j<l$, 
then the Pochammer symbol in the numerator of $|A_l|$ is zero, so 
\begin{align}
|A_l-B_l| &= |B_l| \leq C_1(k) q^{m-j} 
 \leq C_1(k) q^{-2k} q^{m+j}. \label{eq:Legendre_estimate3}
\end{align}
If $j\geq l$, we have 
\begin{eqnarray*}
|A_l - B_l| &=&
|B_l| \, \left|1 \;-\; \frac{q^{-(2l+1)(m-j) -\half}}{(q-q^{-1})} \frac{1}{[m+1]} \frac{[2j+1]}{[j]^\half [j+1]^\half} \frac{(q^{-2j},q^{2(j+1)}; q^2)_l}{(q^{-2m},q^{2(m+2)}; q^2)_l}\right| \\
   & = & |B_l| \left| 1 \;-\; \frac{1}{(1-q^{2(m+1)})} \frac{(1-q^{2(2j+1)})}{(1- q^{2j} )^\half  (1-q^{2(j+1)})^\half} \right. \\
   && \hspace{4cm} \left. \times \frac{ \prod_{i=1}^l (1-q^{2j-2i+2}) \prod_{i=1}^l (1-q^{2j+2i}) }{ \prod_{i=1}^l (1-q^{2m-2i+2}		) \prod_{i=1}^l (1-q^{2m+2i+2}) } \right|.
\end{eqnarray*}
In the latter expression, all the exponents of $q$ are positive and bounded below by $2(j-l)$.  The estimate of Lemma \ref{lem:product_estimate} yields
\begin{align}
  |A_l - B_l| &\leq  C_1(k) q^{m-j}  (4l+3)q^{2(j-l)} \nonumber \\
    & \leq C_1(k) q^{m-j}  (4k+3) q^{2(j-k)} \nonumber \\
    & \leq C_1(k) (4k+3) q^{-2k} q^{m+j}.    \label{eq:Legendre_estimate4}
\end{align}

Taken together, the estimates \eqref{eq:Legendre_estimate3} and \eqref{eq:Legendre_estimate4} yield a constant $C_2(k)$ such that $|A_l - B_l| \leq C_2(k) q^{m+j}$ for all $l,j,m$.  The left hand side of \eqref{eq:legendre_approximation2} is then bounded by
$$
  \sum_{l=0}^k \left| A_l - B_l \right| \leq k C_2(k) q^{m+j}.
$$
This yields the claim. \end{proof}

Finally, we describe the coefficients of $\ph(E_1) \bigket{\smallGTsO{m}^{\!\!\low}}$ asymptotically as $m\to\infty$.

\begin{proposition}
\label{prop:phase_estimate}
For any $k\in\NN$,
\begin{multline}
\label{eq:phase_estimate}
\lim_{m\to\infty}    \bigbra{
    \left( \begin{array}{cccccc}
  \multicolumn{2}{c}{m} &
  \multicolumn{2}{c}{0} &
  \multicolumn{2}{c}{-m} \\
&\multicolumn{2}{c}{\!\!k}
&\multicolumn{2}{c}{-k\!\!+\!\!1}  \\
&&\multicolumn{2}{c}{0}
\end{array} \right)^{\!\!\low}}
 \ph(E_1) \bigket{\smallGTsO{m}^{\!\!\low}}  \\
 =
 (-1)^k \frac{[k]}{[2k]^\half} \left( \left[ \frac{2k-1}{2} \right]^{-1} - \left[ \frac{2k+1}{2} \right]^{-1} \right).
\end{multline}
\end{proposition}

\begin{proof}
From Equation \eqref{eq:GTs_comparison1_for_SU3} we have
$$
 \bigket{\smallGTsO{m}^{\!\!\low}} = \sum_{j=0}^m (-1)^{j+m} \frac{[2j+1]}{[m+1]} \xket{j},
$$
where the $\xket{j}$ are as defined in Proposition \ref{prop:GTs_comparison2_for_SU3}.
Now $\xket{j}$ belongs to a $\U_q^\upp(2)$-subrepresentation of highest weight $ (j,-j)$, and it has weight $0$.  By the standard formulae for $\Uq(\slx_2)$-representations, 
$\xket{j}$ is an eigenvector of $|E_1| = (F_1E_1)^\half$ with eigenvalue $[j]^\half [j+1]^\half$.
Thus,
\begin{equation*}
 \ph(E_1) \bigket{\smallGTs{m}{0}^{\!\!\low}} = E_1 \cdot \left(
   \sum_{j=1}^m \frac{(-1)^{j+m}}{[m+1]} \frac{[2j+1]}{[j]^\half [j+1]^\half} \xket{j} \right).
\end{equation*}
Hence the inner product in the statement of the proposition is equal to
\begin{equation}
\label{eq:phase_estimate1}
 \sum_{j=1}^m \frac{(-1)^{j+m}}{[m+1]} \frac{[2j+1]}{[j]^\half [j+1]^\half} 
  \left. \left. \left\langle
    \left(\begin{array}{cccccc}
  \multicolumn{2}{c}{m} &
  \multicolumn{2}{c}{0} &
  \multicolumn{2}{c}{-m} \\
&\multicolumn{2}{c}{\!\!k}
&\multicolumn{2}{c}{-k\!\!+\!\!1}  \\
&&\multicolumn{2}{c}{0}
\end{array} \right)^{\!\!\low} \right| E_1 \right| \;\mathrm{x}_j \right\rangle.
\end{equation}
Now $E_1^*=F_1$ acts on the lower Gelfand-Tsetlin basis by Formula \eqref{eq:GTs-F} for $\Psi(F_1) = F_2$:
\begin{align*}
  E_1^* \bigket { \left( \begin{array}{cccccc}
  \multicolumn{2}{c}{m} &
  \multicolumn{2}{c}{0} &
  \multicolumn{2}{c}{-m} \\
&\multicolumn{2}{c}{\!\!k}
&\multicolumn{2}{c}{-k\!\!+\!\!1}  \\
&&\multicolumn{2}{c}{0}
\end{array} \right)^{\!\!\low}} 
 &= \textstyle \frac{[m-k+1]^\half [m+k+1]^\half [k]}{[2k-1]^\half [2k]^\half} 
   \bigket{  \left( \begin{array}{cccccc}
  \multicolumn{2}{c}{m} &
  \multicolumn{2}{c}{0} &
  \multicolumn{2}{c}{-m} \\
&\multicolumn{2}{c}{\!\!k\!\!-\!\!1}
&\multicolumn{2}{c}{\!\!-k\!\!+\!\!1}  \\
&&\multicolumn{2}{c}{0}
\end{array} \right)^{\!\!\low} } \\
 & \qquad  \textstyle + \frac{[m-k+1]^\half [m+k+1]^\half [k]}{[2k+1]^\half [2k]^\half} 
   \bigket{ \smallGTs{m}{k}^{\!\!\low} } \\
 & =  \textstyle \frac{[m-k+1]^\half [m+k+1]^\half [k]}{[2k]^\half} 
   \left(\yket{k-1} + \yket{k}\right).
\end{align*}
Putting this into \eqref{eq:phase_estimate1} yields
\begin{multline}
\label{eq:phase_estimate2}
 \bigbra{
    \left( \begin{array}{cccccc}
  \multicolumn{2}{c}{m} &
  \multicolumn{2}{c}{0} &
  \multicolumn{2}{c}{-m} \\
&\multicolumn{2}{c}{\!\!k}
&\multicolumn{2}{c}{-k\!\!+\!\!1}  \\
&&\multicolumn{2}{c}{0}
\end{array} \right)^{\!\!\low}}
 \ph(E_1) \bigket{\smallGTs{m}{0}^{\!\!\low}}  \\
  =  \textstyle\frac{[m-k+1]^\half [m+k+1]^\half}{[m+1]} \frac{[k]}{[2k]^\half} 
    \left( \ds \sum_{j=1}^m (-1)^{j+m} \textstyle \frac{[2j+1]}{[j]^\half [j+1]^\half} \left(\yxbraket{k-1}{j} + \yxbraket{k}{j}\right) \right).
\end{multline}

Now we let $m\to\infty$.  Since $\lim_{m\to\infty} \textstyle\frac{[m-k+1]^\half [m+k+1]^\half}{[m+1]}  =1$,
it only remains to estimate the sum in \eqref{eq:phase_estimate2}.  Lemma \ref{lem:legendre_approximation2} gives us the estimate
\begin{equation}
\label{eq:phase_estimate3}
 \sum_{j=1}^m {\textstyle (-1)^{j+m} \frac{[2j+1]}{[j]^\half [j+1]^\half} \yxbraket{k}{j}}
  = \sum_{j=1}^m {\textstyle (-1)^k q^{m-j+\half} (q-q^{-1}) \legendre_k(q^{2(m-j)}|q^2) }
    \; +\; R_m,
\end{equation}
where 
$$
 |R_m| \leq \sum_{j=1}^m C(k)q^{m+j}
 \longrightarrow  0 \qquad \text{as $m\to\infty$}.
$$
Also,
\begin{align*}
\lefteqn{\sum_{j=1}^m (-1)^k q^{m-j+\half} (q-q^{-1}) \legendre_k(q^{2(m-j)}|q^2) } 
 \qquad\qquad \\
 &= (-1)^{k+1} q^{-\half} (1-q^2) \sum_{i=0}^{m-1} q^{i} \legendre_k(q^{2i}|q^2)
\end{align*}
which is a partial sum for the $q$-integral
\begin{align*}
(-1)^{k+1} q^{-\half} \int_0^1 x^{-\half} \legendre_k(x|q^2) \; d_{q^2} x
= (-1)^{k+1} \left[{\textstyle k+\half } \right]_q^{-1}. 
\end{align*}
The calculation of this $q$-integral is explained in Appendix \ref{sec:integral_identity}.
Putting all this into \eqref{eq:phase_estimate2}, we arrive at
\begin{multline*}
\lim_{m\to\infty} \bigbra{
    \left( \begin{array}{cccccc}
  \multicolumn{2}{c}{m} &
  \multicolumn{2}{c}{0} &
  \multicolumn{2}{c}{-m} \\
&\multicolumn{2}{c}{\!\!k}
&\multicolumn{2}{c}{-k\!\!+\!\!1}  \\
&&\multicolumn{2}{c}{0}
\end{array} \right)^{\!\!\low}}
 \ph(E_1) \bigket{\smallGTs{m}{0}^{\!\!\low}} \\
 = (-1)^k \frac{[k]}{[2k]^\half} \left( \left[{\textstyle k-\half }\right]_q^{-1} - \left[ {\textstyle k+\half} \right]_q^{-1} \right).
\end{multline*}
This finishes the proof. 
\end{proof}


\section{Essential orthotypicality}
\label{sec:essorth}

The notion of essential orthotypicality was introduced in \cite{Yuncken:PsiDOs} as a tool for studying harmonic analysis on manifolds with multiple fibrations.  In brief, two closed subgroups of a compact group are essentially orthotypical if their isotypical subspaces are approximately mutually orthogonal. In \cite{Yuncken:orthotypical} it is shown that in a compact Lie group two subgroups are essentially orthotypical if and only if they generate the entire 
group. Since we do not have an analogous characterization in the quantum case, we shall prove essential orthotypicality for quantum subgroups 
of $\SU_q(n)$ by direct calculation.

\subsection{Definitions and basic properties}

\begin{definition}
\label{def:essential_orthotypicality}
Two closed quantum subgroups $K_{1,q}$, $K_{2,q}$ of a compact quantum group ${K_q}$ are {\em essentially orthotypical} if for any $\tau_1\in\Irr(K_{1,q})$, $\tau_2\in\Irr(K_{2,q})$ 
and any $\epsilon>0$ there are only finitely many $\sigma\in\Irr({K_q})$ for which
$$
\sup \{ |\ip{p_{\tau_1}\xi,p_{\tau_2}\eta}| \st  \xi,\eta\in V^\sigma,~ \|\xi\| = \|\eta\| = 1 \} \geq \epsilon.
$$

\end{definition}

\begin{lemma}
\label{prop:essential_orthotypicality_equivalent_defns}
Let $K_{1,q}$, $K_{2,q}$ be closed quantum subgroups of a compact quantum group ${K_q}$. The following conditions are equivalent.
\begin{bnum}
\item[a)] $K_{1,q}$ and $K_{2,q}$ are essentially orthotypical.
\item[b)] For any $\tau_1\in\Irr(K_{1,q})$ and any $\epsilon>0$, there are only finitely many irreducible unitary ${K_q}$-representations $\sigma\in\Irr({K_q})$ for which
$$
\sup \{ |\ip{p_{\tau_1}\xi,p_{\triv_{K_{2,q}}}\eta}| \st  \xi,\eta\in V^\sigma,~ \|\xi\| = \|\eta\| = 1 \} \geq \epsilon,
$$
where $\triv_{K_{2,q}}$ denotes the trivial representation of $K_{2,q}$.
\item[c)] For any finite sets $S_1 \finitesubset \Irr(K_{1,q})$ and $S_2 \finitesubset \Irr(K_{2,q})$, $p_{S_2}p_{S_1} \in \KK_{K_q}(H)$ on any unitary ${K_q}$-representation space $H$.
\end{bnum}
\end{lemma}

\begin{proof}
This is essentially Lemma 5.1 from \cite{Yuncken:PsiDOs}. Here, we will only prove the implication $ b) \Rightarrow c) $. 
The other implications can easily be adapted from the proof in \cite{Yuncken:PsiDOs}.

Let $\tau_1\in\Irr(K_{1,q})$. Fix $\epsilon>0$ and let $S\finitesubset\Irr({K_q})$ be the finite set of ${K_q}$-types which satisfy the condition in $ b) $.  Then $ (1-p_S) p_{\triv_{K_{2,q}}}p_{\tau_1} = p_{\triv_{K_{2,q}}}p_{\tau_1}(1-p_S) $ has norm at most $\epsilon$.  
By Lemma \ref{lem:K_equivalent_defns}, we therefore have $p_{\triv_{K_{2,q}}}p_{\tau_1} \in \KK_{K_q}(H)$.  From this, we obtain $p_{\triv_{K_{2,q}}}p_{S_1} \in \KK_{K_q}(H)$ for any 
finite set $S_1 \finitesubset \Irr(K_{1,q})$.  

Now let $\tau_2\in\Irr(K_{2,q})$ be arbitrary.    
Choose a finite-dimensional ${K_q}$-representation $V$ in which $\tau_2$ occurs as a $K_{2,q}$-type. By Lemma \ref{lem:reduction_to_triv} there are 
linear maps $\iota:\CC\to V^c \otimes V$ and $ \bar{\iota}: V^c \otimes V \to \CC$ so that $p_{\tau_2}$ factorizes as
$$
 \xymatrix{
  p_{\tau_2}:  H \ar[r]^-{\Id_H \otimes \iota} 
   & H \otimes V^c \otimes V \ar[r]^-{p_{\triv_{K_{2,q}}}\otimes\Id_{V}} 
   & H \otimes V^c \otimes V \ar[r]^-{\Id_H\otimes \bar{\iota}} 
   & H.
 }
$$ 
Let $S_1$ be the finite set of all $K_{1,q}$-types occurring in $p_{\tau_1}(H) \otimes V^c$. 
We get the factorization
$$
 \xymatrix{
  p_{\tau_2}p_{\tau_1}:  H \ar[r]^-{p_{\tau_1} \otimes \iota} 
   & H \otimes V^c \otimes V \ar[r]^-{p_{\triv_{K_{2,q}}}p_{S_1}\otimes\Id_{V}} 
   & H \otimes V^c \otimes V \ar[r]^-{\Id_H\otimes \bar{\iota}} 
   & H.
 }
$$ 
But $p_{\triv_{K_{2,q}}}p_{S_1} \in \KK_{K_q}(H\otimes V^c)$ by the preceding paragraph, and since $V$ is finite-dimensional, we obtain $p_{\triv_{K_{2,q}}}p_{S_1} \otimes \Id_{V} \in \KK_{K_q}(H\otimes V^c \otimes V)$.  We also have $p_{\tau_1} \otimes \iota \in \AA_{K_q}(H,H\otimes V^c \otimes V)$ and $\Id_H\otimes \bar{\iota} \in \AA_{K_q}(H\otimes V^c \otimes V, H)$, since they preserve $K_q$-types.
We deduce that $p_{\tau_2}p_{\tau_1} \in \KK_{K_q}(H,H)$.
\end{proof}

\begin{corollary}
\label{cor:orthotypicality_and_intersection_of_Ks}
Let $K_{1,q}$ and $K_{2,q}$ be essentially orthotypical quantum subgroups of ${K_q}$.  Suppose $H,H',H''$ are unitary ${K_q}$-representations and that $H'$ has finite ${K_q}$-multiplicities. Then $ \KK_{K_{2,q}}(H',H'') \KK_{K_{1,q}}(H,H') \subseteq \KK(H,H'') \subseteq \KK_{K_q}(H,H'')$. 
\end{corollary}

\begin{proof}
Suppose that $A \in \KK_{K_{1,q}}(H,H') $ is $K_{1,q}$-harmonically finite and that
$B \in \KK_{K_{2,q}}(H', H'') $ is $K_{2,q}$-harmonically finite. Then there are finite sets $S_1\finitesubset\Irr(K_{1,q})$, $S_2\finitesubset\Irr(K_{2,q})$ such that $A = p_{S_1}A$ and $B=Bp_{S_2}$. By essential orthotypicality and Lemma \ref{lem:compact_operators}
we have $p_{S_2}p_{S_1} \in \KK_{K_q}(H') = \KK(H')$.  Thus $BA = Bp_{S_2}p_{S_1}A$ is compact. The result follows.
\end{proof}

\begin{remark}
\label{rmk:counterexample}
It is important that $H'$ has finite ${K_q}$-multiplicities in the above statement. Corollary \ref{cor:orthotypicality_and_intersection_of_Ks} 
can fail when $H'$ has infinite ${K_q}$-multiplicities.
\end{remark}

\begin{lemma}
\label{lem:essential_orthotypicality_for_supergroups}
Let $K_{1,q}$, $K_{1,q}'$, $K_{2,q}$, $K_{2,q}'$ be closed quantum subgroups of ${K_q}$, with $K_{1,q} \subseteq K_{1,q}'$ and $K_{2,q}\subseteq K_{2,q}'$. If $K_{1,q}$ and $K_{2,q}$ are 
essentially orthotypical then $K_{1,q}'$ and $K_{2,q}'$ are essentially orthotypical.
\end{lemma}

\begin{proof}
For $i=1,2$, let $\tau_i$ be an irreducible $K_{i,q}'$-type, and let $S_i\subseteq\Irr(K_{i,q})$ be the finite collection of $K_{i,q}$-types which occur non-trivially in $\tau_i$. Then on any $ {K_q} $-representation $ H $, we have 
$$
p_{\tau_1}p_{\tau_2} = p_{\tau_1} p_{S_1} p_{S_2} p_{\tau_2}.
$$
The product $p_{S_1} p_{S_2}$ belongs to $\KK_{K_q}(H)$ by Lemma \ref{prop:essential_orthotypicality_equivalent_defns}, and the other projections belong to $\AA_{K_q}(H)$ since they commute 
with ${K_q}$-isotypical projections. Hence the claim follows from Lemma \ref{prop:essential_orthotypicality_equivalent_defns}. 
\end{proof}

\subsection{Essential orthotypicality of subgroups of $\U_q(n)$}
\label{sec:Uqn_orthotypicality}

We now specialize to $ \U_q(n) $. Recall that we defined block-diagonal quantum subgroups $ \U_q^{\rootset}(n) $ for 
any $ \rootset \subseteq \Sigma $ in section \ref{sec:subgroups_of_U_q}.  

\begin{lemma}
\label{lem:essential_orthotypicality_base_case}
The quantum subgroups $\U_q^{\{1,\ldots,n-2\}}(n)$ and $\U_q^{\{2,\ldots,n-1\}}(n)$ are essentially orthotypical in $\U_q(n)$.
\end{lemma}

\begin{proof}
By Lemma \ref{lem:essential_orthotypicality_for_supergroups}, it suffices to prove that the quantum subgroups $\U_q^\upp(n-1)$ 
and $\U_q^{\{2,\ldots,n-1\}}(n)$ are essentially orthotypical.
To show this, we will verify condition $b)$ of Lemma \ref{prop:essential_orthotypicality_equivalent_defns}. 

Fix $\tau_1$ an irreducible  representation of $\U_q^\upp(n-1)$ and $\epsilon>0$.  Let $\sigma$ be an irreducible $\U_q(n)$-representation.  Arguing as in Section \ref{sec:class_1}, we observe that
the trivial $\U_q^{\{2,\ldots,n-1\}}(n)$-type does not occur in $\sigma$ unless $ \sigma $ has highest weight of the form $\mu= (m,0,\ldots,0,-m)$ for some $m\in\NN$, in which case the trivial $\U_q^{\{2,\ldots,n-1\}}(n)$-isotypical subspace is spanned by the lower Gelfand-Tsetlin vector
$$
 \fixedvec = \bigket{\GTsO{m}^{\!\!\!\!\biglow\;\;}}.
$$
By proposition \ref{prop:GTs-comparison1},
\begin{multline}
\label{eq:fixedvec}
   \fixedvec  \\
  = \hspace{-3ex} \sum_{0\leq m_2 \leq \cdots \atop \cdots \leq m_{n-1} \leq m_n=m} \hspace{-3ex} 
      \frac{(-1)^{|\mathbf{m}|} \prod_{k=2}^{n-1} [2m_k+k-1]^\half}
	{ [n-2]!^{\half} \qchoose{m+n-2 \\ n-2} }
      \bigket{\GTs{m}{-m}{m_{n-1}\!\!}{\!\!-m_{n-1}}{m_2}{\!-m_2}{0}^ {\!\!\!\!\bigupp\;\;}}.
\end{multline}
We see that the $\tau_1$-isotypical subspace of $\sigma$ will be orthogonal to $\fixedvec$ unless $\tau_1$ has highest weight of the form $(m_{n-1},0,\ldots,0,-m_{n-1})$ for some $m_{n-1} \in \NN$.  

So, let $ \tau_1 $ be the $ \U_q^\upp(n-1) $-type with highest weight $(m_{n-1},0,\ldots,0,-m_{n-1})$. Regardless of $m$, the sum 
in \eqref{eq:fixedvec} contains at most a fixed finite number of vectors of $ \U_q^\upp(n-1) $-type $ \tau_1 $, since they all 
have to verify $m_2 \leq \cdots \leq m_{n-2} \leq m_{n-1}$. Moreover, the coefficient of each of these terms is bounded in absolute value by
$$
  \frac{  [2m_{n-1}+(n-1)-1]^{\half(n-2)} }{ [n-2]!^{\half} \qchoose{m+n-2 \\ n-2} },
$$
which tends to zero as $m\to\infty$.  It follows that there are only finitely many $m\in\NN$ for which there exists a unit vector $\xi$ of $\Uq(\lie{gl}_{n-1}^\upp)$-type $\tau_1$ such that $|\braket{\xi}{(M_{m,0,\ldots,0})^\low} | > \epsilon$.  This completes the proof.
\end{proof}

\begin{proposition}
\label{thm:essential_orthotypicality_in_SUq}
Let $n\geq 2$ and let $\rootset_1,\rootset_2$ be sets of simple roots of $\U_q(n)$. If $\rootset_1\cup\rootset_2 = \Sigma$ then $\U_q^{\rootset_1}(n)$ and $\U_q^{\rootset_2}(n)$ are 
essentially orthotypical in $\U_q(n)$.
\end{proposition}

\begin{proof}
The result is trivial for $n=2$.  Suppose now that $n>2$ and that the result has been proven for $\U_q(n-1)$.  

We claim first that $\U_q^{J_1}(n)$ and $\U_q^{J_2}(n)$ are essentially orthotypical in $\U_q^{\{1,\ldots,n-2\}}(n)$ whenever $J_1 \cup J_2 = \{1,\ldots,n-2\}$.  Recall that $\U_q^{\{1,\ldots,n-2\}}(n) = \U_q^\upp(n-1) \times Z^\upp_{n-1}$.  Moreover, we have $\U_q^{J}(n) = \U_q^{J}(n-1) \times Z^\upp_{n-1}$ for any $J\subseteq\{1,\ldots,n-2\}$. Note that $\U_q^J(n-1)$ and $Z^\upp_{n-1}$ are commuting and generating quantum subgroups of $\U_q^J(n)$. Using Lemma \ref{lem:product_groups}, the inductive hypothesis implies that the quantum subgroups $\U_q^{J_1}(n), \U_q^{J_2}(n) \subseteq \U_q^{\{1,\ldots,n-2\}}(n)$ satisfy condition $c)$ of Lemma \ref{prop:essential_orthotypicality_equivalent_defns} whenever $J_1 \cup J_2 = \{1,\ldots,n-2\}$.  This proves the claim.

Suppose now that $\rootset_1\cup\rootset_2 = \{1, \ldots, n - 1\} $. Assume without loss of generality that $n-1 \in \rootset_1$.
Let $\tau_i \in \Irr(\U_q^{\rootset_i}(n))$ for $i=1,2$. Moreover, let $S_i$ denote the set of $\U_q^{\rootset_i \setminus \{n-1\}}(n)$-types that occur in $\tau_i$, so that we have $p_{\tau_i} = p_{\tau_i}p_{S_i} $ on any $\U_q(n)$-representation $H$. By the claim above, $p_{S_1}p_{S_2}$ is in $\KK_{\U_q^{\{1,\ldots,n-2\}}(n)}(H)$, so for any $\epsilon>0$ there exists a finite set $F_1 \finitesubset \Irr(\U_q^{\{1, \ldots, n-2\}}(n))$ such that $\|(1-p_{F_1})p_{S_1}p_{S_2}\| < \epsilon$.  Note that $p_{F_1}$ commutes with both $p_{S_1}$ and $p_{S_2}$. We therefore obtain
\begin{equation}
 \label{eq:essential_orthotypicality_estimate1}
\|  p_{\tau_1}p_{\tau_2} - p_{\tau_1}p_{F_1}p_{\tau_2}\|
  = \|  p_{\tau_1}p_{S_1}(1-p_{F_1})p_{S_2}p_{\tau_2}\| < \epsilon.
\end{equation}

Now we repeat this trick, this time removing the first simple root instead of the last.  Let $T_1$ denote the finite collection of $\U_q^{\rootset_1\setminus\{1\}}(n)$-types which occur in $\tau_1$, and let $T_2$ denote the finite collection of $\U_q^{\{2,\ldots, n-2\}}(n)$-types which occur in any of the $\U_q^{\{1,\ldots, n -2 \}}(n) $-types in $F_1$. Since we assumed that $n-1\in\rootset_1$, we have $(\rootset_1\setminus\{1\}) \cup\{2,\ldots, n-2\}= \{2,\ldots, n-1\}$. Another application of the above claim implies that $p_{T_1}p_{T_2} \in \KK_{\U_q^{\{2,\ldots, n-1\}}(n)}(H)$.  Thus, there is a finite subset $F_2\finitesubset\Irr({\U_q^{\{2,\ldots, n-1\}}(n)})$ such that $\|(1-p_{F_2})p_{T_1}p_{T_2}\| < \epsilon$, and we obtain
\begin{equation}
 \label{eq:essential_orthotypicality_estimate2}
\|  p_{\tau_1}p_{F_1}p_{\tau_2} - p_{\tau_1}p_{F_2}p_{F_1}p_{\tau_2}\|
  = \|  p_{\tau_1}p_{T_1}(1-p_{F_2})p_{T_2}p_{F_1}p_{\tau_2}\| < \epsilon.
\end{equation}

Combining Equations \eqref{eq:essential_orthotypicality_estimate1} and \eqref{eq:essential_orthotypicality_estimate2} gives
$$
\|  p_{\tau_1}p_{\tau_2} - p_{\tau_1}p_{F_2}p_{F_1}p_{\tau_2}\| < 2\epsilon.
$$
By Lemma \ref{lem:essential_orthotypicality_base_case}, $p_{F_2}p_{F_1} \in \KK_{\U_q(n)}(H)$, so $p_{\tau_1}p_{F_2}p_{F_1}p_{\tau_2}\in \KK_{\U_q(n)}(H)$.  Since $\epsilon$ was arbitrary, $p_{\tau_1}p_{\tau_2} \in \KK_{\U_q(n)}(H)$.  This completes the proof.
\end{proof}

\begin{lemma}
\label{lem:essential_orthogonality}
For any $\rootset_1,\rootset_2 \subseteq \Sigma$, $\U_q^{\rootset_1}(n)$ and $\U_q^{\rootset_2}(n)$ are essentially orthotypical as quantum subgroups of $\U_q^{\rootset_1\cup\rootset_2}(n)$.
\end{lemma}

\begin{proof}
Write $\rootset = \rootset_1 \cup \rootset_2$.  The quantum group $\U_q^\rootset(n)$ has a block diagonal decomposition, which we 
shall write as $\U_q^\rootset(n) = \prod_k \U_q(n_k)$.  
Let $\Sigma_k \subseteq \Sigma$ be the set simple roots of the block $\U_q(n_k)$, and put $\rootset_{i,k}= \rootset_i \cap \Sigma_k$ for $i=1,2$.  We obtain decompositions $\U_q^{\rootset_i}(n) = \prod_k \U_q^{\rootset_{i,k}}(n_k)$.  For each $k$, $\rootset_{1,k} \cup \rootset_{2,k} = \Sigma_k$, so Proposition \ref{thm:essential_orthotypicality_in_SUq} says that the subgroups $\U_q^{\rootset_{1,k}}(n_k)$ and $\U_q^{\rootset_{2,k}}(n_k)$ are essentially orthotypical in $\U_q(n_k)$. A repeated application of Lemma \ref{lem:product_groups} completes the proof.
\end{proof}

We finish this section with the analogous result for quantum subgroups of $K_q=\SU_q(n)$. 

\begin{proposition}
\label{prop:SUq_orthotypicality}
For any $\rootset_1,\rootset_2 \subseteq \Sigma$, $K_q^{\rootset_1}$ and $K_q^{\rootset_2}$ are essentially orthotypical as quantum subgroups of $K_q^{\rootset_1\cup\rootset_2}$.
\end{proposition}

\begin{proof}
Let $T$ be the diagonal maximal torus of $\SU(n)$ and let $Z$ be the centre of $\U(n)$. Both $T$ and $Z$ can be naturally identified with subgroups of the maximal 
torus of $\U_q(n)$.

Fix $\tau_1\in\Irr(K_q^{\rootset_1})$ and let $\triv_2$ denote the trivial representation of $K_q^{\rootset_2}$.  Suppose $\sigma \in \Irr(K_q^{\rootset_1 \cup \rootset_2})$ contains both of these as subrepresentations.  Then in particular, $T$ acts trivially on the trivial $K_q^{\rootset_2}$-isotypical subspace, and by Schur's Lemma $T\cap Z$ acts trivially on all of $V^\sigma$.  By comparison with the representation theory of the classical groups, we therefore know that $\sigma$ and $\tau_1$ extend uniquely to representations $\tilde\sigma \in \Irr(\U_q^{\rootset_1 \cup \rootset_2}(n))$ and $\tilde\tau_1\in\Irr(\U_q^{\rootset_1}(n))$, respectively, in which $Z$ acts trivially.  Denote by $\tilde\triv_{2}$ the trivial representation of $\U_q^{\rootset_2}(n)$.  By Lemma \ref{lem:essential_orthogonality}, for any $\epsilon>0$, there are only finitely many  $\tilde{\sigma}\in\Irr(\U_q^{\rootset_1 \cup \rootset_2}(n))$ for which
$$
\sup \{ |\ip{p_{\tilde\tau_1}\xi,p_{\tilde\triv_2}\eta}| \st  \xi,\eta\in V^{\tilde\sigma},~ \|\xi\| = \|\eta\| = 1 \} \geq \epsilon.
$$
The result therefore follows from Lemma \ref{prop:essential_orthotypicality_equivalent_defns}.
\end{proof}

\begin{corollary}
\label{cor:projections_in_A}
Let $I\subseteq\Sigma$ and let $H$ be a unitary representation of $K_q$.  For any $\tau\in\Irr(K_q^I)$ the isotypical projection $p_\tau$ belongs to $\AA(H)$.
\end{corollary}

\begin{proof}
Let $I'\subseteq\Sigma$.  For any $\tau'\in\Irr(K_q^{I'})$, Proposition \ref{prop:SUq_orthotypicality} implies that $p_\tau p_{\tau'}$ and $p_{\tau'} p_\tau$ are in $\KK_{I\cup I'}(H) \subseteq \KK_{I'}(H)$. From Lemma \ref{lem:A_equivalent_defns} we deduce that $p_\tau \in \AA_{I'}(H)$.  Since $I'$ was arbitrary $p_\tau\in\AA(H)$.
\end{proof}


\subsection{Application to the lattice of ideals}
\label{sec:lattice_of_ideals_II}

Essential orthotypicality is the crucial property for proving Theorem \ref{thm:lattice_of_ideals}.

\begin{lemma}
\label{lem:K_intesections}
Let $ K_q = \SU_q(n) $ and let $\rootset_1, \rootset_2 \subseteq \Sigma$.  Then
\begin{bnum}
\item[a)] $\KK_{\rootset_1} \cap \KK_{\rootset_2} = \KK_{\rootset_1\cup\rootset_2}$,
\item[b)] $\AA_{\rootset_1} \cap \AA_{\rootset_2} \subseteq \AA_{\rootset_1\cup\rootset_2}$.  
\end{bnum}
\end{lemma}

\begin{proof}
From Lemma \ref{lem:subgroups} we have $\KK_{\rootset_1\cup\rootset_2} \subseteq \KK_{\rootset_1} \cap \KK_{\rootset_2}$. 
For the reverse inclusion, suppose 
$T\in\KK_{\rootset_1}(\sH,\sH') \cap \KK_{\rootset_2} (\sH,\sH')$ for some  fully $K_q$-harmonic spaces  $\sH$ and $\sH'$. Thus for any $\epsilon>0$, there are finite sets $S_i \finitesubset \Irr(K^{\rootset_i}_q)$ such that $\| T - p_{S_i} T p_{S_i} \| <\epsilon$ for  $i=1,2$, and we obtain $\| T - p_{S_1} p_{S_2} T p_{S_2} p_{S_1} \| < 2\epsilon$. By Proposition \ref{prop:SUq_orthotypicality}, there is a finite subset $F \subseteq \Irr(K_q^{I_1\cup I_2})$ such that $\|p_{S_1} p_{S_2} - p_F p_{S_1} p_{S_2} p_F \| < \epsilon / \|T\|$, from which
$$
 \|T - p_{F}p_{S_1}p_{S_2}p_F Tp_F p_{S_2}p_{S_1}p_{F} \| < 4\epsilon.
$$
This proves the first statement. The second claim follows by using the characterization of $\AA_{I}$ as multipliers
of $\KK_I$ in Lemma \ref{lem:A_equivalent_defns}.
\end{proof}

Now we are ready to assemble the above results in order to prove Theorem \ref{thm:lattice_of_ideals}. 
Indeed, parts $a)$ and $c)$ of the theorem now follow as a corollary of Lemma \ref{lem:K_intesections}, and part $d)$ is contained in Lemma \ref{lem:J_compact_operators}. 
To prove part $b)$, note that if $\sigma\in\Irr(K_q^I)$ for some $I\subseteq\Sigma$ then $p_\sigma$ is in $\KK_I(\sH)$ for any fully $K_q$-harmonic space $\sH$, so the result follows from Corollary \ref{cor:projections_in_A}. This completes the proof of Theorem \ref{thm:lattice_of_ideals}.


\section{Longitudinal pseudodifferential operators}
\label{sec:PsiDOs2}

In this section we prove Theorem \ref{thm:PsiDOs}.

\subsection{Multiplication operators}

We shall begin with Theorem \ref{thm:PsiDOs} $a)$, which is a consequence of the next proposition. Let us recall once again that we are 
equipping $L^2(K_q)$, and its weight spaces $L^2(\sE_\mu)$ for $\mu\in\bP$, with the structure of a fully $K_q$-harmonic space coming from 
the {\em right} regular representation.  Thus, if $\tau\in \Irr(K_q^I)$ for some $I\subseteq\Sigma$ and $g = \coeff{\eta^\dual}{\eta}$ is a 
matrix coefficient, then $p_\tau g = \coeff{\eta^\dual}{p_\tau\eta}$.

\begin{proposition}
\label{prop:mult_ops_in_A}
For any $f\in \Poly(K_q)$, the left and right multiplication operators $\Mult{f}$ and $\RMult{f}$ belong to $\AA(L^2(K_q))$.
\end{proposition}

\begin{proof}
Fix $\rootset \subseteq \Sigma$. We may assume that $f=\coeff{\xi^\dual}{\xi}$ is a matrix coefficient of an irreducible $K_q$-representation. 
Moreover we may assume that $\xi$ belongs to a $K^\rootset_q$ subrepresentation, say of type $\sigma$. 

Let $\tau\in\Irr(K^\rootset_q)$. From the formula \eqref{eq:product_of_coeffs} for the product of matrix coefficients, one sees that $\Mult{f}p_\tau = p_S \Mult{f}p_\tau$ where $S$ is the finite set of $K^\rootset_q$-types which occur in $\tau\otimes\sigma$. 
This means $ \Mult{f}p_\tau \in \KK_i(L^2(K_q)) $. 

We therefore obtain $ \Mult{\Poly(K_q)} p_\tau \subseteq \KK_i(L^2(K_q)) $. 
Taking adjoints shows that $ p_\tau \Mult{\Poly(K_q)} \subseteq \KK_i(L^2(K_q)) $. By Lemma \ref{lem:A_equivalent_defns}, 
this implies $\Mult{\Poly(K_q)}\in\AA_\rootset(L^2(K_q)) $. 
Since $\rootset$ was arbitrary this yields the claim for left multiplication operators. The proof for right multiplication operators is similar.
\end{proof}

\subsection{Basic properties of the phase of $E_i$, $F_i$}
\label{sec:tensor_operators}

Before specializing to the section spaces of bundles over the quantum flag manifold, we will first consider the abstract properties of the operators
$$ 
D_i = 
\begin{pmatrix} 0 & F_i \\ E_i & 0 \end{pmatrix} 
\qquad \text{on } H_\mu\oplus H_{\mu+\alpha_i}
$$ 
for any $K_q$-representation $H$ and any $\mu\in\bP$. This operator $D_i$ is essentially self-adjoint with domain the linear span of the $K_q$-isotypical components.  

\begin{lemma}
\label{lem:K_i-resolvent}
With the notation above, $\psi(D_i)\in\KK_i(H_\mu\oplus H_{\mu+\alpha_i})$ for any function $\psi\in C_0(\RR)$, and $\phi(D_i)\in\AA_i(H_\mu\oplus H_{\mu+\alpha_i})$ for any bounded function $\phi\in C_b(\RR)$.
\end{lemma}

\begin{proof}
Recall from Section \ref{sec:lattice_of_subgroups} the subgroup $S_q^i \cong \SU_q(2)$ which is associated to the simple root $\alpha_i$. The standard formulae for irreducible $\Uq(\lie{sl}(2))$-representations 
show that the $S^i_q$-isotypical subspaces of $H_\mu\oplus H_{\mu+\alpha_i}$ are precisely the eigenspaces for $D_i^2$, and that the spectrum of $D_i^2$ is discrete. It follows that if the 
function $\psi $ has compact support, then $\psi(D_i)$ annihilates all but finitely many $S^i_q$-isotypical subspaces. 
On a weight space, the $S^i_q$-isotypical and the $K^i_q$-isotypical decompositions are identical, so $\psi(D_i) \in \KK_i(H_\mu\oplus H_{\mu+\alpha_i})$. By density, $\psi(D_i) \in \KK_i(H_\mu\oplus H_{\mu+\alpha_i}) $ for any $ \psi \in C_0(\RR)$.

If $\phi $ is a bounded function then $\phi(D_i)$ is bounded, and it preserves $K^i_q$-types. This proves the second statement.
\end{proof}

\begin{remark}
\label{rmk:C_0(D2)_in_JJ}
It follows from this proof that if $\psi\in C_c(\RR)$, then $\psi(D_i^2)$ is a finite linear combination of isotypical projections for $K_{i,q}$. By Corollary \ref{cor:projections_in_A}, these isotypical projections are in $\AA(H)$.  We can therefore deduce that $\psi(D_i) \in \JJ_i(H)$ for any even function $\psi\in C_0(\RR)$.  Unfortunately, showing that $\psi(D_i) \in \JJ_i(H)$ for an odd $\psi \in C_0(\RR)$, as required for Theorem \ref{thm:PsiDOs} $c)$,
is more difficult.
\end{remark}

Now let $V$ be a finite dimensional unitary representation of $ K_q $. Then $D_i$ acts on $H\otimes V$, and in particular 
on $(H\otimes V)_\mu\oplus (H\otimes V)_{\mu+\alpha_i}$, as
$$
\hat{\Delta}(D_i) = \begin{pmatrix} 0 & F_i \otimes K_i\\ E_i\otimes K_i & 0 \end{pmatrix}  + \begin{pmatrix} 0 & K_i^{-1} \otimes F_i \\K_i^{-1} \otimes  E_i & 0 \end{pmatrix}.
$$
We will abbreviate this expression as $D_i \otimes K_i + K_i^{-1}\otimes D_i$.  

\begin{lemma}
\label{lem:ph_D_commutator}
Let $H$, $V$ be unitary $ K_q $-representations with $ V $ finite dimensional and let $ \mu \in \bP $. 
As operators on $(H\otimes V)_\mu\oplus (H\otimes V)_{\mu+\alpha_i}$, we have 
$\ph(\hat{\Delta}(D_i)) \equiv \ph(D_i) \otimes \Id_V$ modulo $\KK_i((H\otimes V)_\mu\oplus (H\otimes V)_{\mu+\alpha_i})$. 
\end{lemma}

\begin{proof}
 
Since $K_i$ is strictly positive, we have $\ph(D_i \otimes K_i) = \ph(D_i) \otimes \Id_V$. Let us set $A = \hat{\Delta}(D_i)$, $B = D_i \otimes K_i$. 
Then $A-B = K_i^{-1} \otimes D_i$, which is bounded on $(H\otimes V)_\mu \oplus (H\otimes V)_{\mu+\alpha_i}$ since $V$ is finite dimensional, and so $A-B \in \AA_i((H\otimes V)_\mu \oplus (H\otimes V)_{\mu+\alpha_i})$.

Let $\phi\in C_b(\RR)$ be the function $\phi(x) = x(1+x^2)^{-\half}$. Lemma \ref{lem:K_i-resolvent} implies 
that $\ph(A) \equiv \phi(A)$ modulo $\KK_i((H\otimes V)_\mu \oplus (H\otimes V)_{\mu+\alpha_i})$.  We claim that also $\ph(B)  \equiv \phi(B)$ modulo $\KK_i((H\otimes V)_\mu \oplus (H\otimes V)_{\mu+\alpha_i})$.
To see this, note that $(H\otimes V)_\mu \oplus (H\otimes V)_{\mu+\alpha_i} = \bigoplus_{\nu} (H_{\mu-\nu} \oplus H_{\mu+\alpha_i-\nu})\otimes V_\nu$, where the sum is over all weights of $V$.  This decomposition is invariant for $B$, and on each summand $B$ acts as $q^{\half(\alpha_i,\nu)} D_i\otimes \Id_{V_\nu}$.   Lemma \ref{lem:K_i-resolvent} implies that $\psi(B) \in \KK_i((H\otimes V)_\mu \oplus (H\otimes V)_{\mu+\alpha_i})$ for any $\psi\in C_0(\RR)$, and the claim follows.

Therefore, it suffices to prove that $\phi(A) - \phi(B) \in \KK_i((H\otimes V)_\mu\oplus (H\otimes V)_{\mu+\alpha_i})$.  Now,
$$
\phi(A) - \phi(B) =  (A-B)(1+A^2)^{-\half} + B ((1+A^2)^{-\half} - (1+B^2)^{-\half} ).
$$
The first term $(A-B)(1+A^2)^{-\half} $ is contained in $ \KK_i((H\otimes V)_\mu \oplus (H\otimes V)_{\mu+\alpha_i}) $ by Lemma \ref{lem:K_i-resolvent}. For the second term we use the integral formula
$$
(1+x^2)^{-\half} = \frac{1}{\pi} \int_0^\infty t^{-\half} (1+x^2+t)^{-1} \,dt,
$$
which gives
\begin{align}
\lefteqn{ B ( (1+A^2)^{-\half} - (1+B^2)^{-\half} ) } \qquad  \nonumber \\
&=  \frac{1}{\pi} B \int_0^\infty t^{-\half}  ( (1+A^2+t)^{-1}  - (1+B^2+t)^{-1} )\,dt \nonumber \\
&=  \frac{1}{\pi} B \int_0^\infty t^{-\half}  (1+B^2+t)^{-1} (B^2 - A^2) (1+A^2+t)^{-1} \,dt \nonumber \\
&=  \frac{1}{\pi} B \int_0^\infty t^{-\half}   (1+B^2+t)^{-1} B(B-A) (1+A^2+t)^{-1} \,dt \nonumber \\
&\qquad + \frac{1}{\pi} B \int_0^\infty t^{-\half}   (1+B^2+t)^{-1} (B-A)A (1+A^2+t)^{-1} \,dt.  \label{eq:integralformula}
\end{align}
By Lemma \ref{lem:K_i-resolvent}, we have $B(1+B^2+t)^{-1}B \in \AA_i$ with norm at most $1$, $B-A\in\AA_i$, and $(1+A^2+t)^{-1} \in \KK_i$ with norm at most $(1+t)^{-1}$, so the first integral on the right hand 
side of equation \eqref{eq:integralformula} converges in norm in $\KK_i((H\otimes V)_\mu \oplus (H\otimes V)_{\mu+\alpha_i}) $.   For the second integral, we can write 
\begin{align*}
\lefteqn{ B (1+B^2+t)^{-1} (B-A)A (1+A^2+t)^{-1}  }  \qquad \\
 & = B(1+B^2+t)^{-\half} (1+B^2+t)^{-\half} (B-A) A (1+A^2+t)^{-\half} (1+A^2+t)^{-\half},
\end{align*}
where $B(1+B^2+t)^{-\half}$ and $A (1+A^2+t)^{-\half}$ are in $\AA_i$ with norm at most $1$, $B-A\in\AA_i$, and $(1+B^2+t)^{-\half}$ and $(1+A^2+t)^{-\half}$ are in $\KK_i$ with norm at most $(1+t)^{-\half}$, so we again have norm convergence in $\KK_i((H\otimes V)_\mu \oplus (H\otimes V)_{\mu+\alpha_i}) $. This completes the proof.
\end{proof}

Considering the matrix entries of the operators in Lemma \ref{lem:ph_D_commutator} gives the following result for $\ph(E_i)$ and $\ph(F_i)$.
 
\begin{corollary}
\label{cor:ph_E_commutator}
Let $H$ and $V$ be unitary $K_q$-representations, with $V$ finite dimensional. For any weight $\mu\in\bP$ the operators $(\ph(\hat{\Delta}(E_i)) - \ph(E_i)\otimes\Id_V)p_\mu$ and $p_\mu(\ph(\hat{\Delta}(E_i)) - \ph(E_i)\otimes\Id_V)$ belong to $\KK_i(H\otimes V)$. Likewise, $(\ph(\hat{\Delta}(F_i)) - \ph(F_i)\otimes\Id_V)p_\mu$ and $p_\mu(\ph(\hat{\Delta}(F_i)) - \ph(F_i)\otimes\Id_V)$ belong to $\KK_i(H\otimes V)$.
\end{corollary}

\subsection{The phase of the longitudinal Dirac operators}
\label{sec:order_zero}

In this section we prove Theorem \ref{thm:PsiDOs} $b)$ and $c)$. It is easy to see that $\ph(E_i)$ and $\ph(F_i)$ are multipliers of $\KK_i$, but Theorem \ref{thm:PsiDOs} $b)$ claims a more subtle fact, namely that $\ph(E_i)$ and $\ph(F_i)$ are multipliers of $\KK_j$ for every $j\in\Sigma$. We will prove this fact 
in a series of Lemmas, beginning with the case of $\SU_q(3)$.  

We use the notation for subgroups of $\U_q(3)$ which was introduced in Section \ref{sec:subgroups_of_U_q}.

\begin{lemma}
\label{lem:phase_Ei_p_triv_in_K}
Let $\triv_2$ be the trivial representation of $\U_q^{\{2\}}(3)$.
On any unitary $\U_q(3)$-representation $H$, the operators $\ph(E_1) p_{\triv_2}$ and $\ph(F_1)p_{\triv_2}$ belong to $\KK_{\U_q^{\{2\}}(3)}(H)$.
\end{lemma}

\begin{proof}
We will show that for any $\epsilon > 0$ there is a finite collection $S \finitesubset \Irr(\U_q^{\{2\}}(3))$ such that on every irreducible $\U_q(3)$-representation $V^\sigma$ the estimate 
\begin{equation}
\label{eq:phase_Ei_p_triv_in_K}
  \|(1 - p_S) \ph(E_1) p_{\triv_2} \| < \epsilon
\end{equation}
holds. Since $S$ does not depend on $\sigma$ in this statement, the lemma will follow by decomposing $H$ into irreducibles for $\U_q(3)$.

As before, we write $\sigma^\mu$ for the irreducible $\U_q(3)$-representation with highest weight $\mu$.
It follows from Section \ref{sec:class_1} that the operator $p_{\triv_2}$ is zero on $V^{\sigma^\mu}$ unless $\mu= (m,0,-m)$ for some $m\in\NN$, in 
which case $p_{\triv_2}V^\sigma$ is spanned by the lower Gelfand-Tsetlin vector
$$
\ket{(M_{(m,0,0)})^\low} = \bigket{ \smallGTsO{m}^{\!\!\!\low\;} }.
$$
Note that $ \ket{(M_{(m,0,0)})^\low} $ has weight $ 0 $, so that 
$\ph(E_1) \ket{(M_{(m,0,0)})^\low}$ is contained in the weight space $(V^{\sigma^\mu})_{\alpha_1}$, which is spanned by the vectors 
\begin{equation}
 \label{eq:alpha_1-weight-basis}
  \bigket{
    \left( \begin{array}{cccccc}
  \multicolumn{2}{c}{m} &
  \multicolumn{2}{c}{0} &
  \multicolumn{2}{c}{-m} \\
&\multicolumn{2}{c}{\!\!k}
&\multicolumn{2}{c}{-k\!\!+\!\!1}  \\
&&\multicolumn{2}{c}{0}
\end{array} \right)^{\!\!\!\low\;}}
\end{equation}
for $k = 1,\ldots,m$.  

Let us denote by $\tau_k$ the $\U_q^{\{2\}}(3)$-type of the vector \eqref{eq:alpha_1-weight-basis}, and let
$S_l = \{\tau_1, \ldots, \tau_l\}$.  On $V^{(m,0,-m)}$, the operator $p_{S_l} \ph(E_1) p_{\triv_2}$ satisfies
$$
  \| p_{S_l} \ph(E_1) p_{\triv_2} \|^2 
    = \sum_{k=1}^l \left\langle \left. \left.
    \left( \begin{array}{cccccc}
  \multicolumn{2}{c}{m} &
  \multicolumn{2}{c}{0} &
  \multicolumn{2}{c}{-m} \\
&\multicolumn{2}{c}{\!\!k}
&\multicolumn{2}{c}{-k\!\!+\!\!1}  \\
&&\multicolumn{2}{c}{0}
\end{array} \right)^{\!\!\low\;} \right|
 \ph(E_1)  \right| {(M_{(m,0,0)})^\low} \right\rangle ^2.
$$
Using Proposition \ref{prop:phase_estimate}, we have that
\begin{align*}
\lefteqn{ \lim_{m\to\infty} \| p_{S_l} \ph(E_1) p_{\triv_2} \|^2  } \hspace{2cm} \\
  &= \sum_{k=1}^l \frac{ [k]^2 }{[2k]} \left( \left[ k-\half \right]^{-1} - \left[ k+\half \right]^{-1} \right)^2 \\
  &=  \sum_{k=1}^l \frac{ 1}{[2k] \left[ k-\half \right]^2 \left[ k+\half \right]^2} \left( [k]\left[ k+\half \right] - [k]\left[ k-\half \right] \right)^2 \\
    &=  \sum_{k=1}^l \frac{ 1}{[2k] \left[ k-\half \right]^2 \left[ k+\half \right]^2} \left( \left[\half\right][2k]\right)^2 \\
    &= \sum_{k=1}^l \left[ \half \right]^2 \left( \frac{1}{[k-\half]^2} - \frac{1}{[k+\half]^2} \right) \\
    &= 1 - \frac{[\half]^2}{[l+\half]^2}.
\end{align*}
Let $l$ be sufficiently large that $\frac{[\half]^2}{[l+\half]^2} < \half\epsilon$.  Then
$$
\lim_{m\to\infty} \| p_{S_l} \ph(E_1) p_{\triv_2} \|^2 > 1-\half\epsilon.
$$
This implies that for all $m$ greater than some $m_0$ we have
$$
\| p_{S_l} \ph(E_1) p_{\triv_2} \|^2 > 1-\epsilon
$$
on the representation $V^{\sigma^{(m,0,-m)}}$. Therefore we obtain
$$
\|(1-p_{S_l}) \ph(E_1) p_{\triv_2}\|^2 = \|\ph(E_1) p_{\triv_2}\|^2 - \|p_{S_l} \ph(E_1) p_{\triv_2}\|^2 < \epsilon
$$
for all $ m > m_0 $. 

Let $S \finitesubset \Irr(\U_q^{\{2\}}(3))$ be the finite set containing $S_l$ as well as the finite collection of $\U_q^{\{2\}}(3)$-types which 
appear in any of the representations of highest weight $(m,0,-m)$ for $m=0,\ldots,m_0$.   By construction, we have:
\begin{itemize}
\item $(1-p_S) \ph(E_1) p_{\triv_2}=0$ on every $V^{\sigma^\mu}$ with $\mu$ not of the form $(m,0,-m)$;
\item $(1-p_S) \ph(E_1) p_{\triv_2}=0$ on every  $V^{\sigma^{(m,0,-m)}}$ with $m\leq m_0$;
\item  
$
  \|(1-p_S) \ph(E_1) p_{\triv_2}\| < \sqrt{\epsilon}
$
on every $V^{\sigma^{(m,0,-m)}}$ with $m> m_0$. 
\end{itemize}
We conclude that on any unitary $\U_q(3)$-representation $H$, the operator $\ph(E_1) p_{\triv_2}$ is approximated to 
within $\sqrt{\epsilon}$ by $p_S \ph(E_1) p_{\triv_2}$. This proves that $\ph(E_1) p_{\triv_2} \in \KK_{\U_q^{\{2\}}(3)}(H)$.
The proof that $\ph(F_1) p_{\triv_2}\in\KK_{\U_q^{\{2\}}(3)}(H)$ is similar.
\end{proof}

\begin{corollary}
\label{cor:phase_Ei_p_triv_in_K}
With $K_q = \SU_q(3)$, let $\triv_2$ denote the trivial corepresentation of $K_q^{\{2\}}$.
On any unitary $K_q$-representation $H$, the operators $\ph(E_1) p_{\triv_2}$ and $\ph(F_1)p_{\triv_2}$ belong to $\KK_2(H)$.
\end{corollary}

\begin{proof}
We reuse the notation from the proof of Proposition \ref{prop:SUq_orthotypicality}.  Recall that $T\cap Z$ acts trivially on any irreducible $K_q$-representation which contains the trivial $K_q^{\{2\}}$-type.  Putting $H' = p_{\triv_{T\cap Z}} H$, we have $p_{\triv_2}=0$ on ${H'}^\perp$, so it suffices to prove the result with $H'$ in place of $H$.

The $K_q$-representation on $H'$ extends to a $\U_q(3)$-representation in which $Z$ acts trivially.  With this extension, $\KK_Z(H') = \LL(H')$.  One can check that $Z$ and $K^{\{2\}}_q$ are commuting and generating quantum subgroups of $\U_q^{\{2\}}(3)$, so the result follows from Lemma \ref{lem:phase_Ei_p_triv_in_K} and Lemma \ref{lem:product_groups}.
\end{proof}

\begin{lemma}
\label{lem:SUq3_case}
With $K_q= \SU_q(3)$, let $H$ be a unitary $K_q$-representation. Then we have $\ph(E_i) \in \AA(H)$ and $\ph(F_i)\in \AA(H)$ for $i=1,2$.
\end{lemma}

\begin{proof}
 
Let us first assume that $H$ has finite $K_q$-multiplicities. 

In order to prove $\ph(E_1)\in\AA(H)$, we only need to show $\ph(E_1)\in\AA_2(H)$ since $\ph(E_1)\in\AA_1(H)$ is clear. 
Let $\tau \in \Irr(K_q^{\{2\}})$. Choose a finite dimensional $K_q$-representation $V$ which contains $\tau$ as a $K_{q}^{\{2\}}$-type, and use Lemma \ref{lem:reduction_to_triv} to factorize $\ph(E_1)p_\tau$ on $H$ as
\begin{align}
\ph(E_1) p_\tau  &=  \ph(E_1)  (\Id_H \otimes \bar{\iota}) (p_{\triv_2}\otimes\Id_{V}) 
 (\Id_H \otimes \iota)\nonumber \\
& = (\Id_H \otimes \bar{\iota})(\ph(E_1)\otimes \Id_{V^c} \otimes \Id_{V}) (p_{\triv_2}\otimes\Id_{V}) 
 (\Id_H \otimes \iota).
\label{eq:phE1_tensor_trick}
\end{align}
Write
$$  
(\ph(E_1)\otimes \Id_{V^c}) p_{\triv_2} = \ph(\hat{\Delta}(E_1)) p_{\triv_2} + (\ph(E_1)\otimes \Id_{V^c} - \ph(\hat{\Delta}(E_1))) p_{\triv_2}.
$$
We have $ \ph(\hat{\Delta}(E_1)) p_{\triv_2} \in\KK_2(H\otimes V)$ by Corollary \ref{cor:phase_Ei_p_triv_in_K}.  We also have that 
$(\ph(E_1)\otimes \Id_{V^c} - \ph(\hat{\Delta}(E_1)))p_{\triv_2} \in \KK_2(H\otimes V)\,\KK_1(H\otimes V)$ by Corollary \ref{cor:ph_E_commutator}, and since $H\otimes V$ has finite $K_q$-multiplicities, Lemma \ref{lem:compact_operators} shows that this is 
in $\KK(H\otimes V) \subseteq \KK_2(H \otimes V)$. We conclude that $\ph(E_1) p_\tau \in \KK_2(H)$.  

One can similarly show that $\ph(F_1)p_\tau \in \KK_2(H) $. Moreover, by taking adjoints, we obtain $p_\tau \ph(E_1),~ p_\tau\ph(F_1) \in \KK_2(H)$. 
Using Lemma \ref{lem:A_equivalent_defns}, we conclude that $\ph(E_1)$ and $\ph(F_1)$ are in $\AA_2(H)$.

Suppose now that $H$ does not necessarily have finite $K_q$-multiplicities. We can embed $H$ into the universal $ K_q $-representation 
$ H_0 = L^2(K_q) \otimes \ell^2(\NN) $, where $\ell^2(\NN)$ is equipped with the trivial $K_q$-representation. Since $\ell^2(\NN)$ contains only the trivial $K_q^{\{2\}}$-type we have $\Id_{\ell^2(\NN)} \in \KK_2(\ell^2(\NN))$.  
Now $\ph(E_1)$ acts on $H_0$ as $ \ph(E_1) \otimes \Id_{\ell^2(\NN)}$. This operator belongs to $\AA_2(L^2(K_q)) \otimes \KK_2(\ell^2(\NN))$, 
and hence to $\AA_2(H_0)$ by Lemma \ref{lem:K_tensor_K}. 
It follows that the restriction of $\ph(E_1)$ to $H$ belongs to $\AA_2(H)$. A similar argument shows that $\ph(F_1) \in \AA_2(H)$.  

To show that $\ph(E_2),\ph(F_2) \in \AA_1(H)$ it suffices to use the automorphism $\Psi$ of Equation \eqref{eq:automorphism} to interchange the simple roots. 
This completes the proof. 
\end{proof}

\begin{proposition}
\label{prop:phase_Ei_in_A}
Let $K_q = \SU_q(n)$ and let $H$ be any unitary $K_q$-representation.  For each $i=1,\ldots,n-1$, the operators $ \ph(E_i) : H \to H$ and $\ph(F_i) : H \to H$ belong to $\AA(H)$.
\end{proposition}

\begin{proof}
We need to prove $\ph(E_i) \in \AA_j(H)$ for every $j\in\Sigma$.  Note that $S^j_q$ and $T^{j\perp}$ are commuting and generating quantum subgroups of $K^j_q$.   If $v\in H$ is of $T^{j\perp}$-type $\lambda$, then $\ph(E_i)v$ is of $T^{j\perp}$-type $\lambda + \alpha_i|_{T^{j\perp}}$, where by abuse of notation we are identifying $\alpha_i\in\bP$ with its exponential in $\hat{T}$.  It follows that $\ph(E_i) \in \AA_{T^{j\perp}}(H)$.  By Lemma \ref{lem:product_groups} it remains only to prove $\ph(E_i) \in \AA_{S^j_q}(H)$.

If $i=j$, this is immediate.  If $|i-j|=1$, then $E_i,F_i,E_j,F_j$ belong to a subalgebra of $\Uq(\slx_n)$ isomorphic to $\Uq(\slx_3)$, and the result follows from Lemma \ref{lem:SUq3_case}.  Finally, if $|i-j|>1$, then $\ph(E_i)$ commutes with $\Uq(\lie{s}^j)$ so it preserves $S^j_q$-types and the result follows.  

By taking adjoints, we also obtain $\ph(F_i) \in \AA(H)$.  
\end{proof}

We can now prove parts $ b) $ and $ c) $ of Theorem \ref{thm:PsiDOs}. Consider $D_i = \begin{pmatrix} 0 & F_i \\ E_i & 0\end{pmatrix}$ acting on $L^2(\sE_\mu \oplus \sE_{\mu+\alpha_i})$ for some $\mu\in\bP$. Theorem \ref{thm:PsiDOs} $ b) $ follows directly from Proposition \ref{prop:phase_Ei_in_A}. 
In order to prove part $c)$ let $\phi\in C_b(\RR)$ be a continuous odd function such that $\phi(D_i) = \ph(D_i)$. We know from Remark \ref{rmk:C_0(D2)_in_JJ} that $(1+D_i^2)^{-1} \in \JJ_i(L^2(\sE_\mu \oplus \sE_{\mu+\alpha_i}))$. Since $\phi(D_i) \in \AA(L^2(\sE_\mu \oplus \sE_{\mu+\alpha_i}))$ we also have $\phi(D_i)(1+D_i^2)^{-1} \in \JJ_i(L^2(\sE_\mu \oplus \sE_{\mu+\alpha_i}))$. 
By the Stone-Weierstrass Theorem the functions $x\mapsto (1+x^2)^{-1}$ and  $x\mapsto \phi(x) (1+x^2)^{-1}$ generate a dense subalgebra of $C_0(\RR)$, so Theorem \ref{thm:PsiDOs} $c)$ follows.


\subsection{Commutator of functions with the phase of a longitudinal Dirac operator}
\label{sec:PsiDO_commutators}

In this subsection we prove Theorem \ref{thm:PsiDOs} $d)$. 

For any $\lambda\in\lie{h}^*$ one may define an element $K_\lambda$ in $C(\hat{K}_q)$ by declaring 
that $ K_\lambda $ acts on the weight $\nu $ subspace of any irreducible $K_q$-representation by multiplication 
by $q^{\half(\lambda,\mu)}$. If $\lambda = \alpha_i$ is a simple root, then $K_{\alpha_i}$ is the generator $K_i$ of $\Uq(\slx_n)$. 

The element $K_{2\rho} \in C(\hat{K}_q)$, where $\rho$ is the half-sum of all positive roots, shows up in the Schur orthogonality relations. Specifically, the $L^2$-norms of the matrix coefficients of an 
irreducible unitary representation $\sigma$ of $K_q$ satisfy
\begin{equation}
\label{eq:Schur_orthogonality}
 \| \coeff{\xi^\dual}{\xi} \| = \frac{1}{\dim_q(\sigma)^\half}
 \| K_{2\rho} \cdot \xi^\dual \| \|\xi\|
\end{equation}
for $ \xi \in V^\sigma $ and $ \xi^\dual \in V^{\sigma \dual} $, 
where $ \dim_q $ denotes the quantum dimension.
We remark that the Hilbert space structure on $V^{\sigma\dual}$ is 
induced from the canonical isometric isomorphism of $ V^{\sigma\dual} $ with the conjugate Hilbert space of $ V^\sigma $. 
Moreover $K_{2\rho} \in C(\hat{K}_q)$ acts by the transpose action on $V^{\sigma \dual}$. 

Let us derive an estimate on slightly more general matrix coefficients.

\begin{lemma}
\label{lem:non-irreducible_matrix_coeff}
Fix $\sigma\in\Irr(K_q)$. For any $\tau\in\Irr(K_q)$ and $\zeta\in V^\tau \otimes V^\sigma$, $\zeta^\dual \in V^{\tau\dual} \otimes V^{\sigma\dual}$, we have
$$
  \| \coeff{\zeta^\dual}{\zeta} \| \leq \frac{\dim_q(\sigma)^\half}{\dim_q(\tau)^\half}
   \| K_{2\rho} \cdot \zeta^\dual \| \, \| \zeta \|.
$$
Here all norms are Hilbert space norms. 
\end{lemma}

\begin{proof}
Take an orthogonal decomposition $\tau\otimes\sigma = \bigoplus_j \tau_j$ where the $\tau_j$ are irreducible $K_q$-subrepresentations of $\tau\otimes\sigma$.  Correspondingly, we decompose $\zeta = \sum_j \zeta_j$ and $\zeta^\dual = \sum_j \zeta_j^\dual$ where $\zeta_j \in V^{\tau_j}$, $\zeta_j^\dual \in V^{\tau_j\dual}$. Since $\tau \leq \tau_j\otimes\sigma^c$ we have $\dim_q(\tau) \leq \dim_q(\tau_j) \dim_q(\sigma)$. We obtain
$$ 
 \| \coeff{\zeta^\dual}{\zeta} \|^2
=  \sum_j \frac{1}{\dim_q(\tau'_j)} \|K_{2\rho} \cdot \zeta^\dual_j\|^2 \| \zeta_j \|^2
\leq  \sum_j \frac{\dim_q(\sigma)}{\dim_q(\tau)} \|K_{2\rho} \cdot \zeta^\dual_j\|^2 \| \zeta_j \|^2 .
$$
The result follows.
\end{proof}

Let $\mu,\nu\in\bP$, $i\in\Sigma$ and $f\in C(\sE_\nu)$. We will use the bracket $[\ph(E_i), \Mult{f}]$ to denote the operator $\ph(E_i)\Mult{f} - \Mult{f}\ph(E_i): L^2(\sE_\mu) \to L^2(\sE_{\mu+\nu+\alpha_i})$ of Theorem \ref{thm:PsiDOs} $d)$.  From the other parts of Theorem \ref{thm:PsiDOs}, 
we know that this operator belongs to $\AA(L^2(\sE_\mu), L^2(\sE_{\mu+\nu+\alpha_i}))$, so Theorem \ref{thm:PsiDOs} $d)$ is a consequence of the 
following lemma.

\begin{lemma}
\label{lem:commutator_of_PsiDO_with_mult_op}
Fix $\mu,\nu\in\bP$, $i\in\Sigma$ and $f\in C(\sE_\nu)$. Then $[\ph(E_i),\Mult{f}] \in \KK_i(L^2(\sE_\mu), L^2(\sE_{\mu+\nu+\alpha_i}))$.
Similarly, $[\ph(F_i),\Mult{f}] \in \KK_i(L^2(\sE_{\mu+\alpha_i}), L^2(\sE_{\mu+\nu}))$.
\end{lemma}

\begin{proof}
Again, we write $\sigma^\lambda$ for the irreducible $K_q$-representation with highest weight $\lambda\in\bP^+$. Let $\Wts(\sigma^\lambda) \subset \bP$ 
denote the set of weights occurring in $\sigma^\lambda$.

We will assume that $f = \coeff{\xi^\dual}{\xi}$ is a matrix coefficient of $\sigma^\lambda\in\Irr(K_q)$, and that $\xi$ has weight $\nu$ and $K^i_q$-type $\beta$ for some $\beta\in\Irr(K^i_q)$.  Such $f$ span a dense subspace of $C(\sE_\nu)$. 

Let $\epsilon>0$.  By Corollary \ref{cor:ph_E_commutator} we can find a finite set $S\finitesubset \Irr(K^i_q)$ such that for any unitary $K_q$-representation $H$ we have
\begin{align*}
 \lefteqn{ \| (\ph(\hat{\Delta} E_i) - \ph(E_i)\otimes \Id_{V^{\sigma^\lambda}})  p_{\mu+\nu} (\Id - p_S) \|}
 \hspace{4cm} & \\
   & < \epsilon /
  (| \Wts(\sigma^\lambda) |  \dim_q({\sigma^\lambda})^\half \|K_{2\rho}\cdot\xi^\dual\| \|\xi\| ), \\
 \lefteqn{\| (\Id - p_S) (\ph(\hat{\Delta} E_i) - \ph(E_i)\otimes \Id_{V^{\sigma^\lambda}}) p_{\mu+\nu} \|}
  \hspace*{4cm} & \\
  & < \epsilon /
  (| \Wts(\sigma^\lambda) | \dim_q({\sigma^\lambda})^\half \|K_{2\rho}\cdot\xi^\dual\| \|\xi\| ) ,
\end{align*}
as operators on $H\otimes V^{\sigma^\lambda}$.
Let $S'$ be the set of all $\gamma'\in\Irr(K^i_q)$ such that $\gamma' \otimes\beta$ contains some $K^i_q$-type $\gamma$ belonging to $S$.  This is a finite set, since any such $\gamma'$ is a subrepresentation of $\gamma\otimes\beta^c$ with $\gamma\in S$.

Let $g\in L^2(\sE_\mu)$ be a matrix coefficient $g = \coeff{\eta^\dual}{\eta}$ of an irreducible $K_q$-representation $\sigma^\kappa$.  Using Lemma \ref{lem:non-irreducible_matrix_coeff} we obtain
\begin{align}
 \lefteqn{ \| [\ph(E_i) , \Mult{f}] (\Id - p_{S'}) g \| } \qquad \nonumber \\
  & = \| \coeff{\eta^\dual \otimes \xi^\dual}{
   (\ph(\hat{\Delta} E_i) - \ph(E_i)\otimes\Id)  (((\Id-p_{S'})\eta)\otimes\xi)} \| \nonumber \\
  & = \| \coeff{\eta^\dual \otimes \xi^\dual}{
   (\ph(\hat{\Delta} E_i) - \ph(E_i)\otimes\Id) p_{\mu+\nu} (\Id- p_S) ((\Id-p_{S'})\eta)\otimes\xi} \| \nonumber \\
  & \leq \frac{\dim_q({\sigma^\lambda})^\half}{\dim_q({\sigma^\kappa})^\half} \| K_{2\rho}\cdot\eta^\dual \| \| K_{2\rho}\cdot\xi^\dual \|
   \| (\ph(\hat{\Delta} E_i) - \ph(E_i)\otimes\Id) p_{\mu+\nu} (\Id- p_S) \| \| \eta\| \|\xi\| \nonumber \\
  & < \frac{\epsilon}{| \Wts(\sigma^\lambda) | \, \dim_q({\sigma^\kappa})^\half} \| K_{2\rho}\cdot\eta^\dual \| \| \eta\| \nonumber \\
  & =  \frac{\epsilon}{| \Wts(\sigma^\lambda) | } \| g \|.
 \label{eq:irred_coeff_estimate1}
\end{align}
A similar calculation shows that
\begin{equation}
\label{eq:irred_coeff_estimate2}
\| (\Id - p_{S}) [\ph(E_i) , \Mult{f}]  g \| <  \frac{\epsilon}{| \Wts(\sigma^\lambda) | } \|g\| .
\end{equation}

If $g\in L^2(\sE_\mu)$ is instead a sum of irreducible matrix coefficients, then the inequalities \eqref{eq:irred_coeff_estimate1} and \eqref{eq:irred_coeff_estimate2} do not necessarily hold.  To resolve this, we will take advantage of the fact that the operator $[\ph(E_i) , \Mult{f}] (\Id - p_{S'})$ is band-diagonal with respect to $K_q$-types in the following sense.  The operators $\ph(E_i)$ and $(\Id - p_{S})$ commute with the $K_q$-isotypical projections, while $\Mult{f}$ satisfies $p_{\sigma^{\kappa'}} \Mult{f} p_{\sigma^{\kappa}} = 0$ unless $\sigma^{\kappa'}$ occurs as an irreducible subrepresentation of $\sigma^\kappa \otimes \sigma^\lambda$.  We note that if $\sigma^{\kappa'} \leq \sigma^\kappa\otimes \sigma^\lambda$ then $\kappa' = \kappa + \omega$ for some weight $\omega$ of $\sigma^\lambda$.  Therefore we have a decomposition
\begin{equation}
 \label{eq:band-diagonal}
  [\ph(E_i) , \Mult{f}] (\Id - p_{S'}) = \sum_{\omega \in \Wts(\sigma^\lambda)}
   \left( \sum_{\kappa\in\bP^+} p_{\sigma^{\kappa+\omega}} [\ph(E_i) , \Mult{f}] (\Id - p_{S'}) p_{\sigma^{\kappa}} \right),
\end{equation}
where we take the convention that $p_{\sigma^{\kappa+\omega}} = 0$ if $\kappa+\omega$ is not dominant.  By the calculation \eqref{eq:irred_coeff_estimate1} for irreducible matrix coefficients,  $p_{\sigma^{\kappa+\omega}} [\ph(E_i) , \Mult{f}] (\Id - p_{S'}) p_{\sigma^{\kappa}} $ is norm bounded by $\epsilon / | \Wts(\sigma^\lambda) | $ for each $\kappa, \omega$.  Since the projections $p_{\sigma^\kappa}$ are mutually orthogonal, the sum in parentheses in \eqref{eq:band-diagonal} is bounded by $\epsilon / | \Wts(\sigma^\lambda) |$ for each fixed $\omega$.  We conclude that $\| [\ph(E_i) , \Mult{f}] (\Id - p_{S'}) \| < \epsilon $.  Similarly, one obtains  $\| (\Id - p_{S}) [\ph(E_i) , \Mult{f}] \| < \epsilon $. 
This completes the proof that $[\ph(E_i),\Mult{f}] \in \KK_i(L^2(\sE_\mu), L^2(\sE_{\mu+\nu+\alpha_i}))$ for all $f\in C(\sE_\mu)$. 

Finally, one can obtain $[\ph(F_i),\Mult{f}] \in \KK_i(L^2(\sE_{\mu+\alpha_i}), L^2(\sE_{\mu+\nu}))$ by taking adjoints.
\end{proof}


\section{The action of $\SL_q(n,\CC)$}
\label{sec:SLq}

In this section we recall the definition of the quantum group $\SL_q(n,\CC)$ and its principal series representations, 
and prove some estimates that will be used later. 

\subsection{The complex quantum group $\SL_q(n,\CC)$}
\label{sec:SLqdef}

The quantized complex semisimple Lie group $ G_q = \SL_q(n,\CC)$ is defined as the quantum double of $K_q = \SU_q(n)$. 
More precisely, the $ C^* $-algebra of functions on $ G_q $ is given by 
$$ 
C_0(G_q) = C(K_q) \otimes C_0(\hat{K}_q)  
$$ 
with the comultiplication 
$$
\Delta_{G_q} = (\Id \otimes \sigma \otimes \Id)(\Id \otimes \ad(W) \otimes \Id)(\Delta \otimes \hat{\Delta}), 
$$
where $\ad(W)$ is conjugation with the multiplicative unitary $W \in M(C(K_q) \otimes C^*(K_q))$ of $K_q$ 
and $\sigma$ denotes the flip map. In the special case $n = 2$ this quantum group has been studied in detail by Podle\'s and Woronowicz \cite{PWlorentz}. 

The unitary representations of $ G_q $ are in one-to-one correspondence with unitary Yetter-Drinfeld modules 
for $K_q$, compare \cite{NVpoincare}. Passing to the subspace of $ K_q $-finite vectors, one can study 
such representations algebraically, namely in terms of Yetter-Drinfeld modules over the Hopf algebra $\Poly(K_q)$. 
We recall that a Yetter-Drinfeld module over $ \Poly(K_q) $ is a vector space $ V $
equipped with both a left action and a left coaction of $ \Poly(K_q) $ in the purely algebraic sense, satisfying the compatibility condition 
$$
f_{(1)} \xi_{(-1)} S(f_{(3)}) \otimes f_{(2)} \cdot \xi_{(0)} = (f \cdot \xi)_{(-1)} \otimes (f \cdot \xi)_{(0)}
$$ 
for $ f \in \Poly(K_q) $ and $ \xi \in V $. Here we use the Sweedler notation $ \xi \mapsto \xi_{(-1)} \otimes \xi_{(0)} $ for 
the coaction $  V \rightarrow \Poly(K_q) \otimes V $, and we write $ f \cdot v $ for the action of $ f \in \Poly(K_q) $ on $ v \in V $.

\subsection{Principal series representations of $\SL_q(n,\CC)$}
\label{sec:principal_series}

As mentioned before,  For $\lambda\in\lie{h}^*$ one may define an element $K_\lambda$ of $C(\hat{K_q})$ by declaring 
that $ K_\lambda $ acts on the weight $\nu $ subspace of any irreducible $K_q$-representation by multiplication 
by $q^{\half(\lambda,\mu)}$. We note 
that $K_\lambda = K_{\lambda'} $ if $\lambda \equiv \lambda'$ modulo $2i\hbar^{-1}\bQ$, where $\hbar = \frac{\log(q)}{2\pi}$ and $\bQ$ is the root lattice. 
We write $\lie{h}^*_q =  \lie{h}^*/2i\hbar^{-1}\bQ$ and $ i \lie{t}^*_q = i \lie{t}^*/2i\hbar^{-1}\bQ $. 
Our conventions here are adjusted to the quantized enveloping algebra $\Uq(\slx_n)$; recall that in the notation 
of \cite[\S6.1.2]{KS} the element $K_i$ in $\Univ_q(\slx_n)$ corresponds to $K_i^2$ in $\Uq(\slx_n)$. 

The principal series Yetter-Drinfeld modules are parametrized by pairs $(\mu,\lambda) \in \bP \times \lie{h}^*_q$. We will denote the 
principal series Yetter-Drinfeld module with parameter $(\mu,\lambda)$ by $\Poly(\sE_{\mu, \lambda})$. As a $K_q$-representation it is 
just $\Poly(\sE_\mu)$, and the action of $\Poly(K_q)$ is given by
\begin{equation}
\label{eq:YD_action}
\pi_{\mu,\lambda}(a) g = \, a_{(1)} \,g\, (K^2_{\lambda+\rho} \hit S(a_{(2)}))
\end{equation}
for $a\in \Poly(K_q)$ and $g\in \Poly(\sE_\mu)$.

For $\lambda\in i \lie{t}^*_q $ the representations $ \pi_{\mu, \lambda} $ are $ * $-representations with respect to the standard inner 
product on $ \Poly(\sE_\mu) $, and the resulting unitary representations of $ G_q $ on the completion $ L^2(\sE_\mu) $ of $ \Poly(\sE_\mu) $ are 
called unitary principal series representations. We will write $ L^2(\sE_{\mu,\lambda})=L^2(\sE_\mu)$ if we want to emphasize the corresponding 
Yetter-Drinfeld structure. 

We will only need the unitary principal series representations with parameters $(\mu,0)$, and we shall 
make use of the following properties.  

\begin{theorem}
\label{thm:intertwiners}
\begin{bnum}
\item[a)] The representations $L^2(\sE_{\mu,0})$ are irreducible.
\item[b)] The representations $L^2(\sE_{\mu,0})$ and $L^2(\sE_{\nu,0})$ are equivalent if and only if $\mu$ and $\nu$ belong to the same orbit of the Weyl group action on $\bP$.  In particular, if $\nu = w_i\mu$ where $w_i$ is the reflection associated to a simple root $\alpha_i$, then the representations are intertwined by the operator
\begin{align*}
 \ph(E_i)^n : L^2(\sE_{\mu,0}) \stackrel{\cong}{\longrightarrow} L^2(\sE_{w_i\mu,0}), \qquad &\text{if } w_i\mu - \mu = n\alpha_i \text{ with } n > 0, \\
\ph(F_i)^n : L^2(\sE_{\mu,0}) \stackrel{\cong}{\longrightarrow} L^2(\sE_{w_i\mu,0}), \qquad &\text{if } w_i\mu - \mu = n\alpha_i \text{ with } n < 0.
\end{align*}
\end{bnum}
\end{theorem}

The above facts are at least partially ``known to experts,'' although they do not appear in this form in the literature. 
We refer to \cite{VYprincipal} for a detailed exposition.

\subsection{Almost $\SL_q(n,\CC)$-equivariance of the phases of $E_i$ and $F_i$}
\label{sec:SLq-equivariance}

A straightforward computation shows that the multiplication operators on $L^2(\sE_\mu)$ satisfy 
the following covariance property with respect to principal series representations. 

\begin{lemma}
\label{lem:YD_covariance}
Let $\mu, \nu \in \bP$ and $f\in \Poly(\sE_{\nu})$, so that $\Mult{f}$ defines an operator from $L^2(\sE_\mu)$ to $L^2(\sE_{\mu+\nu})$. 
Then for any $a \in \Poly(K_q)$,
$$
\, \pi_{\mu+\nu, 0}(a) \: \Mult{f} 
     = \Mult{a_{(1)}f S(a_{(2)})} \: \pi_{\mu, 0}(a_{(3)}).
$$
\end{lemma}

The next result will be used in the proof of the equivariance properties of our $K$-homology cycle.

\begin{theorem}
\label{thm:SLq-equivariance}
Let $K_q = \SU_q(n)$ for $n\geq2$. Moreover let $\mu\in\bP$ and $i\in\Sigma$.
\begin{bnum}
\item[a)] For any $a\in\Poly(K_q)$, we have $\pi_{\mu,0}(a) \in \AA(L^2(\sE_\mu))$.
\item[b)] The operators $\ph(E_i): L^2(\sE_{\mu,0}) \to L^2(\sE_{\mu+\alpha_i,0})$ and $\ph(F_i): L^2(\sE_{\mu+\alpha_i,0}) \to L^2(\sE_{\mu,0})$ are $\SL_q(n,\CC)$-equivariant modulo $\JJ_i$, in the sense that for any $a\in\Poly(K_q)$,
\begin{align}
\label{eq:E_i-equivariance}
  \, \pi_{\mu+\alpha_i,0}(a) \ph(E_i)  -  \ph(E_i) \pi_{\mu,0}(a)
    &\in \JJ_i(L^2(\sE_\mu), L^2(\sE_{\mu+\alpha_i})), \\
\label{eq:F_i-equivariance} 
 \, \pi_{\mu,0}(a) \ph(F_i)  -  \ph(F_i) \pi_{\mu+\alpha_i,0}(a)
    &\in \JJ_i(L^2(\sE_{\mu+\alpha_i}), L^2(\sE_{\mu})).
\end{align}
\end{bnum}
\end{theorem}

\begin{proof}
The Yetter-Drinfeld action of $a$ on $L^2(\sE_{\mu,0})$ can be written as
$\pi_{\mu, 0}(a) = \Mult{a_{(1)}}\, \RMult{K_{\rho} \hit S(a_{(2)})}$, 
so $ a) $ follows from Proposition \ref{prop:mult_ops_in_A} 

Let $i\in\Sigma$. We have $w_i\rho = \rho-\alpha_i$, so according to Theorem \ref{thm:intertwiners} the operator $\ph(E_i) : L^2(\sE_{\rho-\alpha_i,0}) \to L^2(\sE_{\rho,0})$ is an intertwiner.  Thus, the differences in \eqref{eq:E_i-equivariance} and \eqref{eq:F_i-equivariance} are zero when $\mu = \rho - \alpha_i$.  For general $\mu$, we use Lemma \ref{lem:bundle_po1} to obtain $f_1,\ldots,f_k \in \Poly(\sE_{\mu+\alpha_i-\rho})$ 
and $ g_1, \ldots, g_k \in \Poly(\sE_{\rho-\alpha_i-\mu}) $ such that $\sum_j g_j f_j = 1$. We can then use Theorem \ref{thm:PsiDOs} and 
Lemma \ref{lem:YD_covariance} to compute
\begin{align*}
\lefteqn{  \pi_{\mu+\alpha_i,0}(a) \ph(E_i)  } \qquad \\
    &= \sum_j \pi_{\mu+\alpha_i,0}(a) \ph(E_i) \Mult{f_j}\Mult{g_j} \\
    &\equiv \sum_j \pi_{\mu+\alpha_i,0}(a)\Mult{f_j} \ph(E_i) \Mult{g_j} \qquad \pmod{\JJ_i(L^2(\sE_\mu), L^2(\sE_{\mu+\alpha_i}))} \\
    &= \sum_j \Mult{a_{(1)}f_j S(a_{(2)})} \pi_{\rho,0}(a_{(3)}) \ph(E_i) \Mult{g_j}, 
\end{align*}
noting that all operators involved belong to $\AA$. 
A similar computation yields
\begin{align*}
\sum_j \ph(E_i)\pi_{\mu,0}(a) \equiv \sum_j \Mult{a_{(1)}f_jS(a_{(2)})}& \ph(E_i) \pi_{\rho-\alpha_i,0}(a_{(3)}) \Mult{g_j} \\
  & \pmod{\JJ_i(L^2(\sE_\mu), L^2(\sE_{\mu+\alpha_i}))}. 
\end{align*}
Thus Equation \eqref{eq:E_i-equivariance} is reduced to the case $\mu = \rho-\alpha_i$ which we have just proved. Equation \eqref{eq:F_i-equivariance} follows by taking adjoints.
\end{proof}


\section{BGG elements in $K$-homology}
\label{sec:BGG}

In \cite{Yuncken:BGG} it was shown how an equivariant Fredholm module can be constructed from the geometric BGG complex for the full flag manifold of $\SU(3)$.  Given the results of the previous sections, that construction can now be applied also to the quantized flag manifold of $\SU_q(3)$. The construction carries over almost word for word, so we shall merely give an outline of the steps involved here.

\subsection{The normalized BGG complex}
\label{sec:normalized_BGG}

The reader can consult \cite{BGG} or \cite{BasEas} for the general combinatorial structure underlying the BGG complex of a complex semisimple Lie group. Since we only need a bounded version of the BGG complex for $\SL_q(3,\CC)$, we will proceed in an {\em ad hoc} manner.

\begin{lemma}
\label{lem:intertwiners}
The following is a commuting diagram of intertwining operators between $\SL_q(3,\CC)$ principal series representations:
\begin{equation}
\label{eq:intertwiners}
 \xymatrix@C=18mm@R=10mm@!0{
  & L^2(\sE_{-\alpha_2,0}) \ar[ddrr]^(0.8){\ph(E_2)^2} 
    && L^2(\sE_{\alpha_1,0}) \ar[dr]^>{\ph(E_2)} 
  \\
  L^2(\sE_{-\rho,0}) \ar[ur]^{\ph(E_1)} \ar[dr]_{\ph(E_2)} 
    &&&& L^2(\sE_{\rho,0}) \\
  & L^2(\sE_{-\alpha_1,0}) \ar[uurr]_(0.8){\ph(E_1)^2} 
    && L^2(\sE_{\alpha_2,0}) \ar[ur]_>{\ph(E_1)} 
 }
\end{equation}
\end{lemma}

\begin{proof}
That these operators are intertwiners results from Theorem \ref{thm:intertwiners}. By Schur's Lemma, the diagram commutes up to a scalar. 
By checking on the minimal $K_q$-type, one can verify that the diagram commutes on the nose.
\end{proof}

To define the normalized BGG complex, we displace all the weights in the above construction by $\rho=\alpha_1+\alpha_2$.

\begin{lemma}
\label{lem:normalized_BGG}
The following diagram commutes modulo $\JJ_1 + \JJ_2$:
\begin{equation*} 
\label{eq:simple_BGG}
\xymatrix@C=18mm@R=10mm@!0{
  & L^2(\sE_{\alpha_1,0}) \ar[ddrr]^(0.8){\ph(E_2)^2} 
    && L^2(\sE_{2\alpha_1+\alpha_2,0}) \ar[dr]^>{\ph(E_2)} 
  \\
  L^2(\sE_{0,0}) \ar[ur]^{\ph(E_1)} \ar[dr]_{\ph(E_2)} 
    &&&& L^2(\sE_{2\rho,0}) \\
  & L^2(\sE_{\alpha_2,0}) \ar[uurr]_(0.8){\ph(E_1)^2} 
    && L^2(\sE_{\alpha_1+2\alpha_2,0}) \ar[ur]_>{\ph(E_1)} 
} 
\end{equation*}
\end{lemma}

\begin{proof}
According to Lemma \ref{lem:bundle_po1} we find $f_1, \dots, f_k \in C(\sE_\rho)$, $g_1,\ldots,g_k \in C(\sE_{-\rho})$ such that $\sum_j f_j g_j = 1$. 
Consider the composition $\ph(E_1)\ph(E_2)^2\ph(E_1):L^2(\sE_{0,0}) \to L^2(\sE_{2\rho,0})$.  
By Theorem \ref{thm:PsiDOs}, 
\begin{align*}
\ph(E_1)\ph(E_2)^2\ph(E_1) &= \sum_i \Mult{f_i}\Mult{g_i} \ph(E_1)\ph(E_2)^2\ph(E_1) \\
 & \equiv \sum_i \Mult{f_i} \ph(E_1)\ph(E_2)^2\ph(E_1)\Mult{g_i} \\
 & \hspace{10ex} \mod{\JJ_1(L^2(\sE_{0}), L^2(\sE_{2\rho}))+\JJ_2(L^2(\sE_{0}), L^2(\sE_{2\rho}))}, 
\end{align*}
where the operators in the last expression are the intertwiners of Lemma \ref{lem:intertwiners}. By a similar calculation, we obtain
\begin{align*}
\ph(E_2)\ph(E_1)^2\ph(E_2) & \equiv \sum_i \Mult{f_i} \ph(E_2)\ph(E_1)^2\ph(E_2)\Mult{g_i} \\
 & \hspace{10ex} \mod{\JJ_1(L^2(\sE_{0}), L^2(\sE_{2\rho}))+\JJ_2(L^2(\sE_{0}), L^2(\sE_{2\rho}))}, 
\end{align*}
and the result then follows from Lemma \ref{lem:intertwiners}.
\end{proof}

\begin{lemma}
 \label{lem:almost_symmetry}
Let $\mu\in\bP$, $i\in\Sigma$ and $n\in\NN$.  Then $\ph(F_i)^n \ph(E_i)^n - \Id \in \JJ_i(L^2(\sE_\mu))$ and 
$\ph(E_i)^n \ph(F_i)^n - \Id \in \JJ_i(L^2(\sE_\mu))$.
\end{lemma}

\begin{proof}
Let $\mu_{\alpha_i}\in \half\NN$ be the restriction of $\mu$ to 
a weight of $S^i_q \cong \SU_q(2)$.
The operator $\ph(E_i)^n:L^2(\sE_\mu) \to L^2(\sE_{\mu+n\alpha_i})$ is a partial isometry, and its kernel is the span of those $S^i_q$-isotypical 
subspaces whose highest weight $l\in\half\NN$ satisfies $l<\mu_{\alpha_i}+n$.  Therefore $\ph(F_i)^n\ph(E_i)^n -\Id$ is a projection onto a finite number of $S^i_q$-types, and hence $K^i_q$-types, 
in $L^2(\sE_\mu)$. A similar statement can be made for $\ph(E_i)^n\ph(F_i)^n -\Id$. The result then follows from Theorem \ref{thm:lattice_of_ideals} $b)$.
\end{proof}

We now augment the diagram of Lemma \ref{lem:normalized_BGG} by adding two more operators:
\begin{equation}
\label{eq:BGG}
\xymatrix@C=18mm@R=10mm@!0{
  & L^2(\sE_{\alpha_1,0}) \ar[ddrr]^(0.8){\ph(E_2)^2} 
    \ar[rr]^{-A_1}
    && L^2(\sE_{2\alpha_1+\alpha_2,0}) \ar[dr]^>{\ph(E_2)} 
  \\
  L^2(\sE_{0,0}) \ar[ur]^{\ph(E_1)} \ar[dr]_{\ph(E_2)} 
    &&&& L^2(\sE_{2\rho,0}) \\
  & L^2(\sE_{\alpha_2,0}) \ar[uurr]_(0.8){\ph(E_1)^2} 
    \ar[rr]_{-A_2}
    && L^2(\sE_{\alpha_1+2\alpha_2,0}) \ar[ur]_>{\ph(E_1)} 
} 
\end{equation}
where
\begin{align*}
A_1 &= \ph(E_1)^2 \ph(E_2) \ph(F_1) \\
A_2 &= \ph(E_2)^2 \ph(E_1) \ph(F_2).
\end{align*}
By Lemma \ref{lem:almost_symmetry}, $A_1^*A_1 \equiv \Id$ modulo $\JJ_1(L^2(\sE_{\alpha_1})) + \JJ_2(L^2(\sE_{\alpha_1}))$ and $A_1A_1^*\equiv \Id$ modulo $\JJ_1(L^2(\sE_{2\alpha_1+\alpha_2})) + \JJ_2(L^2(\sE_{2\alpha_1+\alpha_2}))$. Similar statements hold for $A_2$.  The inclusion of the minus signs before $A_1$ and $A_2$ in the diagram \eqref{eq:BGG} ensures that all squares in the diagram {\em anti}-commute modulo $\JJ_1 + \JJ_2$.  Thus \eqref{eq:BGG} is a complex modulo $\JJ_1 + \JJ_2$.

The combinatorial structure underlying the diagram \eqref{eq:BGG} is the Bruhat graph of the group $G=\SL(3,\CC)$. Rather than detail this in generality, let us simply introduce some convenient notation.

\begin{definition}
\label{def:BGG_graph}
Let $\Gamma$ be the set of arrows in the diagram \eqref{eq:BGG} and $\Gamma^{(0)}$ the set of six vertices.  Denote by $T_\gamma$ the operator corresponding to $\gamma\in\Gamma$ in \eqref{eq:BGG}.  Also, to each arrow $\gamma$ we associate a set of simple roots, denoted $\supp(\gamma)$ and called the {\em support of $\gamma$}, according to the Weyl reflection underlying it as follows:
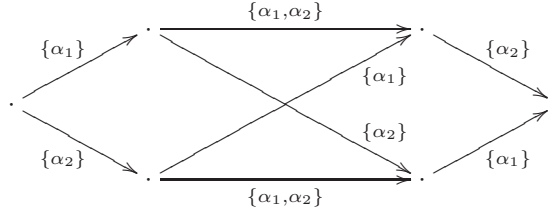
\begin{equation*}
 \xymatrix@C=18mm@R=10mm@!0{
  & \cdot \ar[ddrr]^(0.8){\{\alpha_2\}} 
    \ar[rr]^{\{\alpha_1,\alpha_2\}}
    && \cdot  \ar[dr]^{\{\alpha_2\}} 
  \\
  \cdot  \ar[ur]^{\{\alpha_1\}} \ar[dr]_{\{\alpha_2\}} 
    &&&& \cdot  \\
  & \cdot  \ar[uurr]_(0.8){\{\alpha_1\}} 
    \ar[rr]_{\{\alpha_1,\alpha_2\}}
    && \cdot  \ar[ur]_{\{\alpha_1\}} 
 } 
\end{equation*}
\end{definition}

\subsection{Construction of the Fredholm module}
\label{sec:Fredholm_module}

Let $\sH_\BGG$ be the $\ZZ/2\ZZ$-graded Hilbert space which is the direct sum of the six section spaces in the BGG diagram \eqref{eq:BGG} graded by even and odd Bruhat length, namely $L^2(\sE_0)$, $L^2(\sE_{2\alpha_1+\alpha_2})$ and $L^2(\sE_{\alpha_1+2\alpha_2})$ have degree $0$ and the other three summands have degree $1$.  The sum $\mathbf{T} = \sum_{\gamma} (T_\gamma + T_\gamma^*)$ is an odd $\SU_q(3)$-equivariant operator on $\sH_\BGG$.  It verifies all the axioms of an equivariant Kasparov $K$-homology cycle, but modulo $\JJ_1(\sH_\BGG)+\JJ_2(\sH_\BGG)$ instead of modulo $\KK(\sH_\BGG)$.
To refine this into a genuine Kasparov cycle we use the operator partition of unity constructed in \cite{Yuncken:BGG}, which is described in the 
following lemma.

\begin{lemma}
\label{lem:operator_po1}
Let $K_q=\SU_q(3)$.  There exist mutually commuting operators $N_\gamma \in \LL(\sH_\BGG)$, indexed by the arrows $\gamma\in\Gamma$ above, with the following properties:
\begin{bnum}
\item[a)] For each $\gamma$, $N_\gamma \JJ_i(\sH_\BGG) \subseteq \KK(\sH_\BGG)$ for all $\alpha_i\in\supp(\gamma)$.
\item[b)] For any vertex $v\in \Gamma^{(0)}$, $\sum_{\gamma \ni v} N_\gamma^2 = \Id_{\sH_\BGG}$, where the sum is over all arrows entering or leaving $v$.
\item[c)] Whenever vertices $v,v'\in \Gamma^{(0)}$ are at distance two in the graph we have $N_{\gamma_1}N_{\gamma_2} = N_{\gamma_1'}N_{\gamma_2'}$
where $(\gamma_1,\gamma_2)$ and $(\gamma_1',\gamma_2')$ are the unique two (undirected) paths of length two joining $v, v'$. 
\item[d)] Each $N_\gamma$ is $K_q$-equivariant.
\item[e)] Each $N_\gamma$ commutes modulo compact operators with the left multiplication action of $C(\sX_q)$, the Yetter-Drinfeld action of $\Poly(K_q)$ and 
all of the normalized BGG operators $T_{\gamma'}$.
\end{bnum}
\end{lemma}

\begin{proof}
Using the technical theorem of Kasparov, see \cite{Kasparovconsp}, \cite{Blackadarbook}, \cite{BSKK}, the construction of operators $N_\gamma$ satisfying 
the above properties can be performed as in the proof of \cite[Lemma 4.14]{Yuncken:BGG}. Notice that $ K_q $-invariance is obtained by averaging 
with respect to the Haar functional of $C(K_q)$, applied to the adjoint action of $K_q$ on operators on $L^2(\sH_\BGG)$. 
\end{proof}

\begin{theorem}
\label{thm:K-homology_class}
The operator $\mathbf{F} = \sum_\gamma N_\gamma (T_\gamma+T^*_\gamma)$ is an odd Fredholm operator on $\sH_\BGG$ which defines an 
equivariant $K$-homology class $[\mathbf{F}] \in KK^{\SL_q(3,\CC)}(C(\sX_q),\CC)$.
The $\SU_q(3)$-equivariant index of this class in $KK^{\SU_q(3)}(\CC,\CC) = R(\SU_q(3))$ is the class of the trivial representation.
\end{theorem}

\begin{proof}
The fact that $\mathbf{F}$ defines a $K$-homology class in $KK^{\SL_q(3)}(C(\sX_q),\CC)$ can be proven as for Theorem 4.19 of \cite{Yuncken:BGG}. 
Note that in order to prove $\SL_q(3,\CC)$-equivariance it suffices to check that the action of $ \Poly(\SU_q(3)) $ corresponding to the Yetter-Drinfeld structure commutes with $ \mathbf{F} $ up to compact operators. This in turn follows using Theorem \ref{thm:SLq-equivariance}. 

Finally, the $\SU_q(3)$-types occur in the spaces $L^2(\sE_\mu)$ with the same multiplicities as their classical counterparts, so the index computation follows from the fact that the classical BGG complex is a resolution of the trivial representation.  
\end{proof}

\begin{remark}
\label{rmk:BGG_for_general_weight}
The same general construction can be used to make a Fredholm module $KK^{\SL_q(3,\CC)}(C(\sX_q),\CC)$ with any desired $\SU_q(3)$-equivariant index.  A BGG complex can 
be formed starting with an arbitrary weight $\mu$ (in the notation we have used here it should be an anti-dominant weight), where the weights appearing 
in the equivalent of the diagram \eqref{eq:BGG} are those in the $\rho$-shifted Weyl orbit of $\mu$.  The procedure above applies, and the equivariant 
index of the resulting $KK$-cycle is the class of the irreducible representation with lowest weight $\mu$.
\end{remark}

\section{Applications to Poincar\'e duality and the Baum-Connes conjecture}
\label{sec:bc}

In this section we explain how our previous constructions imply Poincar\'e duality in equivariant $ KK $-theory 
for the flag manifold $ \sX_q = \SU_q(3)/T $, and a certain analogue of the Baum-Connes conjecture for the dual 
of $ \SU_q(3) $. Some of the arguments will only be sketched, and for more information and background we refer to 
\cite{MNtriangulated}, \cite{MNhomalg1}, \cite{Meyerhomalg2}, \cite{NVpoincare}, \cite{Voigtbcfo}. 

Equivariant Poincar\'e duality in $ KK $-theory with respect to quantum group actions was introduced in \cite{NVpoincare}, where it was 
also shown that the standard Podle\'s sphere is equivariantly Poincar\'e dual to itself with respect to the natural action of $\SU_q(2)$. 
An important ingredient in the study of Poincar\'e duality with respect to quantum group actions is the use of braided tensor products, and 
we refer to \cite{NVpoincare} for definitions and more details. 

Our aim here is to exhibit another example of equivariant Poincar\'e duality in the sense of \cite{NVpoincare}, namely for the quantum flag 
manifold $ \sX_q = \SU_q(3)/T $. The key ingredient for this is the class $ [\mathbf{F}] \in KK^{\SL_q(3,\CC)}(C(\sX_q),\CC) $ obtained in 
Theorem \ref{thm:K-homology_class}. It yields a class in $ KK^{\SL_q(3,\CC)}(C(\sX_q) \boxtimes C(\sX_q),\CC) $ by precomposing the representation 
of $ C(\sX_q) $ with the $ * $-homomorphism $ C(\sX_q) \boxtimes C(\sX_q) \rightarrow C(\sX_q) $ induced by multiplication. Here $ \boxtimes $ denotes 
the braided tensor product over $ \SU_q(3) $, and we write $\bD(\SU_q(3)) = \SL_q(3,\CC)$ for the quantum double of $\SU_q(3)$. 

\begin{theorem} \label{PD}
The quantum flag manifold $ \sX_q $ is $ \SU_q(3) $-equivariantly Poincar\'e dual to itself. That is, there is a natural isomorphism
$$
KK^{\bD(\SU_q(3))}_*(C(\sX_q) \boxtimes A, B) \cong KK^{\bD(\SU_q(3))}_*(A, C(\sX_q) \boxtimes B)
$$
for all $ \bD(\SU_q(3)) $-$ C^* $-algebras $ A $ and $ B $.
\end{theorem}

\begin{proof} 
With the class $ [\mathbf{F}] \in KK^{\SL_q(3)}(C(\sX_q) \boxtimes C(\sX_q),\CC) $ at hand, the argument is completely analogous to the proof 
of Theorem 6.5 in \cite{NVpoincare}, reducing it to Poincar\'e duality for the classical flag manifold $ \sX_1 $. We shall 
therefore not go into the details. 

Let us remark that we do not need an explicit description of the element $ \eta_q \in KK^{\bD(\SU_q(3))}_*(\mathbb{C}, C(\sX_q) \boxtimes C(\sX_q)) $ 
corresponding to the unit of the adjunction. 
In fact, this element is uniquely determined from the classical case $ q = 1 $ due to the continuous field structure of $ q $-deformations, 
see \cite{NTpoisson}, \cite{Yamashitaequivariantcomp}. 
\end{proof} 

Let us now come to the Baum-Connes conjecture. We continue to write $ K_q = \SU_q(3) $ and denote by $ \hat{K}_q $ the discrete quantum 
group dual to $ K_q $. 
The starting point of the approach in \cite{MNtriangulated} is to view equivariant Kasparov theory as a triangulated category.  
More precisely, if $ \Gamma $ is a discrete quantum group we consider the category $ KK^\Gamma $ which has as objects all separable
$ \Gamma$-$ C^* $-algebras, and $ KK^\Gamma(A,B) $ as the set of morphisms between two objects $ A $ and $ B $.
Composition of morphisms is given by the Kasparov product. For a description of the structure of $ KK^\Gamma $ as 
a triangulated category we refer to \cite{NVpoincare}, \cite{Voigtbcfo}. Suffice it to say that 
this extra structure allows one to do homological algebra in the context of Kasparov theory. 

In fact, there is one further ingredient needed in the definition of the Baum-Connes assembly map. Namely, 
one has to identify the category $ \sCI_\Gamma $ of compactly induced actions within $ KK^\Gamma $. 
Classically, the objects of $ \sCI_\Gamma $ are the $ C^* $-algebras induced from finite subgroups of the discrete group $ \Gamma $. 
If $ \Gamma $ is torsion-free the situation is particularly simple, in the sense that only the trivial subgroup 
has to be taken into account in this case. 

It turns out that the dual of $ K_q $ behaves like a torsion-free group. More precisely, the quantum group $ \hat{K}_q $ is torsion-free in the sense 
that any ergodic action of $ K_q $ on a finite dimensional $ C^* $-algebra is $ K_q $-equivariantly Morita equivalent 
to the trivial action on $ \mathbb{C} $, see \cite{Meyerhomalg2}, \cite{Goffeng}. 

For a torsion-free quantum group $ \Gamma $ we define the full subcategory $ \sCI_\Gamma $ of $ KK^\Gamma $ by 
\begin{align*}
\sCI_{\Gamma} &= \{C_0(\Gamma) \otimes A| A \in KK \}, 
\end{align*}
where the coaction on $ C_0(\Gamma) \otimes A $ is given by comultiplication on the first tensor factor. Similarly, we let 
$ \sCC_{\Gamma} \subset KK^{\Gamma} $ be the full subcategory of all objects which become isomorphic to $ 0 $ in $ KK $ under the 
obvious forgetful functor. 
The subcategory $ \sCC_{\Gamma} $ is localising, and we denote by $ \langle \sCI_{\Gamma} \rangle $ the localising 
subcategory generated by $ \sCI_{\Gamma} $. Moreover, the pair of localising subcategories 
$ (\langle \sCI_{\Gamma} \rangle, \sCC_{\Gamma}) $ in $ KK^\Gamma $ is complementary, compare \cite{Meyerhomalg2}. That is, $ KK^\Gamma(P,N) = 0 $ for 
all $ P \in \langle \sCI_{\Gamma} \rangle $ and $ N \in \sCC_{\Gamma} $, and every object $ A \in KK^\Gamma $ fits into an exact triangle 
$$
\xymatrix{
\Sigma N \; \ar@{->}[r] & \tilde{A} \ar@{->}[r] & A \ar@{->}[r] & N
}
$$
with $ \tilde{A} \in \langle \sCI_{\Gamma} \rangle $ and $ N \in \sCC_{\Gamma} $. 
Such a triangle is called a Dirac triangle for $ A $, it is uniquely determined up to isomorphism in $ KK^\Gamma $ and depends functorially on $ A $. 

\begin{definition}
Let $ \Gamma $ be a torsion-free discrete quantum group and let $ A $ be a $ \Gamma $-$ C^* $-algebra. 
The Baum-Connes assembly map for $ \Gamma $ with coefficients in $ A $ is the map 
$$ 
\mu_A: K_*(\Gamma \ltimes_\red \tilde{A}) \rightarrow K_*(\Gamma \ltimes_\red A) 
$$
induced from a Dirac triangle for $ A $. If $ \mu_A $ is an isomorphism we shall say that $ \Gamma $ satisfies the Baum-Connes conjecture 
with coefficients in $ A $. 
\end{definition}

By the work of Meyer and Nest \cite{MNtriangulated}, this terminology is consistent with the usual definitions 
in the case that $ \Gamma $ is a torsion-free discrete group. 

Using the Fredholm module for the quantum flag manifold $ \SU_q(3)/T $ in Theorem \ref{thm:K-homology_class} we obtain the following result. 

\begin{theorem} \label{thm:BC}
The dual of $ \SU_q(3) $ for $ q \in (0,1] $ satisfies the Baum-Connes conjecture with trivial coefficients $ \mathbb{C} $. 
\end{theorem} 

\begin{proof} 
We shall follow the arguments in \cite{MNcompact}and show $ \mathbb{C} \in \langle \sCI_{\hat{K}_q} \rangle $. This clearly implies 
that $ \mu_\mathbb{C} $ is an isomorphism. 
Using Baaj-Skandalis duality, it is enough to prove $ C(K_q) \in \langle \mathcal{T}_{K_q} \rangle $, where $ \mathcal{T}_{K_q} \subset KK^{K_q} $ 
denotes the category of all trivial $ K_q $-$ C^* $-algebras. 

We have $ C(T) \subset \langle \mathcal{T}_T \rangle $ by the Baum-Connes conjecture for the abelian group $ \hat{T} $, where 
$ \mathcal{T}_T \subset KK^T $ is the category of trivial $ T $-$ C^* $-algebras. 
This implies $ C(K_q) = \ind_T^{K_q}(C(T)) \in \langle C(K_q/T) \rangle $. Hence it suffices to 
show $ C(K_q/T) \in \langle \mathbb{C} \rangle $. 

In the case $ q = 1 $ one obtains inverse isomorphisms $ \alpha_1: C(K_1/T) \rightarrow \mathbb{C}^{|W|} $ 
and $ \beta_1: \mathbb{C}^{|W|} \rightarrow C(K_1/T) $ in $ KK^{K_1} $ using Poincar\'e duality, where $ |W| = 6 $ is the order of the Weyl 
group of $ K_1 = SU(3) $, see \cite{RS}, \cite{MNcompact}. For general $ q $ we could argue in a similar way by invoking Theorem \ref{PD}. 
Alternatively we may proceed as follows, avoiding the use of braided tensor products. 

The element $ \beta_1 $ is given by induced vector bundles over the flag manifold, and one obtains a corresponding 
class $\beta_q \in KK^{K_q}(\mathbb{C}^{|W|}, C(K_q/T))$ for any $ q \in (0,1] $ using the induction 
isomorphism $KK^{K_q}(\mathbb{C}, C(K_q/T)) \cong KK^T(\mathbb{C}, \mathbb{C})$. 
Similarly, the element $\alpha_1$ is given by twisted Dolbeault operators. Using theorem \ref{thm:K-homology_class} we 
obtain a corresponding class $\alpha_q$ in $KK^{K_q}(C(K_q/T), \mathbb{C}^{|W|})$. 
From $ K_q $-equivariance it is immediate that we have $\beta_q \circ \alpha_q = \Id$ in $ KK^{K_q}(\mathbb{C}^{|W|}, \mathbb{C}^{|W|}) $
for the classes thus obtained. To check $ \alpha_q \circ \beta_q = \Id$ in $ KK^{K_q}(C(K_q/T), C(K_q/T)) $ 
we may use the canonical isomorphism $ KK^{K_q}(C(K_q/T), C(K_q/T)) \cong KK^T(C(K_q/T), \mathbb{C}) $ 
and the fact that the $ C^* $-algebras $ C(K_q/T) $ form a $ T $-equivariant continuous field, implementing a $ KK^T $-equivalence between $ C(K_q/T) $ 
and $ C(K/T) $, see \cite{NTpoisson}, \cite{Yamashitaequivariantcomp}. 
It therefore suffices to consider the effect of $ \alpha_q \circ \beta_q $ on $ K^T_*(C(K_q/T)) \cong R(K) \otimes_{R(T)} R(K) $, which is the same 
for all $ q \in (0,1] $. 
\end{proof} 

We remark that Theorem \ref{thm:BC} is of rather theoretical value. In particular, it does not lead 
to $ K $-theory computations similar to the ones for free orthogonal quantum groups in \cite{Voigtbcfo}.


\appendix

\section{Some results from $q$-calculus}

\subsection{Proof of the change of basis formula of Proposition \ref{prop:GTs_comparison2_for_SU3}}
\label{sec:change_of_basis_proof}

The vectors
\begin{eqnarray*}
  \ket{\xi_j} \defeq \bigket{\smallGTs{m}{j}^\upp},  \qquad
    \ket{\eta_k} \defeq \bigket{\smallGTs{m}{k}^\low},
\end{eqnarray*}
are the basis vectors for the $0$-weight space of $\mu=(m,0,-m)$ in the upper and lower Gelfand-Tsetlin bases, respectively.  
Our calculation of the change-of-basis coefficients avoids the use of raising and lowering operators from \cite{MSK}, instead using a recurrence relation which arises by considering the bracket
\begin{equation}
\label{eq:EF-bracket}
  \bra{\eta_k} E_1^* E_1 \ket{\xi_j}.
\end{equation}

Letting $E_1^* E_1$ act on $\ket{\xi_j}$ first, the Gelfand-Tsetlin formulae \eqref{eq:GTs-E}, \eqref{eq:GTs-F} give
\begin{equation}
 \label{eq:EF-bracket1}
 \bra{\eta_k} E_1^* E_1 \ket{\xi_j} = [j][j+1] \braket{\eta_k}{\xi_j}.
\end{equation}
On the other hand, in the lower Gelfand-Tsetlin basis $E_1$ acts according to the formula \eqref{eq:GTs-E} for $\Psi(E_1)=E_2$ (see the definition of the lower basis in Section \ref{sec:alternative_GTs_bases}).  We get
\begin{eqnarray*}
  E_1^* E_1 \ket{ \eta_k } &=&   
    \frac{[m+k+2] [m-k] [k+1]^2}{[2k+1]^\half [2k+2] [2k+3]^\half} \ket{\eta_{k+1}} \\
    && + \left(
      \frac{[m+k+2] [m-k] [k+1]^2}{[2k+1] [2k+2]} + \frac{[m+k+1] [m-k+1] [k]^2}{[2k] [2k+1]}
      \right) \ket{\eta_{k}} \\
    && +\frac{[m+k+1] [m-k+1] [k]^2}{[2k-1]^\half [2k] [2k+1]^\half} \ket{\eta_{k}}
\end{eqnarray*}
Taking the inner product of this with $\ket{\xi_j}$ and equating with \eqref{eq:EF-bracket1} yields a three-term recurrence relation for $\braket{\eta_k}{\xi_j}$.  The result is simplified if we introduce the non-unit vectors
\begin{equation}
 \label{eq:non-unit_GTS_basis}
 \xket{j} \defeq [2j+1]^{-\half} \ket{\xi_j}, \qquad 
   \yket{k} \defeq [2k+1]^{-\half} \ket{\eta_k}.
\end{equation}
One obtains
\begin{equation}
 \label{eq:recurrence_relation2}
 [j][j+1] \yxbraket{k}{j} =
 a(k) \yxbraket{k+1}{j} + (a(k)+c(k)) \yxbraket{k}{j} + c(k) \yxbraket{k-1}{j},
\end{equation}
where
\begin{eqnarray*}
 a(k) &=& \frac{[m+k+2][m-k][k+1]^2}{[2k+1][2k+2]}, \\
 c(k) &=& \frac{[m+k+1][m-k+1][k]^2}{[2k][2k+1]}.
\end{eqnarray*}

We claim that the solution of Equation \eqref{eq:recurrence_relation2} is given by $q$-Racah coefficients.  Unfortunately, the $q$-Racah polynomials are typically written in terms of the non-symmetric $q$-numbers $[[n]]\defeq\frac{1-q^n}{1-q}$, so we must rewrite the recurrence relation as
\begin{multline}
 \label{eq:recurrence_relation2a}
 (1-q^{2j})(1-q^{2(j+1)}) \yxbraket{k}{j} \\ =
 A(k) \yxbraket{k+1}{j} + (A(k)+C(k)) \yxbraket{k}{j} + A(k) \yxbraket{k-1}{j},
\end{multline}
where
\begin{eqnarray*}
  A(k) &=& \frac{(1-q^{2(m+k+2)})(1-q^{2(k-m)})(1-q^{2(k+1)})^2}{(1-q^{2(2k+1)})(1-q^{2(2k+2)})}, \\
  C(k) &=& \frac{q^2(1-q^{2(k+m+1)})(1-q^{2(k-m-1)})(1-q^{2(k)})^2}{(1-q^{2(2k)})(1-q^{2(2k+1)})}.
\end{eqnarray*}
These are precisely the coefficients in the recurrence relation for the $q^2$-Racah polynomials described in Equation (14.2.3) of \cite{KLS}, with parameters $\alpha=q^{2(m+1)}$, $\beta=q^{-2(m+1)}$, $\gamma=\delta=1$ and $N=m$.  The initial condition for the recurrence relation is fixed by Equation \eqref{eq:GTs_comparison1_for_SU3}, which gives
$$
  \yxbraket{0}{j} = \frac{(-1)^{j+m}}{[m+1]},
$$
and formula of Proposition \ref{prop:GTs_comparison2_for_SU3} follows.

\subsection{A $q$-integral identity for little $q$-Legendre polynomials}
\label{sec:integral_identity}

We recall the definitions of the standard $q$-differentiation and $q$-integration operators:
\begin{eqnarray*}
  \Dq f(x) &\defeq& \frac{f(qx) - f(x)}{(qx-x)}, \\
 \int_0^x f(y) \dq y &\defeq& x (1-q) \sum_{j=0}^\infty q^j f(q^jx) .
\end{eqnarray*}
We also recall the following basic $q$-derivatives, where $[[n]]\defeq\frac{1-q^n}{1-q}$:
\begin{eqnarray}
 \label{eq:monomial_derivative}
 \Dq x^\alpha &=& [[\alpha]]_q x^{\alpha-1}, \qquad \text{ for all } \alpha\in\RR, \\
 \label{eq:Pochhammer_derivative}
 \Dq (x;q^{-1})_n &=& -[[n]]_{q^{-1}} (x;q^{-1})_{n-1} \\
   &=&  -q^{-(n-1)}[[n]]_{q} (x;q^{-1})_{n-1}, \qquad \text{ for all } n\in\NN. \nonumber
\end{eqnarray}

\begin{proposition}
 \label{prop:integral_identity}
 $$ \int_0^1 x^{-\half} \legendre_k(x | q^2) \, d_{q^2} x = q^\half \left[ \frac{2k+1}{2} \right]_q^{-1}.$$
\end{proposition}

\begin{proof}
 Let us put $r\defeq q^2$. The little $q$-Legendre polynomials satisfy the following Rodrigues-type Formula (see \cite{KLS}):
$$
  \legendre_k(x|r) 
  	= \frac{1}{[[k]]_r!} \, D_r^k \left[x^k (x;r^{-1})_k \right].
$$
From Equations \eqref{eq:monomial_derivative} and \eqref{eq:Pochhammer_derivative} one has that for all $0\leq i < k$, $\Dq^i x^k = 0$ at $x=0$ and $\Dq^i (x;r^{-1})_k = 0$ at $x=1$.  Thus, by $k$ applications of $q$-integration by parts, we get
\begin{align*}
  &\int_0^1  x^{-\half} \legendre_k(x | r) \, d_r x \\
    &=  \frac{(-1)^k}{[[k]]_r!} 
     {\textstyle \bigqqnumber[r]{-\frac12} r^{\frac12}  \bigqqnumber[r]{-\frac32}  r^{\frac32} \cdots \bigqqnumber[r]{-\frac{(2k-1)}2} r^{\frac{(2k-1)}2} }
     \int_0^1 x^{-\frac{(2k+1)}2} x^k (x;r^{-1})_k d_rx \\
    &= \frac{1}{[[k]]_r!}  {\textstyle \bigqqnumber[r]{\frac12}  \bigqqnumber[r]{\frac32} \cdots \bigqqnumber[r]{\frac{(2k-1)}2}  } \int_0^1 x^{-\frac12}(x;r^{-1})_k d_rx \\    
\end{align*}
where in the last equality we have used $[[\alpha]]_r = -r^{\alpha}[[-\alpha]]_r$.  The last $q$-integral can be computed by $q$-integrating by parts $k$ more times, giving
\begin{align*}
    \frac{(-1)^k}{[[k]]_r!} r^{\frac12} r^{\frac32} \cdots r^{\frac{2k-1}2}&
      r^k [[-k]]_r [[-k+1]]_r \cdots [[-1]]_r \int_0^1 x^{\frac{(2k-1)}2} d_rx \\
    &=  r^{\frac{k}2} \int_0^1 x^{\frac{(2k-1)}2} d_rx \\
    &=  q^\half\left[ \frac{2k+1}{2} \right]_q^{-1}.
\end{align*}
\end{proof}


\nocite{VV,Voigtbcfo}

\bibliographystyle{alpha}

\bibliography{SU_q}

\end{document}